\documentclass[letter, reqno, 12pt]{amsart}
\AtBeginDocument{\pagenumbering{arabic}}
\usepackage[usenames,dvipsnames]{color}
\usepackage{amsthm,amsfonts,amssymb,amsmath,amsxtra}
\usepackage{tikz}
\usepackage{verbatim}
\usepackage{amsmath}
\usepackage{amsxtra}
\usepackage{amscd}
\usepackage{amsthm}
\usepackage{amsfonts}
\usepackage{amssymb}
\usepackage{eucal}
\usepackage{stmaryrd}
\usepackage{romannum}
\usetikzlibrary{arrows}
\usepackage[all]{xy}
\SelectTips{cm}{}
\usepackage{xr-hyper}
\usepackage[colorlinks=
   citecolor=Black,
   linkcolor=Red,
  urlcolor=Blue]{hyperref}
\usepackage{verbatim}
\usepackage[T1]{fontenc}
\usepackage{leftidx}

\usepackage[margin=1.25in]{geometry}
\usepackage{mathrsfs}

\makeatletter
\newcommand{\extraNode}[6]%
{%
\dynkinPlaceRootRelativeTo{#1}{#2}{#3}{#4}{#5}
\dynkinIndefiniteSingleEdge{#1}{#2}
\dynkinRootMark{o}{#1}
\advance\dynkin@nodes by 1
\dynkinLabelRoot{#1}{#6} 
}%
\makeatother

\RequirePackage{xspace}
\RequirePackage{etoolbox}
\RequirePackage{varwidth}
\RequirePackage{enumitem}
\RequirePackage{tensor}
\RequirePackage{mathtools}
\RequirePackage{longtable}
\RequirePackage{multirow}

\def\<{\langle}
\def\>{\rangle}

\newcommand{\RNum}[1]{\uppercase\expandafter{\romannumeral #1\relax}}

\newcommand{\fkg}{\ensuremath{\mathfrak{g}}\xspace}

\newcommand{\fkn}{\ensuremath{\mathfrak{n}}\xspace}

\newcommand{\fkt}{\ensuremath{\mathfrak{t}}\xspace}

\newcommand{\fkB}{\ensuremath{\mathfrak{B}}\xspace}

\newcommand{\fkH}{\ensuremath{\mathfrak{H}}\xspace}

\newcommand{\heart}{{\heartsuit}}

\newcommand{\BA}{\ensuremath{\mathbb {A}}\xspace}

\newcommand{\BC}{\ensuremath{\mathbb {C}}\xspace}

\newcommand{\BF}{\ensuremath{\mathbb {F}}\xspace}
\newcommand{{\BG}}{\ensuremath{\mathbb {G}}\xspace}

\newcommand{{\BK}}{\ensuremath{\mathbb {K}}\xspace}

\newcommand{\BN}{\ensuremath{\mathbb {N}}\xspace}

\newcommand{\BQ}{\ensuremath{\mathbb {Q}}\xspace}
\newcommand{\BR}{\ensuremath{\mathbb {R}}\xspace}

\newcommand{\BZ}{\ensuremath{\mathbb {Z}}\xspace}
\newcommand{\Bk}{\ensuremath{\mathbf {k}}\xspace}

\newcommand{\CB}{\ensuremath{\mathcal {B}}\xspace}
\newcommand{\CC}{\ensuremath{\mathcal {C}}\xspace}

\newcommand{\CO}{\ensuremath{\mathcal {O}}\xspace}
\newcommand{\CP}{\ensuremath{\mathcal {P}}\xspace}

\newcommand{\CV}{\ensuremath{\mathcal {V}}\xspace}

\newcommand{\Ad}{{\mathrm{Ad}}}

\newcommand{\codim}{{\mathrm{codim}}}

\newcommand{\good}{{\mathrm{good}}}

\newcommand{\Str}{\mathrm{Str}}

\newcommand{\character}{\mathrm{char}}

\newcommand{\Gr}{\mathrm{Gr}}
\newcommand{\Fl}{\mathrm{Fl}}

\newcommand{\val}{\mathrm{val}}
\newcommand{\Inv}{\mathrm{Inv}}

\newcommand{\GL}{\mathrm{GL}}

\newcommand{\id}{\ensuremath{\mathrm{id}}\xspace}

\DeclareMathOperator{\Lie}{Lie}

\newcommand{\reg}{{\mathrm{reg}}}

\newcommand{\SL}{{\mathrm{SL}}}

\newcommand{\SO}{{\mathrm{SO}}}
\newcommand{\Sp}{{\mathrm{Sp}}}

\newtheorem{theorem}{Theorem}
\newtheorem{proposition}[theorem]{Proposition}
\newtheorem{lemma}[theorem]{Lemma}

\newtheorem{corollary}[theorem]{Corollary}

\theoremstyle{definition}
\newtheorem*{acknowledgement}{Acknowledgement}
\newtheorem{definition}[theorem]{Definition}
\newtheorem{example}[theorem]{Example}

\newtheorem{remark}[theorem]{Remark}

\numberwithin{equation}{section}
\numberwithin{theorem}{section}

\theoremstyle{plain}

\setitemize[0]{leftmargin=*,itemsep=\the\smallskipamount}
\setenumerate[0]{leftmargin=*,itemsep=\the\smallskipamount}

\renewcommand{\to}{%
   \ifbool{@display}{\longrightarrow}{\rightarrow}%
   }
\let\shortmapsto\mapsto
\renewcommand{\mapsto}{%
   \ifbool{@display}{\longmapsto}{\shortmapsto}%
   }
\newlength{\olen}
\newlength{\ulen}
\newlength{\xlen}
\newcommand{\xra}[2][]{%
   \ifbool{@display}%
      {\settowidth{\olen}{$\overset{#2}{\longrightarrow}$}%
       \settowidth{\ulen}{$\underset{#1}{\longrightarrow}$}%
       \settowidth{\xlen}{$\xrightarrow[#1]{#2}$}%
       \ifdimgreater{\olen}{\xlen}%
          {\underset{#1}{\overset{#2}{\longrightarrow}}}%
          {\ifdimgreater{\ulen}{\xlen}%
             {\underset{#1}{\overset{#2}{\longrightarrow}}}
             {\xrightarrow[#1]{#2}}}}%
      {\xrightarrow[#1]{#2}}
   }
\makeatother
\newcommand{\xyra}[2][]{%
   \settowidth{\xlen}{$\xrightarrow[#1]{#2}$}%
   \ifbool{@display}%
      {\settowidth{\olen}{$\overset{#2}{\longrightarrow}$}%
       \settowidth{\ulen}{$\underset{#1}{\longrightarrow}$}%
       \ifdimgreater{\olen}{\xlen}%
          {\mathrel{\xymatrix@M=.12ex@C=3.2ex{\ar[r]^-{#2}_-{#1} &}}}%
          {\ifdimgreater{\ulen}{\xlen}%
             {\mathrel{\xymatrix@M=.12ex@C=3.2ex{\ar[r]^-{#2}_-{#1} &}}}
             {\mathrel{\xymatrix@M=.12ex@C=\the\xlen{\ar[r]^-{#2}_-{#1} &}}}}}%
      {\mathrel{\xymatrix@M=.12ex@C=\the\xlen{\ar[r]^-{#2}_-{#1} &}}}%
   }
\makeatletter
\newcommand{\xla}[2][]{%
   \ifbool{@display}%
      {\settowidth{\olen}{$\overset{#2}{\longleftarrow}$}%
       \settowidth{\ulen}{$\underset{#1}{\longleftarrow}$}%
       \settowidth{\xlen}{$\xleftarrow[#1]{#2}$}%
       \ifdimgreater{\olen}{\xlen}%
          {\underset{#1}{\overset{#2}{\longleftarrow}}}%
          {\ifdimgreater{\ulen}{\xlen}%
             {\underset{#1}{\overset{#2}{\longleftarrow}}}
             {\xleftarrow[#1]{#2}}}}%
      {\xleftarrow[#1]{#2}}
   }
\newcommand{\isoarrow}{%
   \ifbool{@display}{\overset{\sim}{\longrightarrow}}{\xrightarrow\sim}%
   }
\begin{document}

\title[Braid representatives and transversal slices]{Good position braid representatives and transversal slices of unipotent orbits}

\thanks{}
\author[Chengze Duan]{Chengze Duan}
\address[C.Duan]{Department of Mathematics, University of Maryland, College Park, MD, 20740}
\email{czduanwork@gmail.com}

\maketitle

\begin{abstract}
Let $G$ be a connected reductive group over an algebraically closed field and $W$ be its Weyl group. Steinberg's cross-section yields a transversal slice of the regular unipotent orbit, which corresponds to the conjugacy class of Coxeter elements in $W$ under Lusztig's map. This was generalized to unipotent orbits associated to any elliptic conjugacy classes of $W$ by He and Lusztig via good minimal length elements. In this paper, we construct transversal slices for all unipotent orbits using the good braid elements in the associated braid monoid of $W$. Furthermore, these braid elements can be generalized again and then related to affine Springer fibers.
\end{abstract}

\section{Introduction}
\subsection{Overview}
Let $G$ be a connected reductive group over an algebraically closed field $\Bk$ with Weyl group $W$ and $\fkg$ be its Lie algebra. In \cite{Ko}, Kostant constructed a section of the adjoint quotient $\fkg \to \fkg \sslash \Ad(G)$, which is also a transversal slice of the regular nilpotent orbit in $\fkg$. Steinberg's cross-section, which is a group version analogue of Kostant section, was constructed in \cite[Theorem 1.4]{St}. It is a transversal slice of the regular unipotent orbit in $G$. It is natural to ask how to construct the transversal slices of more general unipotent orbits. In \cite[\S 3]{HL}, He and Lusztig generalized Steinberg's cross-section and constructed transversal slices for all \textit{basic} unipotent orbits using minimal length elements in elliptic conjugacy classes of $W$ (see \cite[\S 4]{Lu1} for definition of \textit{basic}). These slices are also used in the later work of He on affine Deligne-Lusztig varieties (see \cite[\S 4]{He}).

A crucial point of the construction of He and Lusztig is using minimal length elements which are in good position with respect to the fundamental Weyl chamber in the sense of \cite[\S 5.2]{HN}. However, their construction does not work for all unipotent orbits since the length of minimal length elements are not large enough in non-elliptic case. In this paper, we will generalize the definition of good position to construct \textit{good position braid representatives} in the associated braid monoid of $W$ and use them to construct the new slices. This idea credits to He. Moreover, our construction does not require any assumption on the characteristic of the base field $\Bk$ and works for the disconnected case.

In \cite[Theorem 0.4]{Lu1}, Lusztig constructed a surjective map $\Phi$ from the conjugacy classes in $W$ to unipotent orbits in $G$. It was generalized to twisted case for disconnected groups in \cite[Theorem 1.3]{Lu21}. The map $\Phi$ is a bijection from elliptic conjugacy classes in $W$ to basic unipotent orbits in $G$. However, in general the preimage of an arbitrary unipotent orbit may contain several conjugacy classes in $W$. For any given unipotent orbit, there is a unique ``most elliptic'' conjugacy class in $W$ (see \cite[Theorem 0.2]{Lu2} and \cite[Theorem 1.16]{Lu21}). We shall see that the length of good position braid representatives of this conjugacy class will exactly match our requirements for the dimension of the slice.

Starting from the above braid elements, the new slices will be constructed based on the Deligne-Garside normal form of them. This generalizes He-Lusztig's results and provide transversal slices for all unipotent orbits. Moreover, theses braid elements and slices will be used in a later work of the author to generalize the results in \cite[\S4]{HL} and \cite{Lu3} on Lusztig varieties and Deligne-Lusztig varieties to parabolic version. In a unpublished work of Malten (see \cite{Ma}), he also generalized the slices of He and Lusztig. However, the elements he used are still Weyl group elements.

Finally, we will relate our new braid elements to affine Springer fibers. We will define a more general form of these good braid elements. The length of such elements will reflect the dimension of affine Springer fibers.

\subsection{Main results}
Let $\tilde{W}$ be a twisted finite Coxeter group with associated braid monoid $B^+(\tilde{W})$. 
A \textit{good position pair} consists of an element $\tilde{w}\in \tilde{W}$ and a sequence $\underline{\Theta}$ determined by eigenvalues of $\tilde{w}$-action on its reflection representation (see \S 3.1).

\begin{proposition}\label{A}(see also proposition \ref{a} for more precise statement) Let $(\tilde{w},\underline{\Theta})$ be a good position pair where $\underline{\Theta}$ is an increasing admissible sequence. Then there exists an element $b(\tilde{w},\underline{\Theta})\in B^+(\tilde{W})$, which projects to $\tilde{w}$ under the natural projection, with the ``good element'' property.
\end{proposition}
We refer to \S 2.1 for the definition of increasing admissible sequence and refer to \S 3.1 for the ``good element'' property, which originally due to \cite[Theorem 1.1]{GM}.

Now let $\tilde{G}$ be an affine algebraic group over an algebraic closed field $\Bk$ with identity component $G$ which is reductive. Let $D$ be a connected component of $\tilde{G}$ whose image in $\tilde{G}/G$ is unipotent. Let $W$ be the Weyl group of $G$ and $\tilde{W}=W\rtimes \langle \delta \rangle$ be the twisted Weyl group where $\delta$ is induced by the conjugation action of $D$. The twist $\delta$ extends to a group automorphism of $G$ and we still denote it as $\delta$. Lusztig's map $\Phi$ restricts to a surjection from the set of $W$-conjugacy classes of $\tilde{W}$ contained in $W\delta$ to the unipotent $G$-orbits in $D$. Let $\CO$ be a unipotent $G$-orbit in $D$. We have the following generalization of Steinberg's cross-section.
\begin{theorem}\label{B}
        Let $\CO$ be any unipotent $G$-orbit in $D$ and $\CC$ be the ``most elliptic'' $W$-conjugacy class in $\Phi^{-1}(\CO)$. There exists a braid $\tilde{b}\in B^+(\tilde{W})$, which projects to some $\tilde{w}\in\CC$, and a corresponding slice $S_{Br}^D(\tilde{b})$ in $D$ satisfying: 

    (1)  The conjugation map 
    \begin{equation*}
         \Xi_{\tilde{b}}:U_{R^+\setminus R^{\tilde{w}}}\times S_{Br}^D(\tilde{b})\to S_{Br}^D(\tilde{b})U_{R^+\setminus R^{\tilde{w}}}
    \end{equation*}
    given by $(u,g)\mapsto ugu^{-1}$ is an isomorphism of varieties.

    (2) $S_{Br}^D(\tilde{b})$ intersects any $G$-conjugacy class in $D$ it meets transversally. In particular, it is a transversal slice of $\CO$.
\end{theorem}
We refer to $\S 4.1$ for the definition of $U_{R^+\setminus R^{\tilde{w}}}$. The first part of this result is a generalization of the cross-section property. The second part consists of transversality, which is a condition on the tangent spaces, and the slice condition, which is about the dimension of $S_{Br}^D(\tilde{b})$ and the codimension of $\CO$ in $D$.

Our last main result relates these braid elements to affine Springer fibers. We generalized the definition of good position pair by considering a sequence of indecomposable subspaces in the real reflection representation of $W$. This enable us to construct good braid elements for certain affine Springer fibers. See \S5.1 for the definition of affine Springer fibers $\Gr_{\gamma}$ and the definition of $L\fkg$ and $T_{\gamma}$.
\begin{proposition}\label{C}(see also proposition \ref{c} for more precise statement)        Let $\gamma\in L\fkg$ be topologically nilpotent and generically regular semisimple. There exists $b\in B^+(W)$ which projects to an element in the conjugacy class $\CC$ in $W$ and
        \begin{equation*}
        \dim \Gr_{\gamma}=\frac{l(b)-(r-r_{\gamma})}{2},
\end{equation*}   
    where $r$ (resp. $r_{\gamma}$) is the semisimple rank of $G$ (resp. maximal split subtorus of $T_\gamma$).
\end{proposition}
The conjugacy class $\CC$ is determined by the maximal torus $T_\gamma$. This result gave a reformulation of the dimension formula of affine Springer fibers using our braid elements when the characteristic of $\Bk$ is 0.

\begin{acknowledgement}
I would like to thank my advisor Xuhua He for suggesting the problem and his idea of using braid elements. I also thank Jeffrey Adams, Ana B\u{a}libanu, Thomas Haines, Sian Nie, Minh-T\^am Trinh and Kaitao Xie for helpful discussions and valuable feedbacks.
\end{acknowledgement}

\section{Preliminaries}

\subsection{Twisted Coxeter group and good position elements}
Let $\tilde{W}=W \rtimes \langle \delta \rangle$ be a twisted finite Coxeter group where $W$ is a finite Coxeter group and $\delta$ is a twist of $W$. Let $\{s_i \mid i\in I\}$ be the set of simple reflections of $W$. For any subset $J\subset I$, let $W_J$ be the corresponding standard parabolic subgroup of $W$. The twist $\delta: W \to W$ is a group automorphism preserving its simple reflections. The length function on $W$ induces a length function on $\tilde{W}$ where $l(\tilde{w})=l(w)$ for any $\tilde{w}= w \delta^i\in \tilde{W}$. Denote the set of conjugacy classes (resp. $W$-conjugacy classes) of $W$ (resp. $\tilde{W}$) by $[W]$ (resp. $[\tilde{W}]$). Two elements $w,w'\in W$ are $\delta$-conjugate if $w'=x^{-1}w\delta(x)$ for some $x\in W$. The map $w\mapsto w\delta$ gives a bijection between the $\delta$-conjugacy classes of $W$ and the $W$-conjugacy classes of $\tilde{W}$ contained in the coset $W\delta$ (see \cite[Remark 2.1]{GKP}).

A $\delta$-conjugacy class $\CC^{\delta}$ of $W$ is called $\delta$-\textit{elliptic} if $\CC^{\delta} \cap W_J=\emptyset$ for any $\delta$-stable proper subset $J$ of $I$. A $W$-conjugacy class $\tilde{\CC}$ of $\tilde{W}$ is called \textit{elliptic} if $\tilde{\CC}\cap (W_J\rtimes \langle \delta \rangle)=\emptyset$ for any $\delta$-stable proper subset $J$ of $I$.

Let $V$ be a real reflection representation of $\tilde{W}$ and $(,)$ be the inner product preserved by $\tilde{W}$. Let $\mathfrak{H}$ be the set of hyperplanes of $V$ corresponding to the reflections in $W$. If $W$ is the Weyl group of an algebraic group, then $\fkH$ is the set of root hyperplanes. A connected component of $V- \underset{H\in \fkH}{\bigcup} H$ is called a \textit{Weyl chamber}. Set $V_{\BC}=V\otimes_{\BR}\BC$. The $\tilde{W}$-action and the bilinear form $(,)$ can be extended to $V_{\BC}$ naturally and we still denote it as $(,)$. For any $\theta$, let $(V_{\BC})^{\theta}_{\tilde{w}}$ be the $e^{i\theta}$-eigenspace of $\tilde{w}$ in $V_{\BC}$. 

For any subset $V'$ of $V$, define $\mathfrak{H}_{V'}$ as the set of hyperplanes in $\mathfrak{H}$ containing $V'$ and $W_{V'}$ as the subgroup of $W$ consisting of elements fixing $V'$. We also define $I(V')$ as the subset of $I$ corresponding to the simple reflections that fix $V'$. A point $x\in V'$ is called \textit{regular} if it is not contained in any hyperplane $H\in \mathfrak{H}-\mathfrak{H}_{V'}$. We denote the subset of $V'$ consisting of its regular points by $(V')^{\text{reg}}$.

For any $\tilde{w} \in \tilde{W}$, define 
\begin{equation*}
    \Gamma_{\tilde{w}}^{[0,\pi]}=\{\theta \in [0,\pi] \mid e^{i\theta} \text{ is an eigenvalue of } \tilde{w} \text{ on } V_{\BC}\}.
 \end{equation*}
For $\theta \in \Gamma_{\tilde{w}}^{[0,\pi]}$, define
\begin{equation*}
   V_{\tilde{w}}^{\theta}=\{v \in V \mid \tilde{w}(v) + \tilde{w}^{-1}(v)=2 \cos \theta\cdot v\}.
\end{equation*}
For convenience, in this paper these $V_{\tilde{w}}^{\theta}$ are called \textit{real eigenspaces} of $\tilde{w}$. Then we have $V_{\tilde{w}}^{\theta}\otimes_{\BR}\BC=(V_{\BC})_{\tilde{w}}^{\theta} \oplus (V_{\BC})_{\tilde{w}}^{-\theta}$. It is also clear that $V=\bigoplus\limits_{\theta \in \Gamma_{\tilde{w}}^{[0,\pi]}} V_{\tilde{w}}^{\theta}$. 

Let $\underline{\Theta}=(\theta_1,\theta_2,\ldots,\theta_m)$ be a sequence of elements in $\Gamma_{\tilde{w}}^{[0,\pi]}$. We call $\underline{\Theta}$ \textit{increasing} if $\theta_1<\theta_2<\cdots<\theta_m$. Define $F_i=\sum\limits_{j=1}^i V_{\tilde{w}}^{\theta_j}$ for any $1\leq i \leq m$ and $F_0=0$. This gives a filtration of subspaces of $V$,
\begin{equation*}
    0=F_0 \subset F_1 \subset \cdots \subset F_m \subset V.
\end{equation*}
We say $\underline{\Theta}$ is \textit{admissible} if $F_m \cap V^{\text{reg}}$ is nonempty. In particular, we say $\underline{\Theta}$ is \textit{complete} if $F_m=V$. Notice that the filtration above gives another filtration
\begin{equation*}
   W=W_{F_0} \supset W_{F_1} \supset \cdots \supset W_{F_m}
\end{equation*}
of subgroups of $W$. It is clear that $W_{F_m}$ is trivial if $\underline{\Theta}$ is \textit{admissible}. The \textit{irredundant} subsequence of $\underline{\Theta}$ is $r(\underline{\Theta})=(\theta_{i_1},\theta_{i_2},\ldots,\theta_{i_l})$ such that $W_{F_{i_j}}=W_{F_{i_j+1}}=\cdots=W_{F_{i_{j+1}-1}}\neq W_{F_{i_{j+1}}}$ for any $0\leq j \leq l-1$ where $i_0$ is set as $0$. It is clear that $\fkH_{F_{i_j}}=\fkH_{F_{i_j+1}}=\cdots=\fkH_{F_{i_{j+1}-1}}\neq \fkH_{F_{i_{j+1}}}$ for any $0\leq j \leq l-1$.

Let $C$ be a Weyl chamber in $V- \underset{H\in \fkH}{\bigcup} H$ and $(\tilde{w},\underline{\Theta})$ be a pair as above. We say $C$ is \textit{in good position} with respect to $(\tilde{w},\underline{\Theta})$ if $\overline{C} \cap  F_i^{\reg}$ is nonempty for all $i$. One can check that this definition is equivalent to the original definition in \cite[\S 5.2]{HN}. Moreover, if the sequence $\underline{\Theta}$ is increasing and complete, then we say $C$ is \textit{in good position} with respect to $\tilde{w}$. In particular, we call $\tilde{w}$ a \textit{good position element} if $C$ is the fundamental Weyl chamber.

Let $B^+(W)$ be the positive braid monoid associated to $W$. Let $\pi_{W}: B^+(W) \to W$ be the natural projection map and $j_{W}: W \to B^+(W)$ be the natural embedding. These maps naturally extend to the twisted case for $\tilde{W}$ and $B^+(\tilde{W})$ and we denote them by $\pi_{\tilde{W}}$ and $j_{\tilde{W}}$. For convenience, we write $\underline{\tilde{w}}$ for $j_{\tilde{W}}(\tilde{w})$. Recall the definition in \cite[Theorem 1.1]{GM}, an element $\tilde{w}\in \tilde{W}$ is called \textit{good} if there exist a filtration $I\supset I_0 \supsetneq I_1 \supsetneq \cdots \supsetneq I_l$, a twist $\sigma \in \langle \delta \rangle$ and $d_0,\ldots,d_l\in 2\BZ_{>0}$ such that
\begin{equation*}
   \underline{\tilde{w}}^d=\underline{w_0}^{d_0}\cdots \underline{w_l}^{d_l}\sigma,
\end{equation*}   
where $d$ is the order of $\tilde{w}$ and $w_j$ is the longest element of $W_{I_j}$ for any $0\leq j\leq l$. We recall the following result of He and Nie.

\begin{theorem}\cite[Theorem 5.3]{HN} \label{2.1}
Let $(\tilde{w},\underline{\Theta})$ be a pair where $\underline{\Theta}$ is an increasing admissible sequence in $\Gamma_{\tilde{w}}^{[0,\pi]}$ with $r(\underline{\Theta})=(\theta_{i_1},\theta_{i_2},\ldots,\theta_{i_l})$. Suppose that the fundamental Weyl chamber $C_0$ is in good position with respect to $(\Tilde{w},\underline{\Theta})$. Then
\begin{equation*}
   \underline{\tilde{w}}^d= \underline{w_0}^\frac{d\theta_{i_1}}{\pi} \underline{w_1}^\frac{d(\theta_{i_2}-\theta_{i_1})}{\pi} \cdots \underline{w_{l-1}}^\frac{d(\theta_{i_l}-\theta_{i_{l-1}})}{\pi}\sigma^d,
\end{equation*}   
where $\sigma \in \langle \delta \rangle$ is the twist part of $\tilde{w}$. Here $d\in \BN$ such that $d\frac{\theta_{i_j}}{2\pi}\in \BZ$, and $w_j$ is the longest element in $W_{F_{i_j}}$ for $0\leq j \leq l$. Moreover, if $d$ is even, then 
\begin{equation*}
   \underline{\tilde{w}}^{\frac{d}{2}}=\underline{w_0}^\frac{d\theta_{i_1}}{2\pi} \underline{w_1}^\frac{d(\theta_{i_2}-\theta_{i_1})}{2\pi} \cdots \underline{w_{l-1}}^\frac{d(\theta_{i_l}-\theta_{i_{l-1}})}{2\pi}\sigma^{\frac{d}{2}} .
\end{equation*}   
\end{theorem}   
\begin{remark}
    The moreover part is also called \textit{very good} in \cite[Theorem 1.1]{GM}.
\end{remark}
By \cite[\S 5.3]{HN}, good position elements are of minimal length in its conjugacy class.

\subsection{Unipotent orbits and Lusztig's map}
We recall some facts on unipotent orbits and Lusztig's map $\Phi$, which will be used in \S 3.4 for the slice condition. Let $G$ be the identity component of an affine algebraic group $\tilde{G}$ over an algebraic closed field $\Bk$ with $\text{char}(\Bk) \geq 0$. For convenience, in this subsection we may assume $G$ is almost simple and $\tilde{G}/G$ is cyclic. Let $D$ be a fixed connected component of $\tilde{G}$ whose image in $\tilde{G}/G$ is unipotent (i.e. its order is a power of $p$ if $p>1$ and $1$ if $p=0$). Let $[D_u]$ be the set of unipotent $G$-orbits in $D$. The conjugation action of $D$ on $G$ defines a twist $\delta:W \to W$ of the Weyl group $W$ of $G$ (see \cite[\S 1.1]{Lu21}). We denote the corresponding group automorphism of $G$ also as $\delta$. One can find $\tilde{g}_D\in D$ such that $\tilde{g}_Dx\tilde{g}_D^{-1}=\delta(x)$ for any $x\in G$. In this paper, following \cite{Lu21}, the disconnected case when $G$ is of type $\text{D}_4$ and $\text{char}(\Bk)=3$ is viewed as an exceptional case.

We first recall a combinatorial parametrization of $[D_u]$ for classical groups. Let $\mathcal{P}(n)$ be the set of partitions of $n$. For any partition $\lambda=(\lambda_1,\ldots , \lambda_m)$ and any $k\in \BZ_{>0}$, define $m_{\lambda}(k)=|\{i\mid \lambda_i=k\}|$. For $\epsilon=\pm 1$, we set 
\begin{equation*}
    \mathcal{P}_{\epsilon}(n)=\{\lambda \in \mathcal{P}(n)\mid m_{\lambda}(k) \text{ is even if } (-1)^k=\epsilon \}.
\end{equation*}

(1) If $\text{char}(\Bk)\neq 2$, then unipotent elements of $\tilde{G}$ are all contained in $G$ and $D$ must be equal to $G$. The unipotent $G$-orbits of $\tilde{G}$ are just unipotent orbits of $G$, which are parametrized as follows based on their Jordan forms. 
\begin{itemize}
  \item Type $\text{A}_n$: $\mathcal{P}(n+1)$;
  \item Type $\text{B}_n$: $\mathcal{P}_{1}(2n+1)$;
  \item Type $\text{C}_n$: $\mathcal{P}_{-1}(2n)$;
  \item Type $\text{D}_{2n+1}$: $\mathcal{P}_{1}(4n+2)$;
  \item Type $\text{D}_{2n}$: The unipotent orbits are parametrized by $\mathcal{P}_{1}(4n)$, except that every partition $\lambda$ with $\lambda_i$ all even actually corresponds to two unipotent orbits.
\end{itemize}   

(2) If $\text{char}(\Bk)= 2$, then $D$ may not be $G$ only in two cases (type $\text{A}$ and $\text{D}$) and the unipotent $G$-orbits of $D$ are parametrized as follows.
\begin{itemize}
  \item Type $\text{A}_n$ and $D=G$: Same as the case when $\text{char}(\Bk)\neq 2$;
  \item Type $\text{A}_n$ and $D\neq G$ (Type ${}^2\text{A}_n$): In this case, following \cite[\S I.2.7]{Sp}, we assume $\tilde{G}=G\cup D$ where $G=\GL_n$ and $D$ is the set of non-singular bilinear forms $\phi:V\times V\to \Bk$ by identifying $G$ as $\GL(V)$. 
  
  The group structure on $\tilde{G}$ is defined as follows. The structure on $G$ is usual. For $g\in G$ and $\phi\in D$, define $(g\phi)(v_1,v_2)=\phi(g^{-1}v_1,v_2)$ and $(\phi g)(v_1,v_2)=\phi(v_1,gv_2)$. For $\phi,\psi\in D$, define $\phi\psi$ as the unique element of $G$ such that $\phi((\phi\psi)v_1,v_2)=\psi(v_2,v_1)$. 
  It is easy to check that $\tilde{G}/G\simeq \BZ/ 2\BZ$. The twist $\delta:G\to G$ defined by $D$ is $\delta (g)=(g^\intercal)^{-1}$. The corresponding twist of $W$ is given by the conjugation of the longest element $w_0$.

  Define $\widetilde{\CP_1}(n+1)$ to be the set of pairs $(\lambda,\epsilon)$, where:
  \begin{itemize}
      \item   (i) $\lambda=(\lambda_1,\ldots,\lambda_m)\in \CP_1(n+1)$;
    \item   (ii) $\epsilon:\{2k+1\mid k\in \BZ_{\geq 0}, m_{\lambda}(2k+1)\neq 0\} \to \{0,1\}$ satisfying that $\epsilon(2k+1)=1$ if $m_{\lambda}(2k+1)$ is odd.
  \end{itemize}
  Let $\phi\in D$ be any unipotent element. By \cite[\S 2.7]{Sp}, the conjugacy class of $\phi$ corresponds to some $(\lambda,\epsilon_{\lambda})\in \widetilde{\CP_1}(n+1)$. Specifically, we have $u=\phi^2$ is unipotent in $G$ and $\lambda$ is the Jordan form of $u$. As for $\epsilon_{\lambda}$, we have $\epsilon_{\lambda}(2k+1)=0$ if $m_{\lambda}(2k+1)\neq 0$ is even and $\phi(v,(u-\id)^{2k}v)=0$ for any $v\in \ker(u-\id)^{2k-1}$. Otherwise, we set $\epsilon_{\lambda}(2k+1)=1$.
  
  \item Type $\text{B}_n, \text{C}_n, \text{D}_n$ and $D=G$: In this case, Jordan form cannot uniquely determine the unipotent orbit. Let $u$ be a unipotent element in $G$. Recall in \cite[\S 3]{Hes}, every indecomposable block of $u$ is of the form $W(k)$ or $V(2k)$. Here $W(k)$, which has Jordan form $(k,k)$, corresponds to a regular unipotent element in a subgroup $\SL_k$ of $G$ and $V(2k)$, which has Jordan form $(2k)$, corresponds to a regular unipotent element in some $\text{Sp}_{2k}$ which lies in an orthogonal group $\text{O}_{2k}$. Moreover, the number of $V(2k)$ blocks is no larger than $2$ for any $k$ since $V(2k)^3\simeq W(2k)+V(2k)$ by \cite[Lemma 3.6]{Hes}.
      
  Following \cite[\S 2.6]{Sp}, define $\widetilde{\mathcal{P}_{-1}}(2n)$ as the set of pairs $(\lambda,\epsilon)$, where: 
  \begin{itemize}
      \item   (i) $\lambda=(\lambda_1, \ldots, \lambda_m) \in \mathcal{P}_{-1}(2n)$;
      \item     (ii) $\epsilon: \{2k \mid k\in \BZ_{>0}, m_{\lambda}(2k)\neq 0\}\to \{0,1\}$ satisfying that $\epsilon(2k)=1$ if $m_{\lambda}(2k)$ is odd. 
  \end{itemize}
  Now let $\CO$ be any unipotent orbit in $\Sp_{2n}$ or $\text{O}_{2n}$, its Jordan form gives a partition $\lambda$ in $\mathcal{P}_{-1}(2n)$. For any even number $2k>0$, set $\epsilon_{\lambda}(2k)=1$ if there are $V(2k)$ blocks and $\epsilon_{\lambda}(2k)=0$ otherwise. This gives a bijection between the unipotent orbits in $\Sp_{2n}$ or $\text{O}_{2n}$ and $\widetilde{\mathcal{P}_{-1}}(2n)$. Since there is a bijection between unipotent orbits in $\Sp_{2n}$ and $\SO_{2n+1}$, the unipotent orbits in type $\text{B}_n$ and $\text{C}_n$ are both parametrized by $\widetilde{\mathcal{P}_{-1}}(2n)$.
  
  For type $\text{D}_n$, we know the unipotent $G$-orbits in $\SO_{2n}$ are exactly the unipotent $\tilde{G}$-orbits of $\text{O}_{2n}$ contained in $\SO_{2n}$, which are parametrized by $\{(\lambda,\epsilon) \in \widetilde{\mathcal{P}_{-1}}(2n) \mid \lambda=(\lambda_1, \ldots, \lambda_m) \text{ with } m \text{ even}\}$, except that $\CO_{\lambda,\epsilon}$ splits into two $G$-orbits if $\lambda_i$ and $m_{\lambda}(i)$ are even for all $i$ and $\epsilon=0$. However, in this paper we denote these two orbits as the same for convenience since they have the same dimension.

\item Type $\text{D}_n$ and $D\neq G$: In this case, by above we have the unipotent $G$-orbits in $D$ are parametrized by $\{(\lambda,\epsilon) \in \widetilde{\mathcal{P}_{-1}}(2n) \mid \lambda=(\lambda_1, \ldots, \lambda_m) \text{ with } m \text{ odd}\}$. 
\end{itemize}

Based on the parametrization above, we have an explicit description of Lusztig's map $\Phi$ for classical groups. First recall the conjugacy classes of $\tilde{W}$. We know $\tilde{W}=W$ in type $\text{B}$ and $\text{C}$. It is also well-known that $W(\text{B}_n)\simeq W(\text{C}_n)$ and $W(\text{D}_n)$ is a subgroup of $W(\text{C}_n)$ of index 2.
   
\begin{definition} 
   A bipartition of $n$ is a pair $(\lambda, \mu)$ where $\lambda=(\lambda_1,\ldots, \lambda_a)$ and $\mu=(\mu_1,\ldots,\mu_b)$ are two partitions with $\sum\limits_i \lambda_i+\sum\limits_j \mu_j=n$.
\end{definition}
Let $\mathcal{BP}(n)$ be the set of bipartitions of $n$. The conjugacy classes of $W(\text{C}_n)$, $W(\text{B}_n)$ and $\tilde{W}(\text{D}_n)$ are parametrized by $\mathcal{BP}(n)$ based on its corresponding signed cycle types. We let $\lambda$ be corresponding to all negative cycles and $\mu$ be corresponding to all positive cycles. If $\lambda$ has even number of parts, then the corresponding conjugacy class of $\tilde{W}(\text{D}_n)$ lies in $W(\text{D}_n)$. Moreover, if $\lambda$ is empty and $\mu_i$ is even for all $i$, then there are two conjugacy classes corresponding to this bipartition. (See \cite[\S 7]{Ca} for details on signed cycle types and classification). 

Given a partition $\lambda=(\lambda_1,\ldots,\lambda_a)$, we define $\psi_{\alpha}:\{1,2,\ldots,a\}\to 
\{1,0,-1\}$ as follows. We formally set $\lambda_0=\lambda_{a+1}=0$. If $i$ is odd and $\lambda_{i-1}\neq \lambda_i$, then set $\psi(i)=1$. If $i$ is even and $\lambda_{i+1}\neq \lambda_i$, then set $\psi(i)=-1$. For all remaining cases, set $\psi(i)=0$. We have the following explicit description of $\Phi:[\tilde{W}]\to [\tilde{G}_u]$ where $[\tilde{G}_u]$ is the set of all unipotent $G$-conjugacy classes in $\tilde{G}$. The description  is computed based on \cite[\S 2]{Lu1} and \cite[\S 3, \S 5]{Lu21}.
\begin{itemize}
   \item Type $\text{A}_n$: The conjugacy classes of $W$ are parametrized by $\mathcal{P}(n+1)$ based on its corresponding cycle types. For any $\lambda\in \mathcal{P}(n+1)$, Lusztig's map $\Phi$ sends $\CC_{\lambda}\in [W]$ to $\CO_{\lambda}$ in all characteristics.

   \item Type ${}^2 \text{A}_n$: The conjugacy classes of $\tilde{W}=W\rtimes \langle \delta \rangle$ can be divided into two part. Now we denote a conjugacy class lying in $W$ as $\CC^W$ and a conjugacy class lying in $W\delta$ as $\CC^{W\delta}$. The conjugacy classes lying in $W$ are again parametrized by $\CP(n+1)$ and the map $\Phi$ again sends $\CC_{\lambda}^W$ to the unipotent $G$-orbit $\CO_{\lambda}^G$ in $[G_u]$ in all characteristics. 
   
   Conjugacy classes lying in $W\delta$ are also parametrized by $\CP(n+1)$. Specifically, the conjugacy class of $\tilde{w}=w\delta$ is parametrized by the cycle type of $ww_0$ where $w_0\in W$ is the longest element. We only need to care about $\Phi$ if $\text{char}(\Bk)=2$ as $[D_u]$ is empty otherwise. Then $\Phi$ will send the conjugacy class $\CC_{\lambda}^{W\delta}$ to the unipotent $G$-orbit $\CO_{(\lambda',\epsilon_{\lambda'})}^{D}$ in $[D_u]$ where $(\lambda',\epsilon_{\lambda'})\in \widetilde{\CP_1}(n+1)$ satisfying
   \begin{equation*}
     m_{\lambda'}(k)=
    \begin{cases}
      m_{\lambda}(k)+2m_{\lambda}(2k) & \text{$k$ odd},\\
      2m_{\lambda}(2k) & \text{$k$ even},
    \end{cases}       
    \end{equation*}
   and
      \begin{equation*}
     \epsilon_{\lambda'}(2k+1)=
    \begin{cases}
      1 & m_{\lambda}(2k+1)\neq 0,\\
      0 & m_{\lambda}(2k+1)=0 \text{ and } m_{\lambda}(4k+2)\neq 0.
    \end{cases}       
    \end{equation*}
   
   \item Type $\text{C}_n$: Let $(\lambda,\mu)\in \mathcal{BP}(n)$ be a bipartition where $\lambda=(\lambda_1,\ldots,\lambda_a)\in \mathcal{P}(n_1)$ and $\mu=(\mu_1,\ldots,\mu_b)\in \mathcal{P}(n_2)$. Define a partition $\tilde{\lambda}=(2\lambda_1,\ldots, 2\lambda_a)$. Define another partition $\tilde{\mu}=(\mu_1,\mu_1,\mu_2,\mu_2,\ldots, \mu_b,\mu_b )$. By putting all the parts of $\tilde{\lambda}$ and $\tilde{\mu}$ together and rearranging them, we then obtain a partition $\nu \in \mathcal{P}_{-1}(2n)$.
   
   If $\text{char}(\Bk)\neq 2$, then Lusztig's map $\Phi$ sends $\CC_{(\lambda,\mu)}$ to $\CO_{\nu}$.

   If $\text{char}(\Bk)=2$, then Lusztig's map $\Phi$ sends $\CC_{(\lambda,\mu)}$ to $\CO_{(\nu,\epsilon_{\nu})}$. Here $\epsilon_{\nu}(2k)=1$ if $m_{\lambda}(k)\neq 0$ and $\epsilon_{\nu}(2k)=0$ otherwise.
   
   \item Type $\text{B}_n$: If $\text{char}(\Bk)= 2$, then $\Phi$ for type $\text{B}_n$ is exactly the same as $\Phi$ for type $\text{C}_n$ under the bijection of unipotent classes in type $\text{B}_n$ and $\text{C}_n$. 
   
   If $\text{char}(\Bk)\neq 2$, still let $(\lambda,\mu)\in \mathcal{BP}(n)$ be as above. Define a partition $\tilde{\lambda}=(2\lambda_1+\psi_{\lambda}(1),\ldots, 2\lambda_a+\psi_{\lambda}(a) )$ if $a$ is odd or $\tilde{\lambda}=(2\lambda_1+\psi_{\lambda}(1),\ldots, 2\lambda_a+\psi_{\lambda}(a),1 )$ if $a$ is even. Define another partition $\tilde{\mu}=(\mu_1,\mu_1,\mu_2,\mu_2,\ldots, \mu_b,\mu_b )$. One can check that $\tilde{\lambda}\in \mathcal{P}_{1}(2n_1+1)$ and $\tilde{\lambda}\in \mathcal{P}_{1}(2n_2)$. By putting all the parts of $\tilde{\lambda}$ and $\tilde{\mu}$ together and rearrange them, we then obtain a partition $\nu \in \mathcal{P}_{1}(2n+1)$. Lusztig's map $\Phi$ sends $\CC_{(\lambda,\mu)}$ to $\CO_{\nu}$.

   \item Type $\text{D}_n$: If $\text{char}(\Bk)= 2$, then $\Phi$ for type $\text{D}_n$ is just a restriction $\Phi$ for type $\text{C}_n$, except for the case when $\CC_{(\lambda,\mu)}\in [W(\text{D}_n)]$ splits into $(\CC_{(\lambda,\mu)})_{\Romannum{1}}$ and $(\CC_{(\lambda,\mu)})_{\Romannum{2}}$ and $\CO_{(\nu,\epsilon_{\nu})}\in [G_u]$ also splits into $(\CO_{(\nu,\epsilon_{\nu})})_{\Romannum{1}}=\Phi((\CC_{(\lambda,\mu)})_{\Romannum{1}})$ and $(\CO_{(\nu,\epsilon_{\nu})})_{\Romannum{2}}=\Phi((\CC_{(\lambda,\mu)})_{\Romannum{2}})$.
   
   If $\text{char}(\Bk)\neq 2$, Let $(\lambda,\mu)\in \mathcal{BP}(n)$ be a bipartition where $\lambda=(\lambda_1,\ldots,\lambda_a)\in \mathcal{P}(n_1)$ and $\mu=(\mu_1,\ldots,\mu_b)\in \mathcal{P}(n_2)$ with $a$ even. Define a partition $\tilde{\lambda}=(2\lambda_1+\psi_{\lambda}(1),\ldots, 2\lambda_a+\psi_{\lambda}(a),1 )$. Define another partition $\tilde{\mu}=(\mu_1,\mu_1,\mu_2,\mu_2,\ldots, \mu_b,\mu_b )$. By putting all the parts of $\tilde{\lambda}$ and $\tilde{\mu}$ together and rearrangeing them, we then obtain a partition $\nu \in \mathcal{P}_{1}(2n)$. Lusztig's map $\Phi$ sends $\CC_{(\lambda,\mu)}$ to $\CO_{\nu}$, except for the case when $\CC_{(\lambda,\mu)}$ splits into $(\CC_{(\lambda,\mu)})_{\Romannum{1}}$ and $(\CC_{(\lambda,\mu)})_{\Romannum{2}}$ and $\CO_{\nu}$ also splits into $(\CO_{\nu})_{\Romannum{1}}=\Phi((\CC_{(\lambda,\mu)})_{\Romannum{1}})$ and $(\CO_{\nu})_{\Romannum{2}}=\Phi((\CC_{(\lambda,\mu)})_{\Romannum{2}})$;
   \item Type $\text{D}_n$ with $D\neq G$: We again only need to care about $\Phi$ when $\text{char}(\Bk)=2$ as $[D_u]$ is empty otherwise. In this case, the map $\Phi$ is just a restriction of type $\text{C}_n$ which sends $\CC_{(\lambda,\mu)}\in [W\delta]\subset [\tilde{W}]$ to $\CO_{(\nu,\epsilon_{\nu})}$. We have $\CO_{(\nu,\epsilon_{\nu})}\subset D$ since $\lambda$ will have an odd number of parts in this case.
   \item 
For exceptional cases, one can check \cite[\S 6]{AHN} for elliptic case. As for non-elliptic case, one may follow the steps in \cite[\S 6]{Mi}.
\end{itemize}

\subsection{Strata of reductive groups}
In this subsection we assume $G$ is defined over a field of characteristic $0$. Let $\mathscr{P}$ be the set of prime numbers. Set $G^{(0)}=G$ and $G^{(r)}$ be a connected reductive group over an algebraically closed field of characteristic $r$ of the same type as $G$ for any $r\in \mathscr{P}$. For $r\in \{0\}\cup \mathscr{P}$, Let $[G^{(r)}_u]$ be the set of unipotent orbits of $G^{(r)}$. Let $\text{Irr}(W)$ be the set of isomorphism classes of irreducible representations of $W$ over $\BQ$ and $\mathcal{S}_2(W)\subset \text{Irr}(W)$ be the subset of isomorphism classes of all 2-special representations (see \cite[\S 1.1]{Lu4} for 2-special representations).

In \cite[\S2]{Lu4}, Lusztig defined a map from $G$ to $\mathcal{S}_2(W)$ and the fibers of this map are called the \textit{strata} of $G$. Let $\Str (G)$ be the set of strata of $G$. Lusztig proved that one can identify each stratum $\Sigma \in \Str(G)$ with a set $\{\CO_i^{(r_i)} \mid \CO_i^{(r_i)}\in  [G^{(r_i)}_u], r_i\in \{0\}\cup \mathscr{P}\}$ of unipotent orbits in all characteristics. Moreover, for a given stratum, all $\CO_i^{(r_i)}$ have the same dimension, independent of $r_i$.

In \cite[\S4]{Lu4}, Lusztig defined a surjective map $'\Phi$ from $[W]$ to $\mathcal{S}_2(W)$ and proved that this is compatible with Lusztig's map $\Phi^{(r)}$ in any characteristic for any $r\in \{0\}\cup \mathscr{P}$. In other words, for any $r,r'\in \{0\}\cup \mathscr{P}$, the unipotent orbits $\Phi^{(r)}(\CC)$ and $\Phi^{(r')}(\CC)$ are both in the stratum corresponding to $'\Phi(\CC)$. Therefore, the map $'\Phi$ is actually the mixed characteristic version of Lusztig's map.

In \cite[\S 6]{Lu21}, the strata is extended to the disconnected case and we denote it as $\Str(\tilde{G}):=\Str(G)\sqcup \Str(D)$ if $D\neq G$. The map $'\Phi$ is also extended to the disconnected case. However, it is only a map to unipotent $G$-orbits when restricted to $r=2$.

\subsection{Transversal slices}
In this subsection we recall the previous work on transversal slices. Fix a Borel subgroup $B$ of $G$ and let $B^-$ be its opposite Borel subgroup. Let $U,U^-$ be the unipotent radicals of $B,B^-$ respectively and $T=B\cap B^-$ be a maximal torus of $G$ with Weyl group $W$. Let $R$ be a root system of $W$ and $R^+\subset R$ be the set of positive roots. For any subset $R'$ of $R$, set $(R')^+=R'\cap R^+$.

Let $w\in W$ be a Coxeter element (a product of all simple reflections) and $\dot{w}$ be a representative of $w$. Set $U^{w}=U\cap \dot{w}U^- \dot{w}^{-1}$. Steinberg proved that $U^w\dot{w}$ is a transversal slice of the regular unipotent orbit in $G$. Furthermore, Steinberg stated the following results.
\begin{itemize}
    \item (1) The conjugation action of $U$ on $U\dot{w}U$ has trivial isotropy groups;
    \item (2) The intersection of $U^{w} \dot{w}$ with any $U$-orbit on $U\dot{w}U$ is a singleton;
    \item (3) The set of $U$-orbits on $U\dot{w}U$ is isomorphic to $\BA^{l(w)}$.
\end{itemize}
However, Steinberg did not give a proof for these. Now let $w$ be a minimal length element in any elliptic conjugacy class $\CO\in [W]$, Lusztig proved (1) is true in \cite[\S 5]{Lu1}. Moreover, if $G$ is of classical type, Lusztig proved (3) is also true in \cite[Theorem 0.4]{Lu3}. Later in the paper of He and Lusztig, they proved that (2) and (3) are true in \cite[Theorem 3.6]{HL} for any $G$ and even in twisted case. Actually $S(w):=U^{w} \dot{w}$ is a transversal slice of certain unipotent orbit of $G$.
\begin{remark} \label{2.R}
   Different choices of the representative $\dot{w}$ will give different slices. However, for convenience we will abuse the notation and denote them all as $S(w)$. 
\end{remark}

\begin{theorem}\cite[Theorem 3.6]{HL} \label{2.4}
    Let $w$ be a minimal length element in some elliptic conjugacy class $\CC \in [W]$. There is an isomorphism of varieties
        \begin{equation*}
         \Xi_{\tilde{b}}:U\times S(w)\to S(w)U
    \end{equation*}
    given by conjugation $(u,g)\mapsto ugu^{-1}$.
    Moreover, the slice $S(w)$ intersects the unipotent orbit $\Phi(\CC)$ transversally.
\end{theorem}

\section{Good position representatives in braid monoids}

In this section, we generalize the definition of \text{good position} elements and construct the braid representatives in $B^+(\tilde{W})$.

\subsection{Good position braid representatives}
Keep the notations in \S 2.1. For any element $\tilde{w} \in \tilde{W}$, define
\begin{equation*}
    \Gamma_{\tilde{w}}=\{\theta\in [0,\infty) \mid e^{i\theta} \text{ is an eigenvalue of } \tilde{w} \text{ on } V_{\BC}\}.
\end{equation*}
For any $\theta \in \Gamma_{\tilde{w}}$, define $V_{\tilde{w}}^{\theta}$ and $(V_{\BC})_{\tilde{w}}^{\theta}$ as in \S 2.1. Still focus on the position of Weyl chambers with respect to $(\tilde{w},\underline{\Theta})$ and $\underline{\Theta}$ is a sequence in $[0,\infty)$ in this case. 
\begin{definition}\label{3.1}
    A pair $(\tilde{w},\underline{\Theta})$ as above is called a \textit{good position pair} if the fundamental Weyl chamber $C_0$ is in good position with respect to it.
\end{definition}

Now we give the complete form of our first main result Proposition \ref{A}.

\begin{proposition}\label{a}
   Let $(\tilde{w},\underline{\Theta})$ be a good position pair where $\underline{\Theta}$ is increasing and admissible in $\Gamma_{\tilde{w}}$ with irredundant subsequence $r(\underline{\Theta})=(\theta_{i_1},\theta_{i_2},\ldots,\theta_{i_l})$. Write $\tilde{w}= w\sigma$ with $\sigma\in \langle \delta \rangle$ and $w\in W$. Then there exists a braid $b(\tilde{w},\underline{\Theta})= b(w,\underline{\Theta})\sigma\in B^+(\tilde{W})$ such that $b(\tilde{w},\underline{\Theta})$ projects to $\tilde{w}$ under the natural projection and
\begin{equation*}
   b(\tilde{w},\underline{\Theta})^d= \underline{w_0}^\frac{d\theta_{i_1}}{\pi} \underline{w_1}^\frac{d(\theta_{i_2}-\theta_{i_1})}{\pi} \cdots \underline{w_{l-1}}^\frac{d(\theta_{i_l}-\theta_{i_{l-1}})}{\pi}\sigma^d.
\end{equation*}   
Here $d\in \BN$ such that $d\frac{\theta_{i_j}}{2\pi}\in \BZ$, and $w_j$ is the longest element in $W_{F_{i_j}}$ for $0\leq j \leq l$. Furthermore, if $d$ is even, then 
\begin{equation*}
   b(\tilde{w},\underline{\Theta})^{\frac{d}{2}}=\underline{w_0}^\frac{d\theta_{i_1}}{2\pi} \underline{w_1}^\frac{d(\theta_{i_2}-\theta_{i_1})}{2\pi} \cdots \underline{w_{l-1}}^\frac{d(\theta_{i_l}-\theta_{i_{l-1}})}{2\pi}\sigma^{\frac{d}{2}}.
\end{equation*}   
\end{proposition}

Based on above result, we have the following definition.

\begin{definition}\label{3.2}
   Let $(\tilde{w},\underline{\Theta})$ be a good position pair. The element $b(\tilde{w},\underline{\Theta})$ in the proposition is called a \textit{good position braid representative} of $(\tilde{w},\underline{\Theta})$. In particular, if $\underline{\Theta}$ is the increasing complete sequence in $(0,\pi]\cup \{2\pi\}$ for $\tilde{w}$, then $b(\tilde{w},\underline{\Theta})$ is called a \textit{good position braid representative} of $\tilde{w}$, denoted as $b(\tilde{w})$.
\end{definition}

\begin{remark} 
If we replace the increasing complete sequence in $(0,\pi] \cup \{2\pi\}$ by any increasing admissible sequence $\underline{\Theta}$ in $\Gamma_{\tilde{w}}$ such that $(\tilde{w},\underline{\Theta})$ is a good position pair, we can obtain many different braid representatives.
\end{remark}

By \cite[Lemma 5.1]{HN}, for any $W$-conjugacy class $\tilde{\CC} \in [\tilde{W}]$ and any sequence $\underline{\Theta}$ in $\Gamma_{\tilde{w}}^{[0,\pi]}$, there exists $\tilde{w} \in \tilde{\CC}$ such that $(\tilde{w},\underline{\Theta})$ is a good position pair. Since the proof of this result actually does not require anything on $\underline{\Theta}$, it holds for any $\underline{\Theta}$ in $\Gamma_{\tilde{w}}$. Now for any $\tilde{\CC} \in [\tilde{W}]$, it admits a good position braid representative $b(\tilde{w})$ of some $\tilde{w}\in \tilde{\CC}$, we call $b(\tilde{w})$ a \textit{good position braid representative} of $\tilde{\CC}$. This element $\tilde{w}$ here is not unique in general. Moreover, it is not even unique up to inverse in general. However, all good position braid representatives of a conjugacy class will have the same length. Let $l$ be the length function of $\tilde{W}$. Extend it to a length function of $B^+(\tilde{W})$ naturally and we still denote it as $l$. 

\begin{lemma}\label{3.3}
   Keep the notations in proposition \ref{a}. We have
   \begin{equation*}
       l(b(\tilde{w},\underline{\Theta}))=  \sum\limits_{j=1}^l \frac{\theta_{i_j}}{\pi} \lvert  \fkH_{F_{i_{j-1}}}-\fkH_{F_{i_{j}}}   \rvert.
   \end{equation*}
   Moreover, the right hand side will be the same if we replace $\tilde{w}$ by any other $\tilde{w}'\in \tilde{\CC}$, even when $(\tilde{w}',\underline{\Theta})$ is not a good position pair.
\end{lemma}   
\begin{proof}
    The equation is direct from proposition \ref{a}. It suffices to prove the moreover part. First we notice that applying a twist will not change the right hand side. Let $\tilde{w}'=u\tilde{w}u^{-1}$ for some $u\in W$ and write $\underline{\Theta}=(\theta_1,\ldots,\theta_k)$. For any $1\leq i\leq k$, let $F_i'=\sum\limits_{j=1}^i V_{\tilde{w}'}^{\theta_j}$.  Let $r(\underline{\Theta})'$ be the irredundant subsequence corresponding to $(\tilde{w}',\underline{\Theta})$. For any $\theta$ in $\Gamma_{\tilde{w}}=\Gamma_{\tilde{w}'}$ and $v\in V_{\tilde{w}}^{\theta}$, we have $\tilde{w}'(u(v))+\tilde{w}'^{-1}(u(v))=u\tilde{w}(v)+u\tilde{w}^{-1}(v)=2\cos \theta u(v)$. Then $V_{\tilde{w}'}^{\theta}=u (V_{\tilde{w}}^{\theta})$. This tells that $F_i'=u(F_i)$ and $r(\underline{\Theta})'=r(\underline{\Theta})$. Thus $(\alpha,F_i)=0$ is equivalent to $(u\alpha, F_i')=0$ for any root $\alpha\in R$. Therefore, we have $\lvert  \fkH_{F_{i_{j-1}}}-\fkH_{F_{i_{j}}}   \rvert=\lvert  \fkH_{F_{i_{j-1}}'}-\fkH_{F_{i_{j}}'}   \rvert$ and the right hand side are the same.
\end{proof}

\subsection{Proof of proposition \ref{a}} We follow the idea of the proof of \cite[Theorem 5.3]{HN}. We prove by induction on $\left| W \right|$. Assume that this theorem holds for any $\tilde{W'}$ with $\left| W' \right|<\left| W \right|$. We first reduce $\underline{\Theta}$ to its irredundant subsequence.

\begin{lemma}\label{3.4}
Suppose that $(\tilde{w},\underline{\Theta})$ is a good position pair, then $(\tilde{w},r(\underline{\Theta}))$ is also a good position pair. 
\end{lemma}

\begin{proof}
It suffices to show that $\overline{C_0}\cap (\sum\limits_{j=1}^k V_{\tilde{w}}^{\theta_{i_j}})^{\reg}$ is nonempty for all $1\leq k\leq l$. We prove by induction. This is clearly true when $k=1$. 

Suppose that the statement is true for $k$, then there exists $v=v_1+\cdots +v_k\in \overline{C_0}\cap (\sum\limits_{j=1}^k V_{\tilde{w}}^{\theta_{i_j}})^{\reg}$ where $v_j \in V_{\tilde{w}}^{\theta_{i_j}}$. Take $u\in \overline{C_0}\cap F_{i_{k+1}}^{\reg}$ and let $v_{k+1}$ be its projection onto $V_{\tilde{w}}^{\theta_{i_{k+1}}}$. We set $v'=v+\epsilon v_{k+1}\in \sum\limits_{j=1}^{k+1} V_{\tilde{w}}^{\theta_{i_j}}$ where $\epsilon>0$ is sufficiently small. For any positive root $\alpha$ such that $H_{\alpha} \not\supset \sum\limits_{j=1}^{k} V_{\tilde{w}}^{\theta_{i_j}}$, we have $(\alpha,v')>0$ since $\epsilon$ is sufficiently small. For any positive root $\alpha$ such that $H_{\alpha} \supset \sum\limits_{j=1}^{k} V_{\tilde{w}}^{\theta_{i_j}}$ and $H_{\alpha} \not\supset \sum\limits_{j=1}^{k+1} V_{\tilde{w}}^{\theta_{i_j}}$, we have $(\alpha,v')=\epsilon(\alpha,v_{k+1})>0$. Therefore the statement is true for $k+1$. 
\end{proof}

Since $(\tilde{w},\underline{\Theta})$ is a good position pair, then by definition $\overline{C_0}\cap F_{i_1}^{\reg}$ is nonempty. By \cite[Proposition 2.2]{HN}, we have $\tilde{w}=u_1\tilde{u} $ where $u_1 \in W_{F_{i_1}}$ and $\tilde{u}\in  W\sigma$ is a minimal double coset representative in $W_{I(F_{i_1})} \backslash \tilde{W} / W_{I(F_{i_1})}$ preserving the simple reflections in $I(F_{i_1})$. Therefore, the conjugation action of $\tilde{u}$ on $W_{F_{i_1}}$ is actually a twist and we denote it by $\delta_{\tilde{u}}$.

Under our identification, let $V'\subset V$ be the subspace generated by $\alpha\in R$ with $H_\alpha\supset F_{i_1}$. Then $V'$ is a real reflection representation of $\tilde{W}_{F_{i_1}}=W_{F_{i_1}}\rtimes \langle \delta_{\tilde{u}} \rangle$ and $C_0\cap V'$ is its fundamental Weyl chamber. Set $\underline{\Theta}_1=(\theta_{i_2},\ldots,\theta_{i_l})$. Since $C_0$ is in good position with respect to $(\tilde{w},r(\underline{\Theta}))$ by lemma \ref{3.4}, then $C_0\cap V'$ is in good position with respect to $( u_1\delta_{\tilde{u}},\underline{\Theta}_1)$. 

Now focus on $\tilde{W}_{F_{i_1}}$, we have $\underline{\Theta}_1$ is increasing and admissible with $r(\underline{\Theta}_1)=\underline{\Theta}_1$ by our proof of lemma \ref{3.4}. By induction hypothesis, there exists an element $b( u_1,\underline{\Theta}_1)$ in $B^+(\tilde{W}_{F_{i_1}})\subset B^+(\tilde{W})$ satisfying that $\pi_{\tilde{W}}(b(u_1,\underline{\Theta}_1))=u_1$ and
\begin{equation*}
    (b( u_1,\underline{\Theta}_1)\delta_{\tilde{u}} )^d=\underline{w_1}^{\frac{d\theta_{i_2}}{\pi}}\underline{w_2}^{\frac{d(\theta_{i_3}-\theta_{i_2})}{\pi}}\cdots \underline{w_{l-1}}^{\frac{d(\theta_{i_l}-\theta_{i_{l-1}})}{\pi}}(\delta_{\tilde{u}})^{d}.
\end{equation*}

First assume that $\theta_{i_1}=0$ or $\frac{2p\pi}{q}\in (0,2\pi)$ for some $p,q\in \BN_{>0}$ with $(p,q)=1$. Clearly $q\vert d$ and $q>1$. If $\theta_{i_1}=0$ then we choose $\tilde{u}'=\tilde{u}$ which is equal to $\sigma$ by \cite[Lemma 5.2]{HN}. If $\theta_{i_1}>0$, let $t,s\in \BN$ be such that $pt=sq+1$. Let $v_0\in \overline{C_0}$ be a regular point of $V_{\tilde{w}}^{\theta_{i_1}}$. Since $\tilde{u}^q$ fixes $v_0$ and the connected component of $V-\bigcup\limits_{H\in \fkH_{F_{i_1}}}H$ containing $C_0$, we have $\tilde{u}^q=\sigma^q$. Consider the element $\tilde{u}':= \tilde{u}^t  \in \tilde{W}$. For any $v\in V_{\tilde{w}}^{\theta_{i_1}}$ and any simple reflection $s_{\alpha}\in W_{F_{i_1}}$, we have $s_{\alpha}(v)=v$ since $H_{\alpha}\supset V_{\tilde{w}}^{\theta_{i_1}}$. Notice that $W_{F_{i_1}}<W$ is a standard parabolic subgroup by \cite[\S 2.1]{HN}, then any element in $W_{F_{i_1}}$ fixes $v$. This tells that $V_{\tilde{w}}^{\theta_{i_1}}=V_{\tilde{u}}^{\theta_{i_1}}=V_{\tilde{u}^t}^{\frac{2\pi}{q}}$. Clearly $\frac{2\pi}{q}\in (0,\pi)$. By \cite[Lemma 5.2]{HN}, we have 
\begin{equation*}
    (\underline{\tilde{u}'})^q=(\underline{w_1w_0} \cdot \underline{w_0w_1})\sigma^{tq}.
\end{equation*}
Consider $\tilde{b}=b(u_1,\underline{\Theta}_1)(\underline{\tilde{u}'})^p\sigma^{-sq}  \in B^+(\tilde{W})$. One can see that $(\underline{\tilde{u}'})^p\sigma^{-sq} $ acts on $W_{F_{i_1}}$ as $\tilde{u}^{pt}\sigma^{-sq}=\tilde{u}$, which is exactly $\delta_{\tilde{u}}$. Therefore, by induction hypothesis we have
\begin{eqnarray*}   
   \tilde{b}^d&=&b(u_1,\underline{\Theta}_1)\delta_{\tilde{u}}(b(u_1,\underline{\Theta}_1))\cdots \delta_{\tilde{u}}^{d-1}(b(u_1,\underline{\Theta}_1))((\underline{\tilde{u}'})^p\sigma^{-sq})^d\\   
   &=&  \underline{w_1}^{\frac{d\theta_{i_2}}{\pi}}\underline{w_2}^{\frac{d(\theta_{i_3}-\theta_{i_2})}{\pi}}\cdots \underline{w_{l-1}}^{\frac{d(\theta_{i_l}-\theta_{i_{l-1}})}{\pi}} ((\underline{\tilde{u}'})^p\sigma^{-sq})^d  \\
     &=& \underline{w_1}^{\frac{d\theta_{i_2}}{\pi}}\underline{w_2}^{\frac{d(\theta_{i_3}-\theta_{i_2})}{\pi}}\cdots \underline{w_{l-1}}^{\frac{d(\theta_{i_l}-\theta_{i_{l-1}})}{\pi}} (\underline{\tilde{u}'})^{pd} \sigma^{-sqd} \\
      &=& (\underline{w_1w_0} \cdot \underline{w_ow_1})^{\frac{d\theta_{i_1}}{2\pi}} \underline{w_1}^{\frac{d\theta_{i_2}}{\pi}}\underline{w_2}^{\frac{d(\theta_{i_3}-\theta_{i_2})}{\pi}}\cdots \underline{w_{l-1}}^{\frac{d(\theta_{i_l}-\theta_{i_{l-1}})}{\pi}} \sigma^{d} \\
   &=& \underline{w_0}^{\frac{d\theta_{i_1}}{\pi}}  \underline{w_1}^{\frac{d(\theta_{i_2}-\theta_{i_1})}{\pi}}  \cdots \underline{w_{l-1}}^{\frac{d(\theta_{i_l}-\theta_{i_{l-1}})}{\pi}}\sigma^d.
\end{eqnarray*}
Here the second equation is because $\sigma^{-sq}=\tilde{u}^{-sq}$. Meanwhile, one can check that $\pi_{\tilde{W}}(\tilde{b})=u_1 \tilde{u}^{pt}\sigma^{-sq}=u_1\tilde{u}=\tilde{w}$.

Finally let $\theta_{i_1}$ be arbitrary. There exists $k\in \BN$ such that $\theta_{i_1}'=\theta_{i_1}-2k\pi\in [0,2\pi)$. By previous discussion, there exists $b(u_1,\underline{\Theta}_1)\cdot \tilde{b}'\in B^+(\tilde{W})$ such that $\pi_{\tilde{W}}(\tilde{b}')=\tilde{u}$ and 
\begin{equation*}
    (b(u_1,\underline{\Theta}_1)\cdot \tilde{b}')^d=\underline{w_0}^{\frac{d\theta_{i_1}}{\pi}}  \underline{w_1}^{\frac{d(\theta_{i_2}-\theta_{i_1})}{\pi}}  \cdots \underline{w_{l-1}}^{\frac{d(\theta_{i_l}-\theta_{i_{l-1}})}{\pi}}\sigma^d .
\end{equation*}
Now we consider $\tilde{b}:=b(u_1,\underline{\Theta}_1) \cdot (\underline{w_1w_0} \cdot \underline{w_0w_1})^{k}\cdot\tilde{b}'$. Clearly we have $\pi_{\tilde{W}}(\tilde{b})=\tilde{w}$.
Notice that $\underline{w_1w_0} \cdot \underline{w_0w_1}$ commutes with $W_{F_{i_1}}$, then we have
\begin{eqnarray*}   
   \tilde{b}^d&=&  \underline{w_1}^{\frac{d\theta_{i_2}}{\pi}}\underline{w_2}^{\frac{d(\theta_{i_3}-\theta_{i_2})}{\pi}}\cdots \underline{w_{l-1}}^{\frac{d(\theta_{i_l}-\theta_{i_{l-1}})}{\pi}}( (\underline{w_1w_0} \cdot \underline{w_0w_1})^{k}\cdot \tilde{b}')^d   \\
     &=& \underline{w_1}^{\frac{d\theta_{i_2}}{\pi}}\underline{w_2}^{\frac{d(\theta_{i_3}-\theta_{i_2})}{\pi}}\cdots \underline{w_{l-1}}^{\frac{d(\theta_{i_l}-\theta_{i_{l-1}})}{\pi}} (\underline{w_1w_0} \cdot \underline{w_0w_1})^{kd} (\tilde{b}')^{d}\\
      &=& (\underline{w_1w_0} \cdot \underline{w_0w_1})^{\frac{d\theta_{i_1}}{2\pi}} \underline{w_1}^{\frac{d\theta_{i_2}}{\pi}}\underline{w_2}^{\frac{d(\theta_{i_3}-\theta_{i_2})}{\pi}}\cdots \underline{w_{l-1}}^{\frac{d(\theta_{i_l}-\theta_{i_{l-1}})}{\pi}} \sigma^{d} \\
   &=& \underline{w_0}^{\frac{d\theta_{i_1}}{\pi}}  \underline{w_1}^{\frac{d(\theta_{i_2}-\theta_{i_1})}{\pi}}  \cdots \underline{w_{l-1}}^{\frac{d(\theta_{i_l}-\theta_{i_{l-1}})}{\pi}}\sigma^d.
\end{eqnarray*}
Here the second equation is because of our choice of $\tilde{b}'$ in previous discussion.

One can prove the furthermore part in the same way.
\qed

\subsection{Relation with Lusztig's map and the strata of reductive groups}

Recall in \cite[Proposition 5.4]{HN} and \cite[Corollary 5.5]{HN}, it is proved that every good position element $\tilde{w}$ of an elliptic conjugacy class $\tilde{\CC}\in [\tilde{W}]$ is a minimal length element in $\tilde{\CC}$ and $\underline{\tilde{w}}^d\in \underline{w_0}^2 B^+(\tilde{W})$. However, for non-elliptic conjugacy classes, minimal length elements do not satisfy the second condition anymore. Therefore, we consider good position braid representatives for non-elliptic conjugacy classes instead. In this subsection we will see these representatives are related to the Lusztig's map $\Phi$.

Let $G$ be the identity component of an affine algebraic group $\tilde{G}$ over an algebraically closed field $\Bk$ and assume $G$ is reductive. Fix a connected component $D$ of $\tilde{G}$ whose image in $\tilde{G}/G$ is unipotent. Fix a Borel subgroup $B$ of $G$ with $T\subset B$ a maximal torus. Let $W$ be the Weyl group of $G$ and $\tilde{W}=W\rtimes \langle \delta \rangle$ where the twist $\delta:W \to W$ is defined by $D$. 

Let $[W\delta]$ be the set of $W$-conjugacy classes of $\tilde{W}$ contained in $W\delta$ and $\Phi:[W\delta]\to [D_u]$ be the corresponding Lusztig's map.  Recall that $\Phi$ is actually surjective and Lusztig defined a canonical section $\Psi:[D_u]\to [W\delta]$ in \cite[Theorem 0.2]{Lu2} and \cite[Theorem 1.16]{Lu21} of $\Phi$. The image $\CC:=\Psi(\CO)$ is called the ``most elliptic'' conjugacy class in $\Phi^{-1}(\CO)$. Here ${\CC}$ is the unique conjugacy class in $[W\delta]$ such that $\dim T^{\tilde{w}} <\dim T^{\tilde{w}'}$ for any $\tilde{w}\in \CC$ and $\tilde{w}'\in \CC'$, where $\CC'$ is any other conjugacy class in $\Phi^{-1}(\CO)$ (see \cite[Proposition 2.2]{CC}). We will use this criterion for our calculation of $\Psi$ in \S 3.4. For convenience, we set $\dim T^{\CC}:=\dim T^{\tilde{w}}$ for some (or any) $\tilde{w} \in \CC$. One can see that $\dim T^{\CC} = 0$ if $\CC$ is elliptic.

By lemma \ref{3.3}, all good position braid representatives of a conjugacy class $\CC$ have the same length, we denote it by $l_{\good}(\CC)$. We claim the following result which relates $l_{\good}(\CC)$ of some $\CC=\Psi(\CO)$ to the codimension of $\CO$.

\begin{proposition}\label{3.5}
   Let $\CO$ be a unipotent orbit in $D$ and $\CC=\Psi(\CO)$. Then we have 
\begin{equation*}
        \codim_{\tilde{G}} \CO=l_{\good}(\CC)+\dim T^{\CC}.
\end{equation*}   
\end{proposition}

Using the strata of (disconnected) reductive groups, we can actually rephrase proposition \ref{3.5}. Since all $G$-conjugacy classes corresponding to a given stratum $\Sigma\in \Str(\tilde{G})$ have the same dimension independent of characteristics, we can set $\codim_{\tilde{G}} \Sigma=\codim_{{\tilde{G}}^{(r)}} \CO^{(r)}$ for any (or all) $ \CO^{(r)}$ corresponding to $\Sigma$. We then have the following equivalent version of proposition \ref{3.5}.

\begin{proposition}\label{3.6}
    Let $\Sigma$ be a stratum of $\tilde{G}$. Suppose that $\CC$ is the most elliptic conjugacy class in $('\Phi)^{-1}(\Sigma)$, then we have 
\begin{equation*}
        \codim_{\tilde{G}} \Sigma=l_{\good}(\CC)+\dim T^{\CC}.
\end{equation*}   
\end{proposition}

\subsection{Proof of proposition \ref{3.5} and \ref{3.6}} 
We start from the reduction procedure. Since all parts of this equation are the same up to isogeny, we only need to prove the statement when $G$ is of the form $T' \times \prod\limits_{i} G_i$ where $T'$ is a torus and all $G_i$ are almost simple. Moreover, one may assume $T'$ is trivial since this does not change any side of the equation. Now suppose that $G=\prod\limits_{i} G_i$ and $W=\prod\limits_{i} W_i$. Since both sides of the equation are just the sum of its value on each almost simple factors, we only need to show the statement when $G$ is almost simple. As for the connected component $D$, we may assume that $\tilde{G}/G$ is cyclic.

Now let $G$ be an almost simple algebraic group of classical types and $\tilde{G}/G$ be cyclic. By \cite[\S 2.5]{Lu4} and the fact that $[D_u]$ is empty if $\character (\Bk)\neq 2$, it suffices to prove proposition \ref{3.5} when $\character (\Bk)=2$. We keep the notations in \S 2.2 and then prove this type by type. The formulas we will use to calculate $\codim_{\tilde{G}} \CO$ are based on \cite[\S3-\S6]{LS} and \cite[\S \RNum{2}.6]{Sp} (See also \cite{Hes}). 

\begin{itemize}
   \item Type $\text{A}_n$ and $D=G$ (nontwisted): In this case the unipotent $G$-orbits in $G$ are parametrized by partitions of $n+1$. For any $\lambda=(\lambda_1, \ldots,\lambda_m) \in \CP(n+1)$, we have $\CC_{\lambda}= \Psi(\CO_{\lambda})$. We may rewrite $\lambda$ as $\lambda=(r_1^{m_1},r_2^{m_2},\ldots, r_l^{m_l})$. Here $r_1>\cdots >r_l$ with $m_{\lambda}(r_j)=m_j$ and $m_{\lambda}(r)=0$ for any other integer $r$. 
   
   First we compute $\codim_{\tilde{G}}\CO_{\lambda}$. By \cite[Theorem 3.1]{LS}, we have 
   \begin{equation*}
      \codim_{\tilde{G}}\CO_{\lambda}=\sum\limits_{j=1}^l m_jr_j(m_j+2\sum\limits_{k=1}^{j-1}m_k)-1.
    \end{equation*}   
   Since this case is nontwisted, we can choose $\tilde{w}=w$ as follows. Identify $W$ with the permutation group of $\{1,2,\ldots,n+1\}$. For any $1\leq j\leq l$, let $\{\sum\limits_{k=1}^{j-1}m_kr_k+1,\ldots,\sum\limits_{k=1}^{j-1}m_kr_k+r_j\},\ldots,\{\sum\limits_{k=1}^{j-1}m_kr_k+ (m_j-1)r_j+1,\ldots,\sum\limits_{k=1}^j m_kr_k\}$ be all $r_j$-cycles of $w$. By lemma \ref{3.3}, we may use this $w$ to compute $l_{\good}(\CC_{\lambda})$.
   
   The corresponding irredundant sequence $r(\underline{\Theta})$ will then be $(\frac{2\pi}{r_1},\ldots,\frac{2\pi}{r_{l-1}})$ if $r_l=m_l=1$ and $(\frac{2\pi}{r_1},\ldots,\frac{2\pi}{r_l})$ otherwise. By lemma \ref{3.3}, we have 
   \begin{equation*}
   l_{\good}(\CC_{\lambda})=2 \sum\limits_{j=1}^l \frac{1}{r_j} \lvert  \fkH_{F_{m_{j-1}}}-\fkH_{F_{m_{j}}}   \rvert,
   \end{equation*}
   where we formally set $m_0=0$. By our choice of $w$, we have $W_{F_{m_j}}$ is a finite Coxeter group of type $\text{A}$ with rank $n-\sum\limits_{k=1}^j m_kr_k$, then
   \begin{eqnarray*}
      \lvert  \fkH_{F_{m_{j-1}}}-\fkH_{F_{m_{j}}}   \rvert &=& \frac{1}{2}((n-\sum\limits_{k=1}^{j-1} m_kr_k)(n-\sum\limits_{k=1}^{j-1} m_kr_k +1)\\
      &-& (n-\sum\limits_{k=1}^j m_kr_k)(n-\sum\limits_{k=1}^j m_kr_k+1)) \\
      &=& \frac{1}{2} (2m_j r_j n-2m_jr_j\sum\limits_{k=1}^{j-1} m_kr_k -(m_jr_j)^2 +m_jr_j).
   \end{eqnarray*}
   Therefore, we have 
   \begin{equation*}
      l_{\good}(\CC_{\lambda})= \sum\limits_{j=1}^l (2m_j  n-2m_j\sum\limits_{k=1}^{j-1} m_kr_k -m_j^2r_j +m_j).
      \end{equation*}
   Finally, noticing that $\dim T^{\CC_{\lambda}}=m-1=\sum\limits_{j=1}^l m_j -1$, we have 
   \begin{eqnarray*}
      l_{\good}(\CC_{\lambda})+\dim T^{\CC_{\lambda}} &=& \sum\limits_{j=1}^l (2m_j n-2m_j\sum\limits_{k=1}^{j-1} m_kr_k -m_j^2r_j +m_j)+(\sum\limits_{j=1}^l m_j -1) \\
      &=& \sum\limits_{j=1}^l (2m_j \sum\limits_{k=j+1}^l m_kr_k +m_j^2r_j )-1 \\
      &=& \sum\limits_{k=1}^l 2m_k \sum\limits_{j=1}^{k-1} m_jr_j +\sum\limits_{j=1}^l m_j^2r_j -1 \\
      &=& \codim_{\tilde{G}}\CO_{\lambda}.
   \end{eqnarray*}
Thus proposition \ref{3.5} holds for type $\text{A}_n$.

   \item Type $\text{A}_n$ and $D\neq G$ (Type ${}^2\text{A}_n$, twisted): In this case, we have $\tilde{G}=G\sqcup D$ and the unipotent $G$-orbits in $D$ are parametrized by $\widetilde{\CP_1}(n+1)$. The twisted Weyl group $\tilde{W}$ is equal to $W\sqcup W\delta$. Let $\CO_{(\nu,\epsilon_{\nu})}^D$ be a unipotent $G$-orbit in $D$. Let $\lambda\in \CP(n+1)$ be such that $\CC^{W\delta}_{\lambda}=\Psi(\CO_{(\nu,\epsilon_{\nu})}^D)$. The partition $\lambda$ is uniquely determined as follows:
\newline
(i) If $k$ is even, then $m_{\lambda}(2k)=\frac{1}{2}m_{\nu}(k)$;
\newline
(ii) If $k$ is odd and $\epsilon_{\nu}(k)=1$, then $m_{\lambda}(k)=m_{\nu}(k)$ and $m_{\lambda}(2k)=0$;
\newline
(iii) If $m$ is odd and $\epsilon_{\nu}(k)=0$, then $m_{\lambda}(k)=0$ and $m_{\lambda}(2k)=\frac{1}{2}m_{\nu}(k)$;
\newline
(iv) If $k$ is odd and $m_{\nu}(k)=0$, then $m_{\lambda}(k)=m_{\lambda}(2k)=0$.

   For calculation, we give another parametrization of conjugacy classes in $[W\delta]$. Recall \S 2.2, the partition $\lambda$ corresponding to the conjugacy class $\CC_{\lambda}^{W\delta}$ is the cycle type of $ww_0$ for any $\tilde{w}=w\delta\in \CC_{\lambda}^{W\delta}$, under the natural identification of $W$ with the permutation group of $\{1,2,\ldots,n+1\}$. One can actually extend the identification to $\tilde{W}$ with a subgroup of the permutation group of $\{\pm 1,\pm 2,\ldots,\pm (n+1)\}$ by viewing $\delta$ as $\delta:\pm i\mapsto \mp (n+2-i)$. Thus the conjugacy classes of $[\tilde{W}]$ are parametrized by signed cycles as mentioned in \S 2.2. 
   
   Specifically, for $\lambda$, we first associate it to a bipartition $(\lambda_{\text{odd}},\lambda_{\text{even}})\in \mathcal{BP}(n+1)$. Let $\lambda_{\text{odd}}=(\lambda_1^{\text{o}},\ldots,\lambda_a^{\text{o}})$ be the subsequence of $\lambda$ consisting of all odd parts and $\lambda_{\text{even}}=(\lambda_1^{\text{e}},\ldots,\lambda_b^{\text{e}})$ be the subsequence of $\lambda$ consisting of all even parts. Here $\lambda_{\text{odd}}$ (resp. $\lambda_{\text{even}}$) corresponds to all negative (resp. positive) cycles of $\CC_{\lambda}^{W\delta}$ in its signed cycle type. We then obtain a partition $\mu$ by combining $(2\lambda_1^{\text{o}},\ldots,2\lambda_a^{\text{o}})$ and $(\lambda_1^{\text{e}},\lambda_1^{\text{e}},\ldots,\lambda_b^{\text{e}},\lambda_b^{\text{e}})$ and rearranging them. Write $\mu=(\mu_1,\ldots,\mu_m)$. By our construction of $\mu$, all parts of $\mu$ are even. One can check that $\nu=(\frac{\mu_1}{2},\ldots,\frac{\mu_m}{2})$.
   
   We proceed by induction on the rank $n$. Assume the statement is true when rank of $G$ is less than $n$. We have $\dim T^{\CC_{\lambda}^{W\delta}}=b$, which is the number of positive cycles. Let $\{e_j-e_{j+1}\mid j\in I\}$ be the natural basis of $V$ such that $\tilde{w}(e_j)=e_{\tilde{w}(j)}$ under our identification above. Here we set $e_{-j}=-e_j$. Choose the simple roots as usual to be $\{e_1-e_2,\ldots,e_n-e_{n+1}\}$ and set $I'=\{1,2,\ldots,n+1\}$. Let $r(\underline{\Theta})=(\theta_{i_1},\ldots,\theta_{i_l})$ be the irredundant sequence corresponding to $\CC_{\lambda}^{W\delta}$.
   We divide the proof into two different cases based on $\mu_1$ and $\epsilon_{\nu}$.

   (1) We first consider the case when $\frac{\mu_1}{2}$ is even. Then the longest signed cycle of $\CC_{\lambda}^{W\delta}$ is a positive $\mu_1$-cycle. Set $m_1=m_{\mu}(\mu_1)$. We have $m_1$ is also equal to $2m_{\lambda}(\mu_1)$ and there are $\frac{m_1}{2}$ positive $\mu_1$-cycles. Therefore, we have $\theta_{i_1}=\frac{2\pi}{\mu_1}$.

   By lemma \ref{3.3}, choose an arbitrary $\tilde{w}\in \CC_{\lambda}^{W\delta}$. Let $I_1'\subset I'$ be corresponding to all positive $\mu_1$-cycles of $\tilde{w}$. Let 
   \begin{center}
   $j_1 \xrightarrow{\tilde{w}} j_2\xrightarrow{\tilde{w}} \cdots \xrightarrow{\tilde{w}} j_{\mu_1}, -j_1 \xrightarrow{\tilde{w}} -j_2\xrightarrow{\tilde{w}} \cdots \xrightarrow{\tilde{w}} -j_{\mu_1}$    
   \end{center}
   be a positive $\mu_1$-cycle of $\tilde{w}$ where $j_1,\ldots,j_{\mu_1}\in \pm I_1'$. By our identification of $\delta$, assume that $j_1\in I_1'$, we have $(-1)^{k+1}j_k\in I_1'$ for all $k$.  One can then check that $\sum\limits_{k=1}^{\mu_1-1}\sin \frac{2k\pi}{\mu_1}e_{j_k}$ and $\sum\limits_{k=1}^{\mu_1-1}\sin \frac{2k\pi}{\mu_1}e_{j_{k+1}}$ both lie in $ V_{\tilde{w}}^{\frac{2\pi}{\mu_1}}\setminus \{0\}$. This is because $\sum\limits_{k=1}^{\mu_1-1}(-1)^k\sin \frac{2k\pi}{\mu_1}=0$ as $\frac{\mu_1}{2}$ is even. We also have similar results for all other positive $\mu_1$-cycles of $\tilde{w}$ and thus gives a basis for $V_{\tilde{w}}^{\frac{2\pi}{\mu_1}}$. 
   
   We have $\fkH_{F_{i_1}}=\{H_{\alpha}\mid (\alpha,e_i)=0,\forall i\in I_1'\}$. Then $W_{F_{i_1}}<W$ is the standard parabolic subgroup of the same type as $W$ with rank $n-\frac{m_1\mu_1}{2}\geq 1$ if $W_{F_{i_1}}$ is non-trivial. If $W_{F_{i_1}}$ is trivial, one can easily check our statement is true.

   Let $G_1$ be almost simple of the same type as $G$ corresponding to Weyl group $W_1:=W_{F_{i_1}}$, which means its rank is $n-\frac{m_1\mu_1}{2}$. Let $\tilde{G}_1$ be the disconnected affine algebraic group with $G_1$ as its identity component and $\delta_1$ be the restriction of $\delta$ and $D_1$ be the corresponding component of $\tilde{G}_1$. Let $T_1$ be a maximal torus of $G_1$. Set $\tilde{W}_1=W_1\rtimes \langle \delta_1 \rangle$. Let $\Phi_1:[\tilde{W}_1]\to [(\tilde{G}_1)_u]$ be the Lusztig's map.
   
   Let $\nu'=(\nu_{m_1+1},\ldots, \nu_m)=(\frac{\mu_{m_1+1}}{2},\ldots,\frac{\mu_m}{2})\in \CP_1(n-\frac{m_1\mu_1}{2}+1)$ and $\epsilon_{\nu'}$ be the restriction of $\epsilon_{\nu}$. Let $\lambda'_{\text{even}}$ be the subpartition of $\lambda_{\text{even}}$ obtained by deleting the first $\frac{m_1}{2}$ terms and $\lambda'_{\text{odd}}$ be exactly $\lambda_{\text{odd}}$. Then we get $\lambda'\in \CP(n-\frac{m_1\mu_1}{2}+1)$ by combining $\lambda'_{\text{even}}$ and $\lambda'_{\text{odd}}$ together. It is clear that $\CC_{\lambda'}^{W_1\delta_1}$ is most elliptic in $\Phi_1^{-1}(\CO_{(\nu',\epsilon_{\nu'})}^{D_1})$. By induction hypothesis, we have
   \begin{equation*}
\codim_{\tilde{G}_1}\CO_{(\nu',\epsilon_{\nu'})}^{D_1}=l_{\text{good}}(\CC_{\lambda'}^{W_1\delta_1})+\dim T_1^{\CC_{\lambda'}^{W_1\delta_1}},
   \end{equation*}
   and $\dim T_1^{\CC_{\lambda'}^{W_1\delta_1}}=b-\frac{m_1}{2}$. Therefore, it suffices to show
   \begin{equation*}
       \codim_{\tilde{G}}\CO_{(\nu,\epsilon_{\nu})}^D-\codim_{\tilde{G}_1}\CO_{(\nu',\epsilon_{\nu'})}^{D_1}=l_{\text{good}}(\CC_{\lambda}^{W\delta})-l_{\text{good}}(\CC_{\lambda'}^{W_1\delta_1})+\frac{m_1}{2}.
   \end{equation*}
By our discussion above on $V_{\tilde{w}}^{\frac{2\pi}{\mu_1}}$ and lemma \ref{3.3}, we have
\begin{equation*}
    l_{\text{good}}(\CC_{\lambda}^{W\delta})-l_{\text{good}}(\CC_{\lambda'}^{W_1\delta_1})=m_1n-\frac{m_1}{2}(\frac{m_1\mu_1}{2}-1).
\end{equation*}
Then we need $\codim_{\tilde{G}}\CO_{(\nu,\epsilon_{\nu})}^D-\codim_{\tilde{G}_1}\CO_{(\nu',\epsilon_{\nu'})}^{D_1}=m_1n-\frac{m_1^2\mu_1}{4}+m_1$.

Let $\nu^*$ and $\nu'^*$ be the dual partition of $\nu$ and $\nu'$ respectively. It is clear that $\nu_i^*-\nu_i'^*=m_1$ for any $i$ such that $\nu_i^*>0$. We also notice $\nu_{i+1}^*-\nu_i^*=\nu_{i+1}'^*-\nu_i'^*$ for all $i$. Then by \cite[\S \RNum{2}.6.3]{Sp}, we have 
\begin{eqnarray*}
    &&\codim_{\tilde{G}}\CO_{(\nu,\epsilon_{\nu})}^D-\codim_{\tilde{G}_1}\CO_{(\nu',\epsilon_{\nu'})}^{D_1}\\&=&\frac{1}{2}(\sum\limits_{i}((\nu_i^*)^2-(\nu_i'^*)^2))-\sum\limits_{i \text{ odd}}(\nu_i^*-\nu_i'^*)+\frac{m_1\mu_1}{4}\\
    &=& \frac{1}{2}m_1\sum_{i}(\nu_i^*+\nu_i'^*)-m_1[\frac{\nu_1}{2}]+\frac{m_1\mu_1}{4}\\
    &=& \frac{1}{2}m_1(n+1+n-\frac{m_1\mu_1}{2}+1)-m_1\frac{\mu_1}{4}+\frac{m_1\mu_1}{4}\\
    &=& m_1n-\frac{m_1^2\mu_1}{4}+m_1.
\end{eqnarray*}
Thus the statement is true in this case.

(2) Then consider the case when $\frac{\mu_1}{2}$ is odd and $\epsilon_{\nu}(\frac{\mu_1}{2})=1$. Then the longest signed cycle of $\CC_{\lambda}^{W\delta}$ must be a negative $\frac{\mu_1}{2}$-cycle. Moreover, there is no positive $\mu_1$-cycle since $\CC_{\lambda}^{W\delta}=\Psi(\CO_{(\nu,\epsilon_{\nu})})$. Set $m_1=m_{\mu}(\mu_1)$. Then there are $m_1$ negative $\frac{\mu_1}{2}$-cycles. We again have $\theta_{i_1}=\frac{2\pi}{\mu_1}$.  If $\mu_1=2$, then $\nu=\lambda=(1,\ldots,1)\in \CP(n+1)$, $\epsilon_{\nu}(1)=1$ and $r(\underline{\Theta})=(\frac{2\pi}{\mu_1})$. Thus $l_{\text{good}}(\CC_{\lambda}^{W\delta})=\frac{n(n+1)}{2}$ and $\dim T^{\CC_{\lambda}^{W\delta}}=0$. By \cite[\S \RNum{2}.6.3]{Sp}, we have $\codim_{\tilde{G}}\CO^D_{(\nu,\epsilon_{\nu})}=\frac{(n+1)^2}{2}-(n+1)+\frac{n+1}{2}=\frac{n(n+1)}{2}$ and the statement is true.

Now assume $\mu_1>2$. Choose an arbitrary $\tilde{w}\in \CC_{\lambda}^{W\delta}$. Let $I_1'\subset I'$ be corresponding to all negative $\frac{\mu_1}{2}$-cycles of $\tilde{w}$. Let
\begin{center}
    $j_1 \xrightarrow{\tilde{w}} j_2\xrightarrow{\tilde{w}} \cdots \xrightarrow{\tilde{w}} j_{\frac{\mu_1}{2}} \xrightarrow{\tilde{w}} -j_1\xrightarrow{\tilde{w}} \cdots \xrightarrow{\tilde{w}} -j_{\frac{\mu_1}{2}}$
\end{center}
 be a negative $\frac{\mu_1}{2}$-cycle of $\tilde{w}$ where $j_1,\ldots,j_{\frac{\mu_1}{2}}\in \pm I_1'$. Again one can check that $\sum\limits_{k=1}^{\frac{\mu_1}{2}-1}\sin \frac{2k\pi}{\mu_1}e_{j_k}$ and $\sum\limits_{k=1}^{\frac{\mu_1}{2}-1}\sin \frac{2k\pi}{\mu_1}e_{j_{k+1}}$ both lie in $V_{\tilde{w}}^{\frac{2\pi}{\mu_1}}\setminus \{0\}$ if $\mu_1>2$.

As above case, we have $\fkH_{F_{i_1}}=\{H_{\alpha}\mid (\alpha,e_i)=0,\forall i\in I_1'\}$. Then $W_{F_{i_1}}<W$ is the standard parabolic subgroup of the same type as $W$ with rank $n-\frac{m_1\mu_1}{2}$. We still only need to consider the case when $W_{F_{i_1}}$ is nontrivial.

As previous case, we get $\tilde{G}_1, G_1, D_1, T_1, W_1, \delta_1, \tilde{W}_1$ and $\Phi_1$ in the same way. Similarly define $(\nu',\epsilon_{\nu'})$. Let $\lambda'_{\text{odd}}$ be the subpartition of $\lambda_{\text{odd}}$ obtained by deleting the first $m_1$ terms and $\lambda'_{\text{even}}$ be exactly $\lambda_{\text{even}}$. Then we get $\lambda'\in \CP(n-\frac{m_1\mu_1}{2}+1)$ by combining $\lambda'_{\text{odd}}$ and $\lambda'_{\text{even}}$ together. It is clear that $\CC_{\lambda'}^{W_1\delta_1}$ is most elliptic in $\Phi_1^{-1}(\CO_{(\nu',\epsilon_{\nu'})}^{D_1})$. By induction hypothesis, we have
\begin{equation*}
    \codim_{\tilde{G}_1}\CO_{(\nu',\epsilon_{\nu'})}^{D_1}=l_{\text{good}}(\CC_{\lambda'}^{W_1\delta_1})+\dim T_1^{\CC_{\lambda'}^{W_1\delta_1}},
\end{equation*}
and $\dim T_1^{\CC_{\lambda'}^{W_1\delta_1}}=b$. Then it suffices to show
   \begin{equation*}
       \codim_{\tilde{G}}\CO_{(\nu,\epsilon_{\nu})}^D-\codim_{\tilde{G}_1}\CO_{(\nu',\epsilon_{\nu'})}^{D_1}=l_{\text{good}}(\CC_{\lambda}^{W\delta})-l_{\text{good}}(\CC_{\lambda'}^{W_1\delta_1}).
   \end{equation*}
As previous case, we again have
\begin{equation*}
    l_{\text{good}}(\CC_{\lambda}^{W\delta})-l_{\text{good}}(\CC_{\lambda'}^{W_1\delta_1})=m_1n-\frac{m_1}{2}(\frac{m_1\mu_1}{2}-1).
\end{equation*}

Still let $\nu^*$ and $\nu'^*$ be the dual partition of $\nu$ and $\nu'$ respectively. By \cite[\S \RNum{2}.6.3]{Sp}, since $\frac{\mu_1}{2}=\nu_1$ is odd, we have
\begin{eqnarray*}
        \codim_{\tilde{G}}\CO_{(\nu,\epsilon_{\nu})}^D-\codim_{\tilde{G}_1}\CO_{(\nu',\epsilon_{\nu'})}^{D_1}&=&\frac{1}{2}m_1\sum_{i}(\nu_i^*+\nu_i'^*)-m_1[\frac{\nu_1}{2}]+\frac{m_1\mu_1}{4}\\
    &=& m_1n-\frac{m_1^2\mu_1}{4}+\frac{m_1}{2}.
\end{eqnarray*}
Thus the statement is true in this case.

(3) Finally consider the case when $\frac{\mu_1}{2}$ is odd and $\epsilon_{\nu}(\frac{\mu_1}{2})=0$. Then the longest signed cycle of $\CC_{\lambda}^{W\delta}$ must be a positive $\mu_1$-cycle and there is no negative $\frac{\mu_1}{2}$-cycle. Set $m_1=m_{\mu}(\mu_1)$. Then $m_1=2m_{\lambda}(\mu_1)$ and there are $\frac{m_1}{2}$ positive $\mu_1$-cycles. Then we have $\theta_{i_1}=\frac{2\pi}{\mu_1}$. If $\mu_1=2$, then $\nu=(1,\ldots,1)\in \CP(n+1)$ and $\epsilon_{\nu}(1)=0$. We also have $\lambda=(2,\ldots,2)\in \CP(n+1)$ and $\delta\in \CC_{\lambda}^{W\delta}$. Thus $\dim T^{\CC_{\lambda}^{W\delta}}=\frac{n+1}{2}$. By \cite[\S \RNum{2}.6.3]{Sp}, we have $\codim_{\tilde{G}}\CO_{(\nu,\epsilon_{\nu})}^D=\frac{(n+1)^2}{2}-(n+1)+\frac{n+1}{2}+(n+1)=\frac{(n+1)(n+2)}{2}$. By applying the formular in lemma \ref{3.3} to $\delta$, we get $\lvert \fkH_{F_{i_1}}\rvert=m_1=\frac{n+1}{2}$. Then $l_{\text{good}}(\CC_{\lambda}^{W\delta})=\frac{n(n+1)-(n+1)}{2}+(n+1)=\frac{(n+1)^2}{2}$ and our statement is true.

Now assume $\mu_1>2$. Choose an arbitrary $\tilde{w}\in \CC_{\lambda}^{W\delta}$. Let $I_1'\subset I'$ be corresponding to all positive $\mu_1$-cycles of $\tilde{w}$. Let 
\begin{center}
    $\{j_{s,1} \xrightarrow{\tilde{w}} j_{s,2}\xrightarrow{\tilde{w}} \cdots \xrightarrow{\tilde{w}} j_{s,\mu_1}$, $-j_{s,1} \xrightarrow{\tilde{w}} -j_{s,2}\xrightarrow{\tilde{w}} \cdots \xrightarrow{\tilde{w}} -j_{s,\mu_1}\mid 1\leq s\leq \frac{m_1}{2}\}$
\end{center}
 be all positive $\mu_1$-cycle of $\tilde{w}$ where $j_{s,1},\ldots,j_{s,\mu_1}\in \pm I_1'$. Again one can check that $\sum\limits_{k=1}^{\mu_1-1} \sin \frac{2k\pi}{\mu_1}e_{j_{s,k}}$ and $\sum\limits_{k=1}^{\mu_1-1} \sin \frac{2k\pi}{\mu_1}e_{j_{s,k+1}}$ for $1\leq s\leq \frac{m_1}{2}$ all lie in $V_{\tilde{w}}^{\frac{2\pi}{\mu_1}}\setminus \{0\}$ if $\mu_1>2$. 

Unlike above two cases, in this case we have $\fkH_{F_{i_1}}$ is actually the disjoint union of $\{H_{\alpha} \mid (\alpha,e_i)=0,\forall i\in I_1'\}$ and $\{H_{\alpha} \mid \alpha=e_{\lvert j_{s,k} \rvert}-e_{\lvert j_{s,\frac{\mu_1}{2}+k} \rvert}, 1\leq s\leq \frac{m_1}{2}, 1\leq k\leq \frac{\mu_1}{2}\}$. Moreover, one may check that $\sum\limits_{k=1}^{\mu_1-1} \sin \frac{4k\pi}{\mu_1}e_{j_{s,k}}$ and $\sum\limits_{k=1}^{\mu_1-1} \sin \frac{4k\pi}{\mu_1}e_{j_{s,k+1}}$ for $1\leq s\leq \frac{m_1}{2}$ all lie in $V_{\tilde{w}}^{\frac{4\pi}{\mu_1}}\setminus \{0\}$ and $\{H_{\alpha} \mid \alpha=e_{\lvert j_{s,k} \rvert}-e_{\lvert j_{s,\frac{\mu_1}{2}+k} \rvert}, 1\leq s\leq \frac{m_1}{2}, 1\leq k\leq \frac{\mu_1}{2}\}$ does not contain the subspace spanned by these vectors.

Then $W_{F_{i_1}}<W$ is the product of $W_1$ and $W'$. Here $W_1$ is generated by the reflections corresponding to $\{H_{\alpha} \mid (\alpha,e_i)=0,\forall i\in I_1'\}$ and $W'$ is generated by the reflections corresponding to $\{e_{\lvert j_{s,k} \rvert}-e_{\lvert j_{s,\frac{\mu_1}{2}+k} \rvert}\mid 1\leq s\leq \frac{m_1}{2}, 1\leq k \leq \frac{\mu_1}{2}\}$. One can check that $W_1$ is of same type as $W$ with rank $n-{\frac{m_1\mu_1}{2}}$ and $W'$ is of type $\text{A}_1\times \cdots \times \text{A}_1$ with $\frac{m_1\mu_1}{4}$ copies of $\text{A}_1$. By above discussion on $V_{\tilde{w}}^{\frac{4\pi}{\mu_1}}\setminus \{0\}$, the root hyperplanes corresponding to $W'$ actually lie in $\fkH_{\theta_{i_j}}\setminus \fkH_{\theta_{i_{j+1}}}$ for $\theta_{i_j}=\frac{4\pi}{\mu_1}$.

Let $G_1$ be almost simple of same type as $G$
 corresponding to Weyl group $W_1$, which again make its rank $n-\frac{m_1\mu_1}{2}$. Then define $\tilde{G}_1,D_1,T_1,\delta_1,\tilde{W}_1$ and $\Phi_1$ as before. Similarly define $(\nu',\epsilon_{\nu'})$ and $\lambda'$ as in the first case. It is clear that $\CC_{\lambda'}^{W_1\delta_1}$ is most elliptic in $\Phi_1^{-1}(\CO^{D_1}_{(\nu',\epsilon_{\nu'})})$. By induction hypothesis, we have
\begin{equation*}
    \codim_{\tilde{G}_1}\CO_{(\nu',\epsilon_{\nu'})}^{D_1}=l_{\text{good}}(\CC_{\lambda'}^{W_1\delta_1})+\dim T_1^{\CC_{\lambda'}^{W_1\delta_1}},
\end{equation*}
and $\dim T_1^{\CC_{\lambda'}^{W_1\delta_1}}=b-\frac{m_1}{2}$. By our above discussion on $\fkH_{F_{i_1}}$ and $W'$, let $\fkH_1$ be the set of root hyperplanes corresponding to $W_1$, we have
\begin{eqnarray*}
        l_{\text{good}}(\CC_{\lambda}^{W\delta})-l_{\text{good}}(\CC_{\lambda'}^{W_1\delta_1})&=&\frac{2}{\mu_1}\lvert \fkH-\fkH_{F_{i_1}} \rvert+\frac{4}{\mu_1} \lvert \fkH_{F_{i_1}}-\fkH_1 \rvert\\
        &=& m_1n-\frac{m_1^2\mu_1}{4}+m_1.
\end{eqnarray*}

Still let $\nu^*$ and $\nu'^*$ be the dual partition of $\nu$ and $\nu'$ respectively. Let $\nu_{max}^*=\nu_i^*$ where $i$ is the largest number such that $\nu_i^*>0$. By \cite[\S \RNum{2}.6.3]{Sp}, we have
\begin{eqnarray*}
        \codim_{\tilde{G}}\CO_{(\nu,\epsilon_{\nu})}^D-\codim_{\tilde{G}_1}\CO_{(\nu',\epsilon_{\nu'})}^{D_1}&=&\frac{1}{2}m_1\sum_{i}(\nu_i^*+\nu_i'^*)-m_1[\frac{\nu_1}{2}]+\frac{m_1\mu_1}{4}+\nu_{max}^*\\
    &=& m_1n-\frac{m_1^2\mu_1}{4}+\frac{m_1}{2}+m_1.
\end{eqnarray*}
Thus the statement is true in this case.

By above three cases and induction, proposition \ref{3.5} holds for type ${}^2\text{A}_n$.
 
   \item Type $\text{B}_n,\text{C}_n,\text{D}_n$: In these cases, we label the simple roots as follows. 
   
   \begin{center}
      $\text{B}_n,\text{C}_n$:
      \begin{tikzpicture}[baseline=0]
      \node[] at (0,0.1) {$\circ$};
      \node[below] at (0,-0.09) {$n$};
      \draw[-] (0.08,0.11) to (0.93,0.11);
      \node[] at (1,0.1) {$\circ$};
      \node[below] at (1,0) {$n-1$};
      \draw[dashed] (1.08,0.11) to (2.92,0.11);
      \node[] at (2.96,0.1) {$\circ$};
      \node[below] at (2.96,0) {$2$};
      \draw[-] (3.02,0.14) to (3.89,0.14);
      \draw[-] (3.02,0.08) to (3.89,0.08);
      \node[] at (3.95,0.1) {$\circ$};
      \node[below] at (3.95,0) {$1$};
      \end{tikzpicture}
   \end{center}

   \begin{center}
      $\text{D}_n$: 
      \begin{tikzpicture}[baseline=0]
      \node[] at (0,0.1) {$\circ$};
      \node[below] at (0,-0.09) {$n$};
      \draw[-] (0.08,0.11) to (0.93,0.11);
      \node[] at (1,0.1) {$\circ$};
      \node[below] at (1,0) {$n-1$};
      \draw[dashed] (1.08,0.11) to (2.93,0.11);
      \node[] at (2.96,0.1) {$\circ$};
      \draw[-] (3.03,0.13) to (3.77,0.55);
      \node[] at (3.84,0.58) {$\circ$};
      \node[right] at (3.85,0.58) {$2$};
      \draw[-] (3.03,0.07) to (3.77,-0.35);
      \node[] at (3.84,-0.38) {$\circ$};
      \node[right] at (3.85,-0.38) {$1$};
      \end{tikzpicture}
 \end{center}
   
   First consider in type $\text{D}_n$, a unipotent $G$-orbit $\CO_{(\nu,\epsilon_{\nu})}$ splits into $(\CO_{(\nu,\epsilon_{\nu})})_{\Romannum{1}}$ and $(\CO_{(\nu,\epsilon_{\nu})})_{\Romannum{2}}$ if $\nu_j$ and $m_{\nu}(j)$ are all even and $\epsilon_{\nu}=0$. By \S2.2, we have $\Phi^{-1}((\CO_{(\nu,\epsilon_{\nu})})_{\Romannum{1}})$ and $\Phi^{-1}((\CO_{(\nu,\epsilon_{\nu})})_{\Romannum{2}})$ are the two conjugacy classes corresponding the same bipartition $(\lambda,\mu)$ where $\lambda$ is the zero partition and $\mu\in \CP(n)$ with $m_{\mu}(k)=\frac{1}{2} m_{\nu}(k)$ for all $k$. Since their corresponding good position braid representatives have the same length, one only need to prove our statement for one of these two classes. 
   
   We proceed by induction. Assume the statement is true if the rank of $G$ is less than $n$. Recall $\tilde{G}=G$ in type $\text{B}_n$ and type $\text{C}_n$, the twisted case is only considered in type $\text{D}_n$. We denote the two components of $\tilde{G}$ in type $\text{D}_n$ as $G$ and $D\neq G$. Let $(\nu,\epsilon_{\nu}) \in \widetilde{\mathcal{P}_{-1}}(2n)$ and $\CO_{(\nu,\epsilon_{\nu})}$ be the corresponding unipotent $G$-orbit in $G$ or $D$. In particular, for type $\text{D}_n$ we have $\CO_{(\nu,\epsilon_{\nu})}\in G$ if $\nu$ has even number of parts and $\CO_{(\nu,\epsilon_{\nu})}\in D$ if $\nu$ has odd number of parts. 
   
   Let $(\lambda,\mu)\in \mathcal{BP}(2n)$ such that $\CC_{(\lambda,\mu)}=\Psi(\CO_{(\nu,\epsilon_{\nu})})$.  Write $\nu=(\nu_1,\ldots,\nu_l)\in \CP_{-1}(2n)$, $\lambda=(\lambda_1,\ldots,\lambda_a)$ and $\mu=(\mu_1,\ldots,\mu_b)$. We have $\dim T^{\CC_{(\lambda,\mu)}}=b$ which is the number of positive cycles of $\CC_{(\lambda,\mu)}$. Let $\{e_j \mid j\in I\}$ be the natural basis of $V$ such that $\tilde{w}(e_j)=e_{\tilde{w}(j)}$ for all $\tilde{w}\in \tilde{W}$. Here we set $e_{-j}=-e_j$. Choose the simple roots to be $\{e_n-e_{n-1},\ldots, e_2-e_1,e_1\}$ in type $\text{B}_n$ and $\{e_n-e_{n-1},\ldots, e_2-e_1,2e_1\}$ in type $\text{C}_n$ and $\{e_n-e_{n-1},\ldots, e_2-e_1,e_2+e_1\}$ in type $\text{D}_n$. One notice that the twist $\delta$ in type $\text{D}_n$ satisfies $\delta(e_1)=e_{-1}$. Let $r(\underline{\Theta})$ be the irredundant sequence corresponding to  $\CC_{(\lambda,\mu)}$. We divide the proof into three different cases based on the size of the longest cycle of  $\CC_{(\lambda,\mu)}$. The ideas are the same for these cases. 
   
   (1) First consider the case when $\nu_1$ is odd. Then the longest signed cycle of $\CC_{(\lambda,\mu)}$ is a positive $\nu_1$-cycles. Then $m_{\nu}(\nu_1)=m_1$ is even and $m_{\mu}(\nu_1)=\frac{m_1}{2}$. There are $\frac{m_1}{2}$ positive $\nu_1$-cycles. Therefore,  we have $\theta_{i_1}=\frac{2\pi}{\nu_1}$. 

 Again choose an arbitrary $\tilde{w}\in \CC_{(\lambda,\mu)}$. Let $I_1\subset I$ be corresponding to all positive $\nu_1$-cycles of $\tilde{w}$. Let 
 \begin{center}
     $j_1 \xrightarrow{\tilde{w}} j_2\xrightarrow{\tilde{w}} \cdots \xrightarrow{\tilde{w}} j_{\nu_1}, -j_1 \xrightarrow{\tilde{w}} -j_2\xrightarrow{\tilde{w}} \cdots \xrightarrow{\tilde{w}} -j_{\nu_1}$ 
 \end{center}
 be a positive $\nu_1$-cycle where $j_1,\ldots,j_{\nu_1}\in \pm I_1$. Then again one can check that $\sum\limits_{k=1}^{\nu_1-1} \sin \frac{2k\pi}{\nu_1} e_{j_k},\sum\limits_{k=1}^{\nu_1-1} \sin \frac{2k\pi}{\nu_1} e_{j_{k+1}} \in V_{\tilde{w}}^{\frac{2\pi}{\nu_1}}\setminus\{0\}$ if $\nu_1>1$ and $e_{j_1}\in V_{\tilde{w}}^{2\pi}\setminus\{0\}$ if $\nu_1=1$. As before, we have $\fkH_{F_{i_1}}=\{H_{\alpha} \mid (\alpha, e_i)=0, \forall i \in I_1 \}$. Then $W_{F_{i_1}}<W$ is the standard parabolic subgroup of the same type as $W$ with rank $n-\frac{m_1\nu_1}{2}$.

   Let $G_1$ be almost simple of same type as $G$ corresponding to Weyl group $W_1:=W_{F_{i_1}}$, which means its rank is $n-\frac{m_1\nu_1}{2}$. Let $\tilde{G}_1$ be the affine algebraic group (disconnected in type $\text{D}$) with $G_1$ as its identity component. Let $\delta_1$ be the restriction of $\delta$ and $D_1$ be the corresponding component of $\tilde{G}_1$. Clearly $D_1=G_1$ and $\delta_1$ is trivial in type $\text{B}, \text{C}$. Let $T_1$ be a maximal torus of $G_1$. Set $\tilde{W}_1=W_1\rtimes \langle \delta_1 \rangle$. Let $\Phi_1:[\tilde{W}_1]\to [(\tilde{G}_1)_{u}]$ be the Lusztig's map. 
   
   Let $\nu'=(\nu_{m_1+1},\ldots,\nu_l)\in \CP_{-1}(2n)$ and $\epsilon_{\nu'}$ be the restriction of $\epsilon_{\nu}$. Set $\mu'=(\mu_{\frac{m_1}{2}+1},\ldots,\mu_b)$. It is clear that $\CC_{(\lambda,\mu')}$ is
    most elliptic in $\Phi_1^{-1}(\CO_{(\nu',\epsilon_{\nu'})})$. By induction hypothesis, we have
   \begin{equation*}
       \codim_{\tilde{G}_1}\CO_{(\nu',\epsilon_{\nu'})}   =l_{\good}(\CC_{(\lambda,\mu')})+\dim T_1^{\CC_{(\lambda,\mu')}},
   \end{equation*}
   and $\dim T_1^{\CC_{(\lambda,\mu')}}=b-\frac{m_1}{2}$. Therefore, it suffices to show
   \begin{equation*}
       \codim_{\tilde{G}}\CO_{(\nu,\epsilon)}- \codim_{\tilde{G}_1}\CO_{(\nu',\epsilon)}=l_{\good}(\CC_{(\lambda,\mu)})-l_{\good}(\CC_{(\lambda,\mu')})+\frac{m_1}{2}.
   \end{equation*}

   By lemma \ref{3.3}, we have 
   \begin{equation*}
     l_{\good}(\CC_{(\lambda,\mu)})-l_{\good}(\CC_{(\lambda,\mu')}) =
    \begin{cases}
      2m_1n-\frac{m_1^2\nu_1}{2} & \text{type $\text{B}_n,\text{C}_n$},\\
      2m_1n-\frac{m_1^2\nu_1}{2}-m_1 & \text{type $\text{D}_n$}.
    \end{cases}       
    \end{equation*}

   Meanwhile, by \cite[Theorem 4.4]{Hes} we have
    \begin{eqnarray*}
        \codim_{\tilde{G}}\CO_{(\nu,\epsilon)}- \codim_{\tilde{G}_1}\CO_{(\nu',\epsilon)}&=&\sum\limits_{i=1}^{m_1} i\nu_i+m_1\sum\limits_{i=m_1+1}^l \nu_i-\sum\limits_{i=1}^{m_1} \chi(\nu_i)\\
             &=& 2m_1n-\frac{m_1^2\nu_1}{2}+\frac{m_1\nu_1}{2}-\sum\limits_{i=1}^{m_1} \chi(\nu_i).
    \end{eqnarray*}

Here $\chi:\BN \to \BZ$ is defined in \cite[\S 3.5]{Hes} as $\chi(s) = \max\{0,\min\{n-s+t,t\}\}$ in rank $n$, where
\begin{equation*}
    t=
    \begin{cases}
      \frac{s}{2} & \text{if $G$ is of type $\text{B}, \text{C}$ and $\CO$ has a $V(s)$ block },\\
      \frac{s+2}{2} & \text{if $G$ is of type $\text{D}$ and $\CO$ has a $V(s)$ block}, \\
      [\frac{s-1}{2}] & \text{if $G$ is of type $\text{B}, \text{C}$ and $\CO$ has no $V(s)$ blocks}, \\
      [\frac{s+1}{2}] & \text{if $G$ is of type $\text{D}$ and $\CO$ has no $V(s)$ blocks}.
    \end{cases}
\end{equation*}

In our case, we have
\begin{equation*}
    \sum\limits_{i=1}^{m_1} \chi(\nu_i)=
    \begin{cases}
      \frac{m_1(\nu_1-1)}{2} & \text{if $G$ is of type $\text{B}, \text{C}$},\\
      \frac{m_1(\nu_1+1)}{2} & \text{if $G$ is of type $\text{D}$} .
    \end{cases}
\end{equation*}

By above equations, the statement is true if $\nu_1$ is odd.

(2) Then we consider the case when $\nu_1$ is even and $\epsilon(\nu_1)=1$. Then $\CO_{(\nu,\epsilon_{\nu})}$ has a $V(\nu_1)$ block and the longest signed cycle is a negative $\frac{\nu_1}{2}$-cycle. Since $\CC_{(\lambda,\mu)}=\Psi(\CO_{(\nu,\epsilon_{\nu})})$, there is no positive $\nu_1$-cycle. Set $m_{\nu}(\nu_1)=m_1$. We then have $m_{\lambda}(\frac{\nu_1}{2})=m_1$ and $\theta_{i_1}=\frac{2\pi}{\nu_1}$.

Choose an arbitrary $\tilde{w}\in \CC_{(\lambda,\mu)}$. Let $I_1\subset I$ be corresponding to all negative $\frac{\nu_1}{2}$-cycles of $\tilde{w}$. Let 
\begin{center}
    $j_1 \xrightarrow{\tilde{w}} j_2\xrightarrow{\tilde{w}} \cdots \xrightarrow{\tilde{w}} j_{\frac{\nu_1}{2}}\xrightarrow{\tilde{w}} -j_1 \xrightarrow{\tilde{w}} \cdots \xrightarrow{\tilde{w}} -j_{\frac{\nu_1}{2}}$
\end{center}
 be a negative $\frac{\nu_1}{2}$-cycle where $j_1,\ldots,j_{\frac{\nu_1}{2}}\in \pm I_1$. One can check that $e_{j_1}\in V_w^\pi\setminus \{0\}$ if $\nu_1=2$ and $\sum\limits_{k=1}^{\frac{\nu_1}{2}-1} \sin \frac{2k\pi}{\nu_1} e_{j_k}$, $\sum\limits_{k=1}^{\frac{\nu_1}{2}-1} \sin \frac{2k\pi}{\nu_1} e_{j_{k+1}}\in V_{w}^{\frac{2\pi}{\nu_1}}\setminus \{0\}$ if $\nu_1>2$. As before, we have $\fkH_{F_{i_1}}=\{H_{\alpha} \mid (\alpha,e_i)=0, \forall i\in I_1  \}$. Then $W_{F_{i_1}}<W$ is the standard parabolic subgroup of the same type as $W$ with rank $n-\frac{m_1\nu_1}{2}$.

Define $\tilde{G}_1, G_1, D_1, T_1, W_1, \delta_1, \tilde{W}_1$ and $\Phi_1$ in the same way as previous case. Similarly define $(\nu',\epsilon_{\nu'})$. Set $\lambda'=(\lambda_{m+1},\ldots,\lambda_a)$. We have $\CC_{(\lambda',\mu)}$ is most elliptic in $\Phi_1^{-1}(\CO_{(\nu',\epsilon_{\nu'})})$. By induction hypothesis, we have
\begin{equation*}
    \codim_{\tilde{G}_1}\CO_{(\nu',\epsilon_{\nu'})}   =l_{\good}(\CC_{(\lambda',\mu)})+\dim T_1^{\CC_{(\lambda',\mu)}},
\end{equation*}
and $\dim T_1^{\CC_{(\lambda,\mu')}}=\dim T ^{\CC_{(\lambda,\mu)}}$. Therefore, it suffices to show
   \begin{equation*}
       \codim_{\tilde{G}}\CO_{(\nu,\epsilon_{\nu})}- \codim_{\tilde{G}_1}\CO_{(\nu',\epsilon_{\nu'})}=l_{\good}(\CC_{(\lambda,\mu)})-l_{\good}(\CC_{(\lambda',\mu)}).
   \end{equation*}
   
   By lemma \ref{3.3}, we have 
   \begin{equation*}
     l_{\good}(\CC_{(\lambda,\mu)})-l_{\good}(\CC_{(\lambda,\mu')}) =2m_1n-\frac{m_1^2\nu_1}{2}  . 
    \end{equation*}

    By \cite[Theorem 4.4]{Hes} and the description of $\chi$ above, we have 
    \begin{eqnarray*}
     \codim_{\tilde{G}}\CO_{(\nu,\epsilon_{\nu})}- \codim_{\tilde{G}_1}\CO_{(\nu',\epsilon_{\nu'})}&=& \sum\limits_{i=1}^{m_1} i\nu_i+m_1\sum\limits_{i=m_1+1}^l \nu_i-\sum\limits_{i=1}^{m_1} \chi(\nu_i)\\
     &=& 2m_1n-\frac{m_1^2\nu_1}{2}+\frac{m_1\mu_1}{2}-\sum\limits_{i=1}^{m_1} \chi(\nu_i)\\
     &=&       2m_1n-\frac{m_1^2\nu_1}{2}.
    \end{eqnarray*}
    Thus the statement is true in this case.

    (3) Finally consider the case when $\nu_1$ is even and $\epsilon(\nu_1)=0$. The longest signed cycle is a positive $\nu_1$-cycle and there are no negative $\frac{\nu_1}{2}$-cycles. Set $m_{\nu}(\nu_1)=m_1$. Then $m_{\mu}(\nu_1)=\frac{m_1}{2}$ and there are $\frac{m_1}{2}$ positive $\nu_1$-cycles. Then we have $\theta_{i_1}=\frac{2\pi}{\nu_1}$.

   Choose an arbitrary $\tilde{w}\in \CC_{(\lambda,\mu)}$. Let 
   \begin{center}
          $\{j_{s,1}\xrightarrow{\tilde{w}} j_{s,2}\xrightarrow{\tilde{w}} \cdots \xrightarrow{\tilde{w}} j_{s,\nu_1} , -j_{s,1}\xrightarrow{\tilde{w}}  \cdots \xrightarrow{\tilde{w}} -j_{s,\nu_1}\mid 1 \leq s \leq \frac{m_1}{2} \}$
   \end{center}
 be all positive $\nu_1$-cycles of $\tilde{w}$. Again $\sum\limits_{k=1}^{\nu_1-1} \sin \frac{2k\pi}{\nu_1} e_{j_{s,k}}, \sum\limits_{k=1}^{\nu_1-1} \sin \frac{2k\pi}{\nu_1} e_{j_{s,k+1}}$ for $1\leq s\leq \frac{m_1}{2}$ all lie in $ V_{\tilde{w}}^{\frac{2\pi}{\nu_1}}$ and$\sum\limits_{k=1}^{\nu_1-1} \sin \frac{4k\pi}{\nu_1}e_{j_{s,k}}$ and  $\sum\limits_{k=1}^{\nu_1-1} \sin \frac{4k\pi}{\nu_1}e_{j_{s,k+1}}$ for $1\leq s\leq \frac{m_1}{2}$ all lie in $V_{\tilde{w}}^{\frac{4\pi}{\nu_1}}$ if $\nu_1>2$. If $\nu_1=2$, then $\{e_{j_{s,1}}-e_{j_{s,2}}\mid 1\leq s\leq \frac{m_1}{2}\}$ is a basis of $V_{\tilde{w}}^{\pi}$ and $\{e_{j_{s,1}}+e_{j_{s,2}}\mid 1\leq s\leq \frac{m_1}{2}\}$ is contained in $V_{\tilde{w}}^{2\pi}$. 
 
 Therefore, we have $W_{F_{i_1}}<W$ is the product of $W_1$ and $W'$. Here $W_1$ is of the same type as $W$ with rank $n-\frac{m_1\nu_1}{2}$ and $W'$ is generated by the reflections corresponding to $\{e_{j_{s,k}}+e_{j_{s,k+r}} \mid 1\leq s \leq \frac{m_1}{2}, 1\leq k \leq \frac{\nu_1}{2}\}$. 

    Let $G_1$ be almost simple of same type as $G$ corresponding to Weyl group $W_1$. Then define $\tilde{G}_1, D_1, T_1, \delta_1, \tilde{W}_1$ and $\Phi_1$ as before. Similarly define $(\nu',\epsilon_{\nu'})$ and $\mu'$ as in the first case. Then $\CC_{(\lambda,\mu')}$ is most elliptic in $\Phi_1^{-1}(\CO_{(\nu',\epsilon_{\nu'})})$. By induction hypothesis, we have
    \begin{equation*}
      \codim_{\tilde{G}_1}\CO_{(\nu',\epsilon_{\nu'})}   =l_{\good}(\CC_{(\lambda,\mu')})+\dim T_1^{\CC_{(\lambda,\mu')}},
   \end{equation*}
   and $\dim T_1^{\CC_{(\lambda,\mu')}}=b-\frac{m_1}{2}$.
    
   Let $\fkH_1$ be the set of root hyperplanes corresponding to $W_1$. By our discussion above, we have 
   \begin{eqnarray*}
      l_{\good}(\CC_{(\lambda,\mu)})-l_{\good}(\CC_{(\lambda,\mu')}) &=& \frac{2}{\nu_1} \lvert \fkH-\fkH_{F_{i_1}} \rvert + \frac{4}{\nu_1} \lvert \fkH_{F_{i_1}}-\fkH_1 \rvert \\
      &=&
     \begin{cases}
       2m_1n-\frac{m_1^2\nu_1}{2}+\frac{m_1}{2} & \text{type $\text{B}_n,\text{C}_n$},\\
       2m_1n-\frac{m_1^2\nu_1}{2}-\frac{m_1}{2} & \text{type $\text{D}_n$}.
     \end{cases}       
     \end{eqnarray*}
     By \cite[Theorem 4.4]{Hes} and the description of $\chi$ above, we have 
     \begin{eqnarray*}
      \codim_{\tilde{G}}\CO_{(\nu,\epsilon_{\nu})}- \codim_{\tilde{G}_1}\CO_{(\nu',\epsilon_{\nu'})}&=& \sum\limits_{i=1}^{m_1} i\nu_i+m_1\sum\limits_{i=m_1+1}^l \nu_i-\sum\limits_{i=1}^{m_1} \chi(\nu_i)\\
      &=& 2m_1n-\frac{m_1^2\nu_1}{2}+\frac{m_1\nu_1}{2}-\sum\limits_{i=1}^{m_1} \chi(\nu_i)\\
      &=& \begin{cases}
       2m_1n-\frac{m_1^2\nu_1}{2}+m_1 & \text{type $\text{B}_n,\text{C}_n$},\\
       2m_1n-\frac{m_1^2\nu_1}{2} & \text{type $\text{D}_n$}.
     \end{cases}
     \end{eqnarray*}
 Thus the statement is true in this case.

By above three cases and induction, proposition \ref{3.5} holds for type $\text{B}_n, \text{C}_n, \text{D}_n$.

\item Combining all discussion above, proposition \ref{3.5} (or proposition \ref{3.6}) is proved for classical groups. As for exceptional groups, we verify this theorem by computer using GAP-CHEVIE\cite{Mi}. The main idea is as follows.
\begin{itemize}
\item 1. Let $W$ be the Weyl group of $G$ and $p$ be the characteristic of $\Bk$.
    \item 2. Starting from any unipotent orbit $\CO$, we can apply 
    \begin{center}
        \texttt{gap> FormatTeX(UnipotentClasses(W,p));}
    \end{center}
    The result will be a table of the information of all unipotent orbits of $G$ in characteristic $p$ case. In the table one can find the dimension of the corresponding Springer fiber $\CB_u:=\{B'\in \CB\mid u\in B'\}$. Here $\CB$ is the variety of Borels.
    \item 3. Following the steps in \cite[\S 6]{Mi}, one can define the function for Lusztig's map $\Phi$. Then we can apply 
    \begin{center}
        \texttt{gap> LusztigMap(W,p);}
    \end{center}
    Get the preimage of the unipotent orbit $\CO$ as a list of conjugacy classes of $W$.
    \item 4. We apply 
    \begin{center}
        \texttt{gap> ChevieClassInfo(W);} \\
        \texttt{gap> ReflectionEigenvalues(W);}
    \end{center}
    The first function gives a table of all conjugacy classes of $W$. The second function gives a list of eigenvalues of all conjugacy classes of $W$. Using this two functions, we can compute the dimension of $T^{\CC}$ for any $\CC\in [W]$. This also help us identify the ``most elliptic'' conjugacy class $\Psi(\CO)$ in $\Phi^{-1}(\CO)\subset [W]$.
    \item 5. Finally, we just need to compute the length of good position braid representatives of $\Psi(\CO)$. We defined a new function to compute this length. The algorithm is pretty simple. For a given conjugacy class of $W$, one can first find all its eigenvalues using the function in step 4. Then given any root and its corresponding root hyperplane, we can figure out its role in lemma \ref{3.3} by checking its orthogonality with above eigenspaces. Finally, we get the length by lemma \ref{3.3}. Combining all steps above we can verify the proposition for exceptional types.
\end{itemize}
We refer to \cite{Dcode} for the code with step-by-step explanation.
\qed
\end{itemize}

\subsection{Indecomposably-good position braid representatives}
 Recall that in \S 2.1, the real eigenspaces $V_{\tilde{w}}^{\theta}$ are not indecomposable in general. An indecomposable real eigenspace $V'$ is a subspace of $V_{\tilde{w}}^{\theta}$ such that $V'$ is closed under the action of $\tilde{w}$ and any proper subspace of $V'$ does not have such a property. It is not hard to see that $\dim V'=1$ if $\cos \theta = \pm 1$ and $\dim V' = 2$ otherwise. In this subsection, we use a sequence of indecomposable real eigenspaces to replace $\underline{\Theta}$. 

Let $\tilde{w}\in \tilde{W}$ and $\CV=(V_1,\ldots, V_k)$ be a sequence of distinct indecomposable real eigenspaces of $\tilde{w}$ in $V$. Define $F_i=\sum\limits_{j=1}^i V_i$ for $1\leq i\leq k$ and $F_0=0$. The definition of being admissible and complete for $\CV$ is the same as in \S 2.1. For any Weyl chamber $C$, we say $C$ is $\textit{in good position}$ with respect to $(\tilde{w},{\CV})$ if $\overline{C}\cap F_i^{\reg}\neq \emptyset$ for all $i$. We call $(\tilde{w},{\CV})$ an \textit{indecomposably-good position pair} if $C$ is the fundamental chamber $C_0$.

For any $i$, by definition $V_{i}$ lies in some real eigenspace $V_{\tilde{w}}^\theta$ of $\tilde{w}$. We know the choice of $\theta$ is not unique. Let $\underline{\Theta}_{\CV}=(\theta_1,\ldots,\theta_k)$ be a sequence such that $V_i\subset V_{\tilde{w}}^{\theta_i}$. As in \S 2.1, one can again define the irredundant subsequnce $r({\CV})=(V_{i_1},\ldots,V_{i_l})$ of ${\CV}$. Set $\underline{\Theta}_{r(\CV)}=(\theta_{i_1},\ldots,\theta_{i_l})$. We call $(\tilde{w},{\CV},\underline{\Theta}_{\CV})$ an \textit{indecomposably-good position triple} if the fundamental chamber $C_0$ is in good position with respect to $(\tilde{w},{\CV})$, or equivalently $(\tilde{w},r({\CV}))$. Note that here $\theta_i$ and $\theta_{i'}$ might be the same. Now if we have $\theta_1\leq \theta_2\leq \cdots\leq \theta_l$, then we say $\underline{\Theta}_{\CV}$ is \textit{weakly-increasing}. One should also note that we can have infinitely many choices of $\underline{\Theta}_{\CV}$ for a given ${\CV}$. Therefore, we have infinitely many choices of indecomposably-good position triples for a given indecomposably-good position pair.

We have following analogy to proposition \ref{a}.

\begin{proposition}\label{3.7}
    Let $(\tilde{w},{\CV},\underline{\Theta}_{\CV})$ be an indecomposably-good position triple where ${\CV}$ is admissible with $r({\CV})=(V_{i_1},\ldots,V_{i_l})$ and $\underline{\Theta}_{r(\CV)}=(\theta_{i_1},\ldots,\theta_{i_l})$ is weakly-increasing. Suppose that $\tilde{w}= w\sigma$ with $\sigma\in \langle \delta \rangle$ and $w\in W$. Then there exists an element $b(\tilde{w},{\CV},\underline{\Theta}_{\CV})= b(w,{\CV},\underline{\Theta}_{\CV})\sigma \in B^+(\tilde{W})$ satisfying that $\pi_{\tilde{W}}(b(\tilde{w},{\CV},\underline{\Theta}_{\CV}))=\tilde{w}$ and
\begin{equation*}
   b(\tilde{w},{\CV},\underline{\Theta}_{\CV})^d=\underline{w_0}^\frac{d\theta_{i_1}}{\pi} \underline{w_1}^\frac{d(\theta_{i_2}-\theta_{i_1})}{\pi} \cdots \underline{w_{l-1}}^\frac{d(\theta_{i_l}-\theta_{i_{l-1}})}{\pi}\sigma^d .
\end{equation*}   
Here $d\in \BN$ such that $d\frac{\theta_{i_j}}{2\pi}\in \BZ$, and $w_j$ is the longest element in $W_{F_{i_j}}$ for $0\leq j \leq l$. Furthermore, if $d$ is even, then 
\begin{equation*}
   b(\tilde{w},{\CV},\underline{\Theta}_{\CV})^{\frac{d}{2}}= \underline{w_0}^\frac{d\theta_{i_1}}{2\pi} \underline{w_1}^\frac{d(\theta_{i_2}-\theta_{i_1})}{2\pi} \cdots \underline{w_{l-1}}^\frac{d(\theta_{i_l}-\theta_{i_{l-1}})}{2\pi}\sigma^{\frac{d}{2}}.
\end{equation*}       
\end{proposition}
\begin{proof}
    The proof of this theorem is almost exactly the same as the proof of proposition \ref{a}. One just have to notice that \cite[Lemma 2.2, Lemma 5.2]{HN} also work, if we replace $V_{\tilde{w}}^{\theta}$ by any $\tilde{w}$-stable subspace $K\subset V_{\tilde{w}}^{\theta}$ satisfying that $\overline{C_0}$ contains a regular point of $K$ and $C_0,\tilde{w}(C_0)$ are in the same connected component of $V-\bigcup\limits_{H\in \fkH_{K}}H$. Since the indecomposable real eigenspace $V_{i_1}$ in our sequence $r(\CV)$ satisfies this condition, by induction the statement is true.
\end{proof}

\begin{definition}\label{3.8}
    The braid element $b(\tilde{w},{\CV},\underline{\Theta}_{\CV})$ in the above theorem is called an \textit{indecomposably-good position braid representative} of the triple $(\tilde{w},{\CV},\underline{\Theta}_{\CV})$. 
\end{definition}

Again we directly have the similar formula for the length of these braid representatives as lemma \ref{3.3}.
\begin{corollary}\label{3.9}
    Keep the notations in proposition \ref{3.7}. We have
    \begin{equation*}
        l(b(\tilde{w},{\CV},\underline{\Theta}_{\CV}))=  \sum\limits_{j=1}^l \frac{\theta_{i_j}}{\pi} \lvert  \fkH_{F_{i_{j-1}}}-\fkH_{F_{i_{j}}}   \rvert.
    \end{equation*}
\end{corollary}

The following result compare these new braid representatives with good position braid representatives defined in \S 3.1. We shall see good position braid representatives are all indecomposably-good.

\begin{proposition}\label{3.10}
     Let $(\tilde{w},\underline{\Theta})$ be a good position pair. Then there exists a sequence ${\CV}$ of indecomposable real eigenspaces of $\tilde{w}$ with $\underline{\Theta}_{r(\CV)}=r(\underline{\Theta})$ such that $(\tilde{w},{\CV},\underline{\Theta}_{\CV})$ is an indecomposably-good position triple. Furthermore, if $b(\tilde{w},\underline{\Theta})$ be a good position braid representative of some $(\tilde{w},\underline{\Theta})$, then $b(\tilde{w},\underline{\Theta})$ is an indecomposably-good position braid representative of $(\tilde{w},{\CV},\underline{\Theta}_{\CV})$.
\end{proposition}
\begin{proof}
    Let $\underline{\Theta}=(\theta_1,\ldots,\theta_k)$ and $F_i:=\sum\limits_{j=1}^i V_{\tilde{w}}^{\theta_j}$ for $1\leq i\leq k$. Let $r(\underline{\Theta})=(\theta_{i_1},\ldots,\theta_{i_l})$. By our assumption and lemma \ref{3.4}, we have $(\tilde{w},r(\underline{\Theta}))$ is a good position pair. Then for any $1\leq j\leq l$, there exists a regular point $v_j'$ of $ \sum\limits_{r=1}^j V_{\tilde{w}}^{\theta_{i_r}}$ lying in $ \overline{C_0}$. It is also a regular point of $F_{i_j}$. Let $v_j$ be its projection to $V_{\tilde{w}}^{\theta_{i_j}}$. Let $\alpha$ be any root such that $H_\alpha\in \fkH_{F_{i_{j-1}}}\setminus \fkH_{F_{i_j}}$. Then we must have $(v_j,\alpha)>0$. Now we choose $V_j=\BR v_j+\BR \tilde{w}(v_j)$. Clearly $V_j\subset V_{\tilde{w}}^{\theta_{i_j}}$ is an indecomposable real eigenspace.

    Let ${\CV}=(V_1,\ldots,V_l)$. Clearly we can choose $\underline{\Theta}_{\CV}=\underline{\Theta}_{r(\CV)}=r(\underline{\Theta})$. Write $F_{{\CV},j}=\sum\limits_{r=1}^j  V_r$. Then for any $1\leq j\leq l$, there exists $\sum\limits_{r=1}^j \lambda_rv_r\in F_{{\CV},j}$ lying in $F_{i_j}^{\reg}\cap \overline{C_0}$ and $\fkH_{F_{i_j}}=\fkH_{F_{{\CV},j}}$ by our choice of $v_j$ above. Therefore, we have $(\tilde{w},{\CV},\underline{\Theta}_{\CV})$ is an indecomposably good position triple. 
    
    Furthermore part is direct by definition.
\end{proof}
Based on above results, we see indecomposably-good position braid representative is a more general form. We will naturally relate them to affine Springer fibers in \S 5.

We end this section by the following result on the existence of indecomposably-good position triples (or pairs since $\underline{\Theta}$ is chosen manually).
\begin{lemma}\label{3.11}
    Let $\CV=(V_1,\ldots,V_l)$ be a sequence of indecomposable real eigenspaces of $\tilde{w}$. Suppose that $W_{F_l}$ is trivial where $F_l:=\sum\limits_{i=1}^l V_i$, then there is an admissible (or complete) sequence $\CV'$ of indecomposable real eigenspaces of $\tilde{w}$ whose irredundant subsequence is a subsequence of $\CV$. 
\end{lemma}

\begin{proof}
    It is clear that one can also complete $\CV$ to a complete sequence $\CV'$ by adding indecomposable real eigenspaces after $V_l$. The statement on irredundant subsequence is direct by definition and the assumption that $W_{F_l}$ is trivial.
\end{proof}

\section{Transversal slices from good representatives}
In this section, let $G$ be a connected reductive algebraic group over an algebraically closed field $\Bk$. We consider the following two cases (see \cite[\S 0.2]{Lu3}).
\begin{itemize}
     \item case \RNum{1}: $G$ is the identity component of a disconnected reductive group $\tilde{G}$ with a fixed connected component $D$ (no requirements on $\text{char}(\Bk)$);
    \item case \RNum{2}: $\Bk=\overline{\BF_q}$ for some finite field $\overline{\BF_q}$ and $\tilde{G}=G$ with Frobenius map $F:G\to G$.
\end{itemize}
Now fix a Borel subgroup $B\subset G$ which is $F$-stable in case \RNum{2} and let $T\subset B$ be a maximal torus of $G$. In case \RNum{1}, let $\tilde{g}_D\in D$ be such that $\tilde{g}_D T\tilde{g}_D^{-1}=T$ and $\tilde{g}_D B\tilde{g}_D^{-1}=B$. Identify $N_G(T)/T$ with $W$.

In the first case, the conjugation of component $D$ defines a twist $\delta_D:W\to W$. This induces $\delta_D:G\to G$ which send the 1-parameter root subgroup $U_{\alpha}$ to $U_{\delta_D(\alpha)}$. In the second case, the Frobenius map $F$ also define a twist $\delta_F:W\to W$ which send the 1-parameter root subgroup $U_{\alpha}$ to $U_{\delta_F(\alpha)}$. We assume that $\tilde{g}_D x\tilde{g}_D^{-1}=\delta_D(x)$ in case \RNum{1} and $F(x)=\delta_F\chi(x)$ in case \RNum{2} where $\chi(x)=x^q$ for all $x\in \Bk$.

The above two cases both lead to some twisted Weyl groups. For convenience, we will denote the twists $\delta_D$ and $\delta_F$ of $W$ both by $\delta$ and the resulting twisted Weyl groups by $\tilde{W}$. We also denote the corresponding automorphism $\delta_D$ and $F$ of $G$ both by $\delta_G$. We will generalize the construction of transversal slices in \cite{HL} using good position braid representatives of non-elliptic $W$-conjugacy classes in $[\tilde{W}]$. The construction of our slices is related to the structures of Lusztig varieties and Deligne-Lusztig varieties attached to these braid elements in my other paper (The proofs in this paper do not rely on that paper). 

\begin{remark}
    We will construct slices in both cases while they may not be transversal in case \RNum{2}. However, these slices still have some nice properties in case \RNum{2}.
\end{remark}

\subsection{Slices for braids}
We start from constructing slices. For any $\tilde{b}=b\delta\in B^+(\tilde{W})$, recall the following canonical decomposition of $\tilde{b}$.
\begin{definition}\cite{DDGKM}
    Let $\tilde{b}=b\delta\in B^+(\tilde{W})$ be any twisted braid. There is a unique \textit{right Deligne-Garside normal form} of $\tilde{b}$
    \begin{equation*}
        \tilde{b}=b_m\cdots b_1\delta.
    \end{equation*}
    Here $b_i$ is the largest right common divisor of $b_m\cdots b_i$ and $\underline{w_0}$ where $w_0$ is the longest element in $W$.
\end{definition}
Now we write the left Deligne-Garside normal form of $\tilde{b}$ as $\tilde{b}=\underline{w_m}\cdots \underline{w_1}\delta$ where $w_i\in W$. (Here $w_i$ are not the same as in proposition \ref{a}). For any $w\in W$, define its inversion set as $\Inv(w)=\{\alpha\in R^+ \mid w(\alpha)\in -R^+\}$. For any $R'\subset R$, define $U_{R'}$ as the subgroup of $G$ generated by root subgroups corresponding to the roots of $R'$ (This may not be a subgroup of $U$ unless $R'$ satisfies some conditions). One can check $U_{\Inv(w^{-1})}=U\cap \dot{w}U^-\dot{w}^{-1}=U^w$ is a subgrouop of $U$. We can naturally define a collection of subsets of roots as follows. 
\begin{itemize}
    \item $R_m:=R^+\cap w_m(R^-)=\Inv(w_m^{-1})$;
    \item $R_{m-1}:=w_m(R^+) \cap (w_mw_{m-1})(R^-)=w_m(\Inv(w_{m-1}^{-1}))$;
    \item $R_{m-2}:=(w_{m}w_{m-1})(\Inv(w_{m-2}^{-1}))$;
    \item \dots
    \item $R_1:=(w_m\cdots w_2)(\Inv(w_1^{-1}))$.
\end{itemize}
For any $1\leq i\leq m$, by above we have $U_{R_i}=\dot{w}_m\cdots \dot{w}_{i+1} U^{w_i}\dot{w}_{i+1}^{-1}\cdots \dot{w}_m^{-1}<G$. By \cite[\S 2.7]{HL} we have $U_{R_i}$ is isomorphic to affine space $\BA ^{l(w_i)}$ and is actually the product of corresponding root subgroups in some certain orders. 

Let $\CC$ be any $W$-conjugacy class  in $[\tilde{W}]$. Consider $\tilde{b}=b(\tilde{w})$ is a good position braid representative of $\CC$ where $\tilde{w}=w\delta$. By our proof in \S 3.2, one can deduce that the right Deligne-Garside normal form of $\tilde{b}$ is $\tilde{b}=b_2b_1\delta=\underline{w'}\cdot \underline{w'w}\delta$. (We formally write $\tilde{b}=\underline{\id}\cdot\underline{w}\delta$ when there is only one term in the Deligne-Garside normal form of $\tilde{b}$). Here $w'$ is the longest element in the standard parabolic subgroup $W'$ of $W$ corresponding to $R^{\tilde{w}}$, which is the set of fixed roots under the action of $\tilde{w}$. The expression $w'w$ is reduced, which means $l(w'w)=l(w')+l(w)$. In this case, we have $R_1, R_2$ are defined where $R_2=\Inv(w')=(R^{\tilde{w}})^+$ and $R_1=(R^{\tilde{w}})^- \sqcup \Inv(w^{-1})$. Since $\delta(R^+)=R^+$, one can see that $\Inv(w^{-1})=\Inv(\tilde{w}^{-1})$. Recall in definition \ref{3.2}, the above braid element $b(\tilde{w})$ is a good position braid representative of $(\tilde{w},\underline{\Theta})$ where $\underline{\Theta}$ is the increasing complete sequence in $(0,\pi]\cup \{2\pi\}$ for $\tilde{w}$. 

More generally, if $2\pi\in \underline{\Theta}$, we consider the case when $\tilde{b}$ is a good position braid representative of $(\tilde{w},\underline{\Theta}')$. Here $\underline{\Theta}'$ is obtained by replacing the $2\pi$ in $\underline{\Theta}$ with $2n\pi$ for any $n\in \BZ_{>1}$. One can notice that $\underline{\Theta}'$ is actually the increasing complete sequence in $(0,\pi]\cup\{2n\pi\}$ for $\tilde{w}$. In this case, as before we have the left Deligne-Garside normal form of $\tilde{b}$ is $\tilde{b}=(\underline{w'})^{2n-1}\cdot \underline{w'w}\delta$. Therefore, we have $R_1=(R^{\tilde{w}})^-\sqcup \Inv(w^{-1})$ and $R_2=R_4=\cdots=R_{2n}=(R^{\tilde{w}})^+$ and $R_3=R_5=\cdots=R_{2n-1}=(R^{\tilde{w}})^-$. Now we define the slices for these braid elements.

\begin{definition}\label{4.2}
    Given a good position braid representative $\tilde{b}=b\delta\in B^+(\tilde{W})$ of a pair $(\tilde{w},\underline{\Theta})$ where $\tilde{w}=w\delta$ and $\underline{\Theta}$ is the increasing complete sequence in $(0,\pi]\cup\{2n\pi\}$ for $\tilde{w}$. Define the \textit{braid slice} $S_{Br}^D(\tilde{b})$ in $D\subset \tilde{G}$ associated to $\tilde{b}=(\underline{w'})^{2n-1}\cdot \underline{w'w}\delta$ as
    \begin{equation*}
        S_{Br}^D(\tilde{b}):=T^{\tilde{w}}U_{R_{2n}}U_{R_{2n-1}}\cdots U_{R_1}\dot{w}\tilde{g}_D.
    \end{equation*}
    Here $\dot{w}$ is a representative of $w$, $T^{\tilde{w}}$ is the subtorus of $T$ consisting of the fixed points under the action of $\tilde{w}$. The element $\tilde{g}_D\in D$ is determined as before if $D\neq G$ and is set to be identity if $D=G$. Denote $U_{R_{2n}}U_{R_{2n-1}}\cdots U_{R_1}$ by $U^{\tilde{b}}$.

    We also define the $G$-\textit{braid slice} $S_{Br}(\tilde{b})$ in $G\subset \tilde{G}$ associated to $\tilde{b}$ as
    \begin{equation*}
        S_{Br}(\tilde{b}):=T^{\tilde{w}}U^{\tilde{b}}\dot{w}.
    \end{equation*}
\end{definition}

\begin{remark}\label{4.4A}
    (1) The slice $S_{Br}^D(\tilde{b})$ is equal to $S_{Br}(\tilde{b})$ if $D=G$, e.g. in case \RNum{2}.

    (2) The ``unipotent part'' $U^{\tilde{b}}$ here does not lie in $G_u$ in general, we just write it as $U^{\tilde{b}}$ to emphasize it is an analogy of $U^w=U\cap \dot{w}U^- \dot{w}^{-1}$. For example, if $\tilde{b}=\underline{w_0}^2$, we have $U^{\tilde{b}}=UU^-$ which is clearly not unipotent.
\end{remark}

\begin{definition}\label{Enl}
    Keep the notations in definition \ref{4.2} and assume that $\tilde{b}$ is actually the good position braid representative $b(\tilde{w})$ of $\tilde{w}$. Define the \textit{enlarged slice} of $\tilde{b}$ as 
    \begin{equation*}
        \widehat{S_{Br}^D}(\tilde{b}):=L_{\tilde{w}}U^w\dot{w}\tilde{g}_D,
    \end{equation*} 
    where $L_{\tilde{w}}$ is the (possibly disconnected) reductive group generated by $T^{\tilde{w}}$ and the root subgroups corresponding to $R^{\tilde{w}}$. We similarly define the \textit{enlarged $G$-slice}
    \begin{equation*}
        \widehat{S_{Br}}(\tilde{b}):=L_{\tilde{w}}U^w\dot{w}.
    \end{equation*} 
\end{definition}
For any given $\tilde{w}\in \CC$ which admits a good position braid representative $b(\tilde{w})$, the above braid slices $S_{Br}^D(\tilde{b})$ (resp. $S_{Br}(\tilde{b})$) in \ref{4.2} are all open dense in $\widehat{S_{Br}^D}(b(\tilde{w}))$ (resp. $\widehat{S_{Br}}(b(\tilde{w}))$).

\begin{example}\label{4.3} (Both are nontwisted case)
Keep the labeling in \S 3.4.

(1) Let $G$ be of type $A_2$, $D=G$ and $\CC$ be the conjugacy class corresponding to cycle type $(2,1)$ (non-elliptic). Then $b=\underline{s_1s_2s_1}=\underline{w_0}$ is a good position braid representative of $\CC$. However, it is not of minimal length in $\CC$. In this case, the braid slice of $b$ and the enlarged slice of $b$ are both $S_{Br}^D(b)=S_{Br}(b)=T^{w_0}U^{w_0}\dot{w_0}$. 
\begin{equation*}
    \{ \begin{pmatrix}
         0 & 0 & -t\\
         0 & t^{-2} & a_3\\
         t & a_1 & a_2
        \end{pmatrix}\mid t
        \in \Bk^*,a_1,a_2,a_3\in \Bk\}.
    \end{equation*}
Indeed, we have $S_{Br}(\tilde{b})=\widehat{S_{Br}}(\tilde{b})$ whenever $R^{\tilde{w}}$ is trivial.

(2) Let $G$ be of type $C_2$, $D=G$ and $\CC$ be the conjugacy class corresponding to a positive $2$-cycle. (This is again non-elliptic). Then one may choose $b= \underline{s_2}\cdot\underline{s_1s_2s_1s_2}=\underline{s_2}\cdot \underline{w_0}$ as a good position braid representative of $\CC$. Thus $w=\pi_W(b)=s_1s_2s_1$. In this case, we have $U^b=U^{s_2}U_{s_2(R^+)}$ and thus the braid slice of $b$ is $S_{Br}^D(b)=S_{Br}(b)=T^{w}U^{s_2}U_{s_2(R^+)}\dot{w}$. It is easy to notice that $U^b$ is not unipotent as $U_{e_2-e_1}U_{e_1-e_2}\subset U^b$. The enlarged slice of $b$ is $\widehat{S_{Br}}(b)=L_{w}U^w\dot{w}$ where $L_w$ is generated by $U_{e_2-e_1}$, $U_{e_1-e_2}$ and $T^w$. From this example, one can see that in the enlarged slice, the part $L_{w}U^w$ is a subgroup of $G$. This may admit some better geometric properties than braid slices.
\end{example}

From now on we always assume that $\tilde{b}$ is a braid element in definition \ref{4.2}. We first prove the following lemma on convexity of subsets of $R$. 

\begin{lemma} \label{4.5}
The subsets $(R^{\tilde{w}})^+\sqcup \Inv(w^{-1})$ and $(R^{\tilde{w}})^-\sqcup \Inv(w^{-1})$ are convex in $R$.
\end{lemma}
\begin{proof}
    As before, let $V$ be the real reflection representation of $W$. Since $\tilde{b}=b(\tilde{w})$ is a good position braid representative of $\CC$, then there exists $v\in V$ such that $(v,\alpha)>0$ for any $\alpha \in R^+\setminus R^{\tilde{w}}$ and $(v,\alpha)=0$ for any $\alpha\in R^{\tilde{w}}$.

    Let $\alpha\in R^{\tilde{w}}=R^{\tilde{w}^{-1}}$ and $\beta\in \Inv(w^{-1})=\Inv({\tilde{w}}^{-1})$. Suppose that $m\alpha+n\beta\in R$ for some $m,n\in \BZ_{>0}$, then we have $(v,m\alpha+n\beta)=(v,n\beta)>0$. This tells that $m\alpha+n\beta\in R^+\setminus R^{\tilde{w}}$. Meanwhile, we have $(v,{\tilde{w}}^{-1}(m\alpha+n\beta))=(v,m\alpha+n{\tilde{w}}^{-1}(\beta))=n(v,\tilde{w}^{-1}(\beta))<0$. Then we have $\tilde{w}^{-1}(m\alpha+n\beta)\in R^-$ and $m\alpha+n\beta\in \Inv(\tilde{w}^{-1})=\Inv(w^{-1})$. Since $(R^{\tilde{w}})^+$, $(R^{\tilde{w}})^-$ and $\Inv(w^{-1})$ are clearly convex, we have $(R^{\tilde{w}})^+\sqcup \Inv(w^{-1})$ and $(R^{\tilde{w}})^-\sqcup \Inv(w^{-1})$ are convex.
\end{proof}

Based on the proof of above lemma and Chevalley commutator formula, we have $U_{(R^{\tilde{w}})^-}U^w=U^wU_{(R^{\tilde{w}})^-}$. specifically, one can see that $U_{\alpha}U^w=U^wU_{\alpha}$ for any $\alpha\in (R^{\tilde{w}})^-$. Then due to the fact that $U_{(R^{\tilde{w}})^-}$ is the product of the root subgroups corresponding to the roots in $(R^{\tilde{w}})^-$ in certain orders by \cite[\S 2.7]{HL}, we get the equation. Now since $U_{R_1}=\dot{w}'U^{w'w}(\dot{w}')^{-1}$, by \cite[\S 2.7]{HL} again we have $U_{R_1}=U_{(R^{\tilde{w}})^-}U^w$ and every element in $U_{R_1}$ can be uniquely written as a product of an element in $U_{(R^{\tilde{w}})^-}$ and an element in $U^w$ as $w'w$ is a reduced expression.

Next we show there is a natural way to relate the $G$-braid slice $S_{Br}(\tilde{b})$ to the original slices in \S 2.4. First consider the case when $w'$ is not the identity.
\begin{lemma}\label{4.4}
    Let $\tilde{b}=(\underline{w'})^{2n-1}\cdot \underline{w'w}\delta$ be a good position braid representative as above. For certain choices of representatives in $G$, there is a natural surjective map from $ T^{\tilde{w}} \times S(w')\times \cdots \times S(w') \times S(w'w)$ with $(2n-1)$-copies of $S(w')$ to the $G$-braid slice $S_{Br}(\tilde{b})$ given by multiplication. In particular, this map is bijective if $n=1$.
\end{lemma}
\begin{proof}
    Since $w'$ is the longest element corresponding to $R^{\tilde{w}}$, one can choose a representative $\dot{w}'$ such that $(\dot{w}')^2\in T^{\tilde{w}}$. Keeping this choice, any element of $ T^{\tilde{w}} \times S(w')\times \cdots \times S(w') \times S(w'w)$ can be identified as $(t,u_1\dot{w}',\ldots,u_{2n-1}\dot{w}',u_{2n}\dot{w}'\dot{w})$ with $t\in T^{\tilde{w}}$, $u_{2n}\in U^{w'w}$ and $u_i\in U^{w'}$ for $1\leq i\leq 2n-1$. Then we have
    \begin{eqnarray*}
        tu_1\dot{w}'\cdots u_{2n-1}\dot{w}'u_{2n}\dot{w}'\dot{w}&=&tu_1(\dot{w}'u_2(\dot{w}')^{-1})\cdots((\dot{w}')^{2n-1}u_{2n}(\dot{w}')^{1-2n})(\dot{w}')^{2n}\dot{w}\\
        &=&tu_1'u_2'\cdots u_{2n}'(\dot{w}')^{2n}\dot{w}\\
        &=&t(\dot{w}')^{2n}u_1''u_2''\cdots u_{2n}''\dot{w},
    \end{eqnarray*}
    where $u_i'=(\dot{w}')^{i-1}u_i(\dot{w}')^{1-i}$ and $u_i''=(\dot{w}')^{-2n}u_i'(\dot{w}')^{2n}$ lie in $U_{R_{2n+1-i}}$ for all $i$ by our choice of $\dot{w}'$. This gives the natural multiplication map. Since any element in $S_{Br}(\tilde{b})$ can be written in the form as the right hand side, we have this map is surjective.

    Now if $n=1$, we claim that any element in $S_{Br}(\tilde{b})$ has a unique decomposition into the form $tu_1''u_2''\dot{w}$ with $t\in T^{\tilde{w}}$, $u_1''\in U_{R_2}=U^{w'}$ and $u_2''\in U_{R_1}=\dot{w}'U^{w'w}(\dot{w}')^{-1}$. By our discussion above and \cite[\S 2.3]{HL} we know $U_{R_2}U_{R_1}=U_{(R^{\tilde{w}})^+}U_{(R^{\tilde{w}})^-}U^w$ and any element in $U_{(R^{\tilde{w}})^+}U_{(R^{\tilde{w}})^-}U^w$ can be written uniquely as a product of elements in $U_{(R^{\tilde{w}})^+}$, $U_{(R^{\tilde{w}})^-}$ and $U^w$. This proves the claim and thus the bijectivity in this case.
\end{proof}

In above lemma we allow $w'=\id$, in which case $S_{Br}(\tilde{b})=T^{\tilde{w}}U^w\dot{w}$ and there is a natural bijective map from $T^{\tilde{w}}\times S(\id)\times S(w)$ to $S_{Br}(\tilde{b})$.

\subsection{Cross-section property}
In this subsection we care about the cross-section property. Recall that an important version of cross-section property of Steinberg's slice and He-Lusztig's slices is about the $U$-orbits on $U\dot{w}U=S(w)U$. We will generalize this for our braid slices and enlarged slices. We consider the action of $U$ on $S_{Br}^D(\tilde{b})$. However, we have to modify this action. 

Let $\tilde{b}=(\underline{w'})^{2n-1}\cdot \underline{w'w}\delta$ be a good position braid representative in definition \ref{4.2}. We first need to fix the choice of representatives for all $w\in W$. In this subsection the choice of representatives will be the same as in \cite[\S 2.4]{HL}. From there we immediately know that these representatives satisfy the braid relations and is compatible with the action of $\delta$. Moreover, we have $(\dot{w}')^2\in T^{\tilde{w}}$ based on this choice. Indeed, one notice that $\lambda^\vee=\tilde{w}(\lambda^\vee)$ for any coweight $\lambda^\vee$ corresponding to $R^{\tilde{w}}$. One also notice that 
\begin{equation*}
    \langle \alpha, \lambda^\vee \rangle=\langle \tilde{w}(\alpha),\tilde{w}(\lambda^\vee) \rangle=\langle \alpha, \tilde{w}(\lambda^\vee)\rangle
\end{equation*}
for any $\alpha\in R^{\tilde{w}}$. By \cite[\S 2.4.(b)]{HL}, we have $(\dot{w}')^2\in T^{\tilde{w}}$.

We first consider slices in case \RNum{1}. Recall above we have $U_{R_1}=U_{(R^{\tilde{w}})^-}U^w=U^wU_{(R^{\tilde{w}})^-}$. Similarly, one can see that $U=U_{R^+\setminus R^{\tilde{w}}}U_{(R^{\tilde{w}})^+}=U_{(R^{\tilde{w}})^+}U_{R^+\setminus R^{\tilde{w}}}$ and $U_{R^+\setminus R^{\tilde{w}}}U_{(R^{\tilde{w}})^-}=U_{(R^{\tilde{w}})^-}U_{R^+\setminus R^{\tilde{w}}}$. Note that if $w'=\id$, then $U_{R_1}=U^w$ and $U_{R_i}$ are all trivial for $i>1$. Therefore, we have
\begin{eqnarray*}
    U_{R^+\setminus R^{\tilde{w}}}S_{Br}^D(\tilde{b})U_{R^+\setminus R^{\tilde{w}}}&=&U_{R^+\setminus R^{\tilde{w}}} U^{\tilde{b}}\dot{w}dU_{R^+\setminus R^{\tilde{w}}}\\
    &=& U_{R^+\setminus R^{\tilde{w}}}T^{\tilde{w}}U_{(R^{\tilde{w}})^+}U_{(R^{\tilde{w}})^-}\cdots U_{(R^{\tilde{w}})^+}U_{(R^{\tilde{w}})^-} U^w\dot{w}dU_{R^+\setminus R^{\tilde{w}}}\\
    &=& T^{\tilde{w}}U_{(R^{\tilde{w}})^+}U_{(R^{\tilde{w}})^-}\cdots U_{(R^{\tilde{w}})^+}U_{(R^{\tilde{w}})^-}U_{R^+\setminus R^{\tilde{w}}}U^w\dot{w}dU_{R^+\setminus R^{\tilde{w}}} \\
    &=& U^{\tilde{b}}\dot{w}d U_{R^+\setminus R^{\tilde{w}}}=S_{Br}^D(\tilde{b})U_{R^+\setminus R^{\tilde{w}}}.
\end{eqnarray*}
Therefore, we would consider the $U_{R^+\setminus R^{\tilde{w}}}$-orbits on $S_{Br}^D(\tilde{b})U_{R^+\setminus R^{\tilde{w}}}$ in case \RNum{1}. 

However, in case \RNum{2}, we need to modify above two equations. In general, we consider the twisted conjugation for the $G$-slice $S_{Br}(\tilde{b})$. Similar as above, we have
\begin{eqnarray*}
   U_{R^+\setminus R^{\tilde{w}}}S_{Br}(\tilde{b})U_{\delta(R^+\setminus R^{\tilde{w}})}&=&U_{R^+\setminus R^{\tilde{w}}}U^{\tilde{b}}\dot{w}U_{\delta(R^+\setminus R^{\tilde{w}})}\\
&=&T^{\tilde{w}}U_{(R^{\tilde{w}})^+}U_{(R^{\tilde{w}})^-}\cdots U_{(R^{\tilde{w}})^+}U_{(R^{\tilde{w}})^-}U_{R^+\setminus R^{\tilde{w}}}U^w\dot{w}U_{\delta(R^+\setminus R^{\tilde{w}})}\\
    &=&U^{\tilde{b}}\dot{w}U_{\delta(R^+\setminus R^{\tilde{w}})}=S_{Br}(\tilde{b})U_{\delta(R^+\setminus R^{\tilde{w}})}.
\end{eqnarray*}
We first show any element in $S_{Br}(\tilde{b})U_{\delta(R^+\setminus R^{\tilde{w}})}$ has a unique expression.

\begin{lemma}\label{4.6}
    There is a natural bijection
    \begin{equation*}
         S_{Br}(\tilde{b}) \times U_{R^+\setminus R^{\tilde{w}}} \xlongrightarrow{\sim}  S_{Br}(\tilde{b})U_{\delta(R^+\setminus R^{\tilde{w}})}
    \end{equation*}
     given by $(g\dot{w},u)\mapsto g\dot{w}\delta_G(u)$.
\end{lemma}
\begin{proof}
    Clearly this is surjective. Suppose that $g_1 \dot{w} \delta_G(u_1)=g_2 \dot{w} \delta_G(u_2)$, then we have $g_1^{-1}g_2=\dot{w} \delta_G(u_1u_2^{-1})\dot{w}^{-1} \in U_{\tilde{w}(R^+\setminus R^{\tilde{w}})}$. Again by \cite[\S 2.7]{HL}, We know $U_{\tilde{w}(R^+\setminus R^{\tilde{w}})}$ is actually the product of root subgroups corresponding to $\tilde{w}(R^+\setminus R^{\tilde{w}})$. This is equal to $(R^+\setminus (R^{\tilde{w}}\sqcup \Inv(\tilde{w}^{-1})))\sqcup (-\Inv(\tilde{w}^{-1}))$. These root subgroups do not appear in $L_{\tilde{w}}U^w$ by our definition above and the fact that $\Inv(w^{-1})=\Inv(\tilde{w}^{-1})$. Since $g_1,g_2\in U^{\tilde{b}}\subset L_{\tilde{w}}U^w$, we must have $g_1=g_2$ and $u_1=u_2$.
\end{proof}

Then we consider the action of $U_{R^+\setminus R^{\tilde{w}}}$. Following the idea in \cite[\S 2, \S 3]{HL}, for any $r\in \BZ_{>0}$, we set
\begin{gather*}
    U(\tilde{b}^r)=S_{Br}(\tilde{b})\times \delta_G(S_{Br}(\tilde{b})) \cdots \times \delta_G^{r-1}(S_{Br}(\tilde{b})) ,\\
    \dot{U}(\tilde{b}^r)= S_{Br}(\tilde{b})U_{\delta(R^+/R^{\tilde{w}})}\times \delta_G(S_{Br}(\tilde{b})U_{\delta(R^+/R^{\tilde{w}})}) \cdots \times \delta_G^{r-1}(S_{Br}(\tilde{b})U_{\delta(R^+/R^{\tilde{w}})}),
\end{gather*}
By above discussion, we know that 
\begin{eqnarray*}
    \delta_G^i(S_{Br}(\tilde{b})U_{\delta(R^+/R^{\tilde{w}})})&=&\delta_G^i(U_{R^+\setminus R^{\tilde{w}}})\delta_G^i(S_{Br}(\tilde{b}))\delta_G^i(U_{\delta(R^+/R^{\tilde{w}})})\\
    &=&U_{\delta^i(R^+\setminus R^{\tilde{w}})}\delta_G^i(S_{Br}(\tilde{b}))U_{\delta^{i+1}(R^+\setminus R^{\tilde{w}})}.
\end{eqnarray*}

Therefore, there is an action of $(U_{R^+\setminus R^{\tilde{w}}})^{r-1}$ on $\dot{U}(\tilde{b}^r)$ given by
\begin{equation*}
    (u_1,u_2,\ldots,u_{r-1})\cdot (g_1,g_2,\ldots,g_r)=(g_1\delta_G(u_1)^{-1},\delta_G(u_1)g_2\delta_G^2(u_2)^{-1},\ldots , \delta_G^{r-1}(u_{r-1})g_r).
\end{equation*}
Let $\tilde{U}(\tilde{b}^r)$ be the set of orbits of above action and $\kappa_{\tilde{b}}^r:\dot{U}(\tilde{b}^r) \to \tilde{U}(\tilde{b}^r)$ be the natural projection map. Denote $\kappa_{\tilde{b}}^r(g_1,g_2,\ldots,g_r)$ as $[g_1,g_2\ldots,g_r]$. One can check there is a natural $U_{R^+\setminus R^{\tilde{w}}}\times U_{\delta^r(R^+\setminus R^{\tilde{w}})}$-action on $\tilde{U}(\tilde{b}^r)$ given by
\begin{equation*}
    (u,u')\cdot [g_1,g_2,\ldots,g_r]=[ug_1,g_2,\ldots,g_r (u')^{-1}].
\end{equation*}

Meanwhile, recall $\tilde{b}=(\underline{w'})^{2n-1}\cdot \underline{w'w}\delta$, we set
\begin{equation*}
U(\text{DG}(\tilde{b}),T^{\tilde{w}}):=T^{\tilde{w}}U_{R^+\setminus R^{\tilde{w}}}\times S(w')U_{R^+\setminus R^{\tilde{w}}}\times \cdots \times  S(w')U_{R^+\setminus R^{\tilde{w}}} \times S(w'w)U_{\delta(R^+\setminus R^{\tilde{w}})},    
\end{equation*}
where there are $(2n-1)$-copies of $S(w')U_{R^+\setminus R^{\tilde{w}}}$ on the left hand side. Based on our choice of representatives above, recall there is a natural surjective map in lemma \ref{4.4}, we can then define
\begin{equation*}
    \Upsilon_{\tilde{b}}:U(\text{DG}(\tilde{b}),T^{\tilde{w}}) 
    \to S_{Br}(\tilde{b})U_{\delta(R^+\setminus R^{\tilde{w}})}
\end{equation*}
as follows. Let $(tu_0,g_1u_1,\ldots, g_{2n}u_{2n})$ be any element of left hand side. Here $t\in T^{\tilde{w}}$, $g_i\in S(w')$ for $1\leq i\leq 2n-1$, $g_{2n}\in S(w'w)$, $u_i\in U_{R^+\setminus R^{\tilde{w}}}$ for $0\leq i\leq 2n-1$ and $u_{2n}\in U_{\delta(R^+\setminus R^{\tilde{w}})}$. As before, by \cite[\S 2.7]{HL} and Chevalley commutator formula one can show that
\begin{eqnarray*}
    U_{R^+\setminus R^{\tilde{w}}}S(w'w)U_{\delta(R^+\setminus R^{\tilde{w}})}&=& U_{R^+\setminus R^{\tilde{w}}} U_{\Inv((w'w)^{-1})}\dot{w}'\dot{w}U_{\delta(R^+\setminus R^{\tilde{w}})}\\
    &=& U_{R^+\setminus R^{\tilde{w}}}\dot{w}'^{-1}U_{(R^{\tilde{w}})^-}U_{\Inv(w^{-1})}\dot{w}'\dot{w}'\dot{w}U_{\delta(R^+\setminus R^{\tilde{w}})}\\
    &=&\dot{w}'^{-1}U_{(R^{\tilde{w}})^-}U_{\Inv(w^{-1})}U_{R^+\setminus R^{\tilde{w}}}\dot{w}'\dot{w}'\dot{w}U_{\delta(R^+\setminus R^{\tilde{w}})}\\
    &=& \dot{w}'^{-1}U_{(R^{\tilde{w}})^-}U_{\Inv(w^{-1})}\dot{w}'\dot{w}' U_{R^+\setminus R^{\tilde{w}}}\dot{w}U_{\delta(R^+\setminus R^{\tilde{w}})} \\
    &=& S(w'w)U_{\delta(R^+\setminus R^{\tilde{w}})}.
\end{eqnarray*}
Similarly, one can show that 
\begin{equation*}
    U_{R^+\setminus R^{\tilde{w}}}S(w') U_{R^+\setminus R^{\tilde{w}}}=U_{(R^{\tilde{w}})^+}U_{R^+\setminus R^{\tilde{w}}}\dot{w}' U_{R^+\setminus R^{\tilde{w}}}=U_{(R^{\tilde{w}})^+}\dot{w}' U_{R^+\setminus R^{\tilde{w}}}=S(w')U_{R^+\setminus R^{\tilde{w}}}.
\end{equation*}
Therefore, we have 
\begin{eqnarray*}
    tu_0g_1u_1\cdots g_{2n}u_{2n}&=&t(u_0g_1)(u_1g_2)\cdots (u_{2n-1}g_{2n})u_{2n}\\
    &=&t(g_1'u_0')(g_2'u_1')\cdots (g_{2n}u_{2n-1}')u_{2n}\\
    &=& \cdots \\
    &=& tgu,
\end{eqnarray*}
where $g\in S(w')\times\cdots \times S(w')\times S(w'w)$ and $u\in U_{\delta(R^+\setminus R^{\tilde{w}})}$. Let $g'$ be the image of $tg$ under the map in lemma \ref{4.4}. Then the image of $(tu_0,g_1u_1,\ldots, g_{2n}u_{2n})$ under $\Upsilon_{\tilde{b}}$ is $g'u\in S_{Br}(\tilde{b})U_{\delta(R^+\setminus R^{\tilde{w}})}$.

Next we define $U(\text{DG}(\tilde{b}),T^{\tilde{w}},r)$ as
\begin{equation*}
    U(\text{DG}(\tilde{b}),T^{\tilde{w}}) \times \delta_G(U(\text{DG}(\tilde{b}),T^{\tilde{w}})) \times \cdots \times \delta_G^{r-1}(U(\text{DG}(\tilde{b}),T^{\tilde{w}})).
\end{equation*}
Here $\delta_G^i(U(\text{DG}(\tilde{b}),T^{\tilde{w}}))$ is defined as 
\begin{equation*}
    \delta_G^i(T^{\tilde{w}}U_{R^+\setminus R^{\tilde{w}}})\times \delta_G^i(S(w')U_{R^+\setminus R^{\tilde{w}}})\times\cdots \times \delta_G^i(S(w')U_{R^+\setminus R^{\tilde{w}}})\times \delta_G^i(S(w'w)U_{\delta(R^+\setminus R^{\tilde{w}})})
\end{equation*}
Recall that $U_{R^+\setminus R^{\tilde{w}}}S(w'w)U_{\delta(R^+\setminus R^{\tilde{w}})}=S(w'w)U_{\delta(R^+\setminus R^{\tilde{w}})}$ and $U_{R^+\setminus R^{\tilde{w}}}S(w') U_{R^+\setminus R^{\tilde{w}}}=S(w')U_{R^+\setminus R^{\tilde{w}}}$. It is also clear that $U_{R^+\setminus R^{\tilde{w}}}T^{\tilde{w}} U_{R^+\setminus R^{\tilde{w}}}=T^{\tilde{w}}U_{R^+\setminus R^{\tilde{w}}}$. Therefore, one can define a $(U_{R^+\setminus R^{\tilde{w}}})^{(2n+1)r-1}$-action on $U(\text{DG}(\tilde{b}),T^{\tilde{w}},r)$ as follows. 

Let $(u_1,\ldots,u_{(2n+1)r-1})\in (U_{R^+\setminus R^{\tilde{w}}})^{(2n+1)r-1}$ and set $u_0=u_{(2n+1)r}$ to be trivial. Let $(g_1^{(1)},\ldots, g_1^{(2n+1)},\delta_G(g_2^{(1)}),\ldots,\delta_G(g_2^{(2n+1)}),\ldots,\delta_G^{r-1}(g_r^{(1)}),\ldots,\delta_G^{r-1}(g_r^{(2n+1)}))$ be an arbitrary element in $U(\text{DG}(\tilde{b}),T^{\tilde{w}},r)$. Here $g_i^{(1)}\in T^{\tilde{w}}U_{R^+\setminus R^{\tilde{w}}}$, $g_i^{(2n+1)}\in S(w'w)U_{\delta(R^+\setminus R^{\tilde{w}})}$ and $g_i^{(j)}\in S(w')U_{R^+\setminus R^{\tilde{w}}}$ for $2\leq j\leq 2n$ and all $i$. It will then be taken to $(g_1'^{(1)},\ldots, g_1'^{(2n+1)},\delta_G(g_2'^{(1)}),\ldots,\delta_G(g_2'^{(2n+1)}),\ldots,\delta_G^{r-1}(g_r'^{(1)}),\ldots,\delta_G^{r-1}(g_r'^{(2n+1)}))$ with
\begin{eqnarray*}
    \delta_G^{i-1}(g_i'^{(1)})&=&\delta_G^{i-1}(u_{(2n+1)(i-1)}g_i^{(1)}u_{(2n+1)i-2n}^{-1})\in \delta_G^{i-1}(T^{\tilde{w}}U_{R^+\setminus R^{\tilde{w}}}),\\
    \delta_G^{i-1}(g_i'^{(2)})&=&\delta_G^{i-1}(u_{(2n+1)i-2n}g_i^{(2)}u_{(2n+1)i-2n+1}^{-1})\in \delta_G^{i-1}(S(w')U_{R^+\setminus R^{\tilde{w}}}),\\
    &\cdots& \\
    \delta_G^{i-1}(g_i'^{(2n)})&=&\delta_G^{i-1}(u_{(2n+1)i-2}g_i^{(2n)}u_{(2n+1)i-1}^{-1})\in \delta_G^{i-1}(S(w')U_{R^+\setminus R^{\tilde{w}}}),\\
    \delta_G^{i-1}(g_i'^{(2n+1)})&=&\delta_G^{i-1}(u_{(2n+1)i-1}g_i^{(2n+1)}\delta_G(u_{(2n+1)i})^{-1})\in \delta_G^{i-1}(S(w'w)U_{\delta(R^+\setminus R^{\tilde{w}})}),
\end{eqnarray*}
for $1\leq i\leq r$ under the action of $(u_1,u_2,\ldots,u_{(2n+1)r-1})$. Denote the orbit set of this action as $\tilde{U}(\text{DG}(\tilde{b}),T^{\tilde{w}},r)$. There is again a similar action of $U_{R^+\setminus R^{\tilde{w}}}\times U_{\delta^r(R^+\setminus R^{\tilde{w}})}$ on $\tilde{U}(\text{DG}(\tilde{b}),T^{\tilde{w}},r)$ as on $\tilde{U}(\tilde{b}^r)$.

We first have the following lemma.
\begin{lemma}\label{4.7}
    The map $\Upsilon_{\tilde{b}}$ induces a natural surjection
    \begin{equation*}
        \tilde{\Upsilon}_{\tilde{b}}^r: \tilde{U}(\text{DG}(\tilde{b}),T^{\tilde{w}},r) \longrightarrow \tilde{U}(\tilde{b}^r).
    \end{equation*}
    This map is compatible with $U_{R^+\setminus R^{\tilde{w}}}\times U_{\delta^r(R^+\setminus R^{\tilde{w}})}$-actions on both sides.
\end{lemma}

\begin{proof}
    First notice that given any $\{u_i\in U_{\delta(R^+\setminus R^{\tilde{w}})}\}_{1\leq i\leq r}$ and $\{x_i\in S_{Br}(\tilde{b})\}_{1\leq i\leq r}$, there is a unique expression $[x_1u_1,\delta_G(x_2u_2),\ldots,\delta_G^{r-1}(x_ru_r)]=[g_1,\ldots,g_ru]$ in $\tilde{U}(\tilde{b}^r)$ with $u\in \delta_G^{r-1}(U_{\delta(R^+\setminus R^{\tilde{w}})})=U_{\delta^r(R^+\setminus R^{\tilde{w}})}$ and $g_i\in \delta_G^{i-1}(S_{Br}(\tilde{b}))$ for all $1\leq i\leq r$. Indeed, since $ \delta_G^i(S_{Br}(\tilde{b})U_{\delta(R^+/R^{\tilde{w}})})=U_{\delta^i(R^+\setminus R^{\tilde{w}})}\delta_G^i(S_{Br}(\tilde{b}))U_{\delta^{i+1}(R^+\setminus R^{\tilde{w}})}$, we have 
    \begin{eqnarray*}
        [x_1u_1,\delta_G(x_2u_2),\ldots,\delta_G^{r-1}(x_ru_r)]&=&[x_1,u_1\delta_G(x_2u_2),\ldots,\delta_G^{r-1}(x_ru_r)]\\
        &=&[x_1,\delta_G(x_2')\delta_G(u_2'),\ldots,\delta_G^{r-1}(x_ru_r)]\\
        &=&\cdots \\
        &=&[x_1,\delta_G(x_2'),\ldots,\delta_G^{r-2}(x_{r-1}'),\delta_G^{r-1}(x_r')\delta_G^{r-1}(u_r')],
    \end{eqnarray*}
    where $u_r'\in U_{\delta(R^+\setminus R^{\tilde{w}})}$ and $x_i'\in S_{Br}(\tilde{b})$ for all $i$. Moreover, by lemma \ref{4.6}, above expression on the right hand side is unique.

    Similarly, given any element in $\tilde{U}(\text{DG}(\tilde{b}),T^{\tilde{w}},r)$, there is a unique expression of it as $[t_1,g_1^{(1)},\ldots,g_1^{(2n)},\delta_G(t_2),\ldots,\delta_G(g_2^{(2n)}),\ldots,\delta_G^{r-1}(t_r),\ldots,\delta_G^{r-1}(g_r^{(2n)})u']$ with $t_i\in T^{\tilde{w}}$, $u'\in U_{\delta^r(R^+\setminus R^{\tilde{w}})}$, $g_i^{(2n)}\in S(w'w)$ and $g_i^{(j)}\in S(w')$ for $1\leq j\leq 2n-1$.
    
    Now for $\tilde{\Upsilon}_{\tilde{b}}^r$, consider any $[t_1,g_1^{(1)},\ldots,g_1^{(2n)},\ldots,\delta_G^{r-1}(t_r),\ldots,\delta_G^{r-1}(g_r^{(2n)})u']$ as above. For any $i$, we have $t_i\prod\limits_{j=1}^{2n} g_i^{(j)}\in S_{Br}(\tilde{b})$ as given by lemma \ref{4.4}. These naturally define
    \begin{eqnarray*}
        &\tilde{\Upsilon}_{\tilde{b}}^r([t_1,g_1^{(1)},\ldots,g_1^{(2n)},\ldots,\delta_G^{r-1}(t_r),\ldots,\delta_G^{r-1}(g_r^{(2n)})u'])&\\
        &=[t_1\prod\limits_{j=1}^{2n} g_1^{(j)},\ldots,\delta_G^{r-2}(t_{r-1}\prod\limits_{j=1}^{2n} g_{r-1}^{(j)}),\delta_G^{r-1}(t_r\prod\limits_{j=1}^{2n} g_r^{(j)})u']&.
    \end{eqnarray*}
    This map is well-defined and surjective due to lemma \ref{4.4} and the unique expression of elements in $\tilde{U}(\tilde{b}^r)$ mentioned above. 
    
    It remains to show $\tilde{\Upsilon}_{\tilde{b}}^r$ is compatible with $U_{R^+\setminus R^{\tilde{w}}}\times U_{\delta^r(R^+\setminus R^{\tilde{w}})}$-actions. Let $u\in U_{R^+\setminus R^{\tilde{w}}}$ and $u'\in U_{\delta^r(R^+\setminus R^{\tilde{w}})}$, then 
    \begin{eqnarray*}
        (u,u')[g_1,\delta_G(g_2),\ldots,\delta_G^{r-1}(g_r)u_r]&=&[ug_1,\ldots,\delta_G^{r-1}(g_r)u_r(u')^{-1}]\\
        &=&[g_1'u_1',\ldots,\delta_G^{r-1}(g_r)u_r(u')^{-1}]\\
        &=&[g_1',u_1'\delta_G(g_2),\ldots,\delta_G^{r-1}(g_r)u_r(u')^{-1}]\\\
        &=&\cdots \\
        &=&[g_1',\delta_G(g_2'),\ldots,\delta_G^{r-1}(g_r')u_r'u_r(u')^{-1}].
    \end{eqnarray*}
    Here $u_0':=u$, $u_{i-1}'\delta_G^{i-1}(g_i)=\delta_G^{i-1}(g_i')u_i'$ with $u_i' \in U_{\delta^i(R^+\setminus R^{\tilde{w}})}$ and $g_i'\in S_{Br}(\tilde{b})$ determined by applying $\delta_G^{i-1}$ to both sides of lemma \ref{4.6}. Similarly, we have
    \begin{eqnarray*}
        &(u,u')[t_1,g_1^{(1)},\ldots,g_1^{(2n)},\ldots,\delta_G^{r-1}(t_r),\ldots,\delta_G^{r-1}(g_r^{(2n)})u_r]&\\
        & =[t_1u_1^{(0)},g_1^{(1)},\ldots,g_1^{(2n)},\ldots, \delta_G^{r-1}(t_r),\ldots,\delta_G^{r-1}(g_r^{(2n)})u_r(u')^{-1}]&\\
        &=
        [t_1,g_1'^{(1)}u_1^{(1)},\ldots,g_1^{(2n)},\ldots,\delta_G^{r-1}(t_r),\ldots,\delta_G^{r-1}(g_r^{(2n)})u_r(u')^{-1}]&\\
        & \cdots& \\
        &=[t_1,g_1'^{(1)},\ldots,g_1'^{(2n)},\ldots,\delta_G^{r-1}(t_r),\ldots,\delta_G^{r-1}(g_r'^{(2n)})u_r^{(2n)}u_r(u')^{-1}].&
    \end{eqnarray*}
    We know $g_1'u_1'=ug_1=ut_1\prod\limits_{j=1}^{2n} g_1^{(j)}=t_1\prod\limits_{j=1}^{2n} g_1'^{(j)}u_1^{(2n)}$. By lemma \ref{4.6}, we have $u_1'=u_1^{(2n)}$ and $g_1'=t_1\prod\limits_{j=1}^{2n} g_1'^{(j)}$. By passing it to $g_r'u_r'$, it is clear that 
    \begin{eqnarray*}
        &\tilde{\Upsilon}_{\tilde{b}}^r ([t_1,g_1'^{(1)},\ldots,g_1'^{(2n)},\ldots,\delta_G^{r-1}(t_r),\ldots,\delta_G^{r-1}(g_r'^{(2n)})u_r^{(2n)}u_r(u')^{-1}])&\\
        &=[g_1',\delta_G(g_2'),\ldots,\delta_G^{r-1}(g_r')u_r'u_r(u')^{-1}]&
    \end{eqnarray*}
    and our statement is proved.
\end{proof}

Based on above lemma, one can relate the twisted $U_{R^+\setminus R^{\tilde{w}}}$-orbits on $S_{Br}(\tilde{b})U_{\delta(R^+\setminus R^{\tilde{w}})}$ to the $(U_{R^+\setminus R^{\tilde{w}}})^{(2n+1)r-1}$-orbit space $\tilde{U}(\text{DG}(\tilde{b}),T^{\tilde{w}},r)$. Nevertheless, we can further combine these torus parts together and move it to one end.

We set $U(\text{DG}(\tilde{b}))=S(w')U_{R^+\setminus R^{\tilde{w}}}\times \cdots \times S(w')U_{R^+\setminus R^{\tilde{w}}}\times 
 S(w'w)U_{\delta(R^+\setminus R^{\tilde{w}})}$ where there are again $(2n-1)$-copies of $S(w')U_{R^+\setminus R^{\tilde{w}}}$. Then define 
 \begin{equation*}
     U(\text{DG}(\tilde{b}),r)=U(\text{DG}(\tilde{b}))\times \delta_G(U(\text{DG}(\tilde{b})))\times \cdots \times \delta_G^{r-1}(U(\text{DG}(\tilde{b})))
 \end{equation*}
  and
  \begin{equation*}
      \dot{U}(\text{DG}(\tilde{b}),r)=U(\text{DG}(\tilde{b}),r)\times \delta_G^r(T^{\tilde{w}}U_{R^+\setminus R^{\tilde{w}}}).
  \end{equation*}
 Then one can again define an action of $(U_{R^+\setminus R^{\tilde{w}}})^{2nr}$ on $\dot{U}(\text{DG}(\tilde{b}),r)$ as follows. Let $(g_1^{(1)},\ldots,g_1^{(2n)},\ldots,\delta_G^{r-1}(g_r^{(1)}),\ldots,\delta_G^{r-1}(g_r^{(2n)}),h_{r+1})$ be any element in $\dot{U}(\text{DG}(\tilde{b}),r)$ with $h_{r+1}\in \delta_G^r(T^{\tilde{w}}U_{R^+\setminus R^{\tilde{w}}})$, $g_i^{(2n)}\in S(w'w)U_{\delta(R^+\setminus R^{\tilde{w}})}$ and $g_i^{(j)}\in S(w')U_{\delta(R^+\setminus R^{\tilde{w}})}$ for $1\leq j\leq 2n-1$. Let $(u_1,u_2,\ldots,u_{2nr})\in (U_{R^+\setminus R^{\tilde{w}}})^{2nr}$ be arbitrary and formally set $u_0$ to be trivial. Then it takes $(g_1^{(1)},\ldots,g_1^{(2n)},\ldots,\delta_G^{r-1}(g_r^{(1)}),\ldots,\delta_G^{r-1}(g_r^{(2n)}),h_{r+1})$ to $(g_1'^{(1)},\ldots,g_1'^{(2n)},\ldots,\delta_G^{r-1}(g_r'^{(1)}),\ldots,\delta_G^{r-1}(g_r'^{(2n)}),h_{r+1}')$ where 
\begin{eqnarray*}
    \delta_G^{i-1}(g_i'^{(1)})&=&\delta_G^{i-1}(u_{2ni-2n}g_i^{(1)}u_{2ni-2n+1}^{-1}), \\
    \delta_G^{i-1}(g_i'^{(2)})&=&\delta_G^{i-1}(u_{2ni-2n+1}g_i^{(2)}u_{2ni-2n+2}^{-1}),\\
    \cdots \\
    \delta_G^{i-1}(g_i'^{(2n)})&=&\delta_G^{i-1}(u_{2ni-1}g_i^{(2n)}\delta_G(u_{2ni})^{-1}),\\
    h_{r+1}'&=&\delta_G^r(u_{2nr})h_{r+1},
\end{eqnarray*}
for $1\leq i\leq r$. Denote the orbit space of this action as $\tilde{U}(\text{DG}(\tilde{b}),r)$. There is also an $U_{R^+\setminus R^{\tilde{w}}}\times U_{\delta^r(R^+\setminus R^{\tilde{w}})}$-action on $\tilde{U}(\text{DG}(\tilde{b}),r)$ similar as before.

Recall the proof of lemma \ref{4.7}, any element in $\tilde{U}(\text{DG}(\tilde{b}),T^{\tilde{w}},r)$ admits a unique expression as $[t_1,g_1^{(1)},\ldots,g_1^{(2n)},\ldots,\delta_G^{r-1}(t_r),\ldots,\delta_G^{r-1}(g_r^{(2n)})u]$ with some $t_i\in T^{\tilde{w}}$, $u\in U_{\delta^r(R^+\setminus R^{\tilde{w}})}$, $g_i^{(2n)}\in S(w'w)$ and $g_i^{(j)}\in S(w')$ for $1\leq j\leq 2n-1$ and all $i$. Similarly, any element in $\tilde{U}(\text{DG}(\tilde{b}),r)$ also admit a unique expression 
\begin{equation*}
    [g_1^{(1)},\ldots,g_1^{(2n)},\ldots,\delta_G^{r-1}(g_r^{(1)}),\ldots,\delta_G^{r-1}(g_r^{(2n)}),\delta_G^r(t)u]
\end{equation*}
 with some $t\in T^{\tilde{w}}$, $u\in U_{\delta^r(R^+\setminus R^{\tilde{w}})}$, $g_i^{(2n)}\in S(w'w)$ and $g_i^{(j)}\in S(w')$ for $1\leq j\leq 2n-1$ and all $i$.

First assume that $r=2$. Let $[t_1,g_1^{(1)},\ldots,g_1^{(2n)},\delta_G(t_2),\delta_G(g_2^{(1)}),\ldots,\delta_G(g_2^{(2n)})u]$ be any element in $\tilde{U}(\text{DG}(\tilde{b}),T^{\tilde{w}},2)$ as above. Since $t_i\in T^{\tilde{w}}$, we have $t_i\dot{w}=\dot{w}\delta_G(t_i)$ and $t_i(\dot{w}')^2=(\dot{w}')^2t_i$. Then $t_1\prod\limits_{j=1}^{2n}g_1^{(j)}=\prod\limits_{j=1}^{2n}g_1'^{(j)}\delta_G(t_1)$ for some $g_1'^{(2n)}\in S(w'w)$ and $g_1'^{(j)}\in S(w')$ for $1\leq j\leq 2n-1$. Similarly, we have $\delta_G(t_1t_2\prod\limits_{j=1}^{2n}g_2^{(j)})=\delta_G(\prod\limits_{j=1}^{2n}g_2'^{(j)})\delta_G^2(t_1t_2)$ for some $g_2'^{(2n)}\in S(w'w)$ and $g_2'^{(j)}\in S(w')$ for $1\leq j\leq 2n-1$.
\begin{lemma}\label{4.8}
    There is a natural surjection 
    \begin{equation*}
        \tilde{\tau}_{\tilde{w}}^2:\tilde{U}(\text{DG}(\tilde{b}),T^{\tilde{w}},2)\longrightarrow \tilde{U}(\text{DG}(\tilde{b}),2)
    \end{equation*}
     which sends $[t_1,g_1^{(1)},\ldots,g_1^{(2n)},\delta_G(t_2),\delta_G(g_2^{(1)}),\ldots,\delta_G(g_2^{(2n)})u]\in \tilde{U}(\text{DG}(\tilde{b}),T^{\tilde{w}},2)$ to $[g_1'^{(1)},\ldots,g_1'^{(2n)},\delta_G(g_2'^{(1)}),\ldots,\delta_G(g_2'^{(2n)}),\delta_G^2(t_1t_2)u]$ from above equations. This map is compatible with $U_{R^+\setminus R^{\tilde{w}}}\times U_{\delta^2(R^+\setminus R^{\tilde{w}})}$-actions on both sides.
\end{lemma}
\begin{proof}
Due to the uniqueness of the above expression in both sides and the computations above, we have $\tilde{\tau}_{\tilde{w}}^2$ is clearly well-defined and surjective. It remains to show the map is compatible with $U_{R^+\setminus R^{\tilde{w}}}\times U_{\delta^2(R^+\setminus R^{\tilde{w}})}$-actions. This can be checked in the same way as in the proof of lemma \ref{4.7}.    
\end{proof}

By repeating this procedure, we get a natural surjective map 
\begin{equation*}
    \tilde{\tau}_{\tilde{w}}^r:\tilde{U}(\text{DG}(\tilde{b}),T^{\tilde{w}},r)\to \tilde{U}(\text{DG}(\tilde{b}),r),
\end{equation*}
which is compatible with $U_{R^+\setminus R^{\tilde{w}}}\times U_{\delta^r(R^+\setminus R^{\tilde{w}})}$-actions.

Recall our choice of $\tilde{b}=(\underline{w'})^{2n-1}\cdot \underline{w'w}\delta$ is such that $\tilde{b}$ is a good position braid representative of some pair $(\tilde{w},\underline{\Theta})$. Choose $d$ such that $\delta^d=\id_W$ and $d$ satisfies the condition in proposition \ref{a}. Therefore, we have 
\begin{eqnarray*}
    \tilde{b}^d&=&((\underline{w'})^n\cdot \underline{w'w}\delta)^d=(\underline{w'})^n\cdot \underline{w'w}\cdot \delta((\underline{w'})^n\cdot \underline{w'w})\cdots \delta^{d-1}((\underline{w'})^n\cdot \underline{w'w})\cdot \delta^d\\
    &=&\underline{w_0}^{d_0}\underline{w_1}^{d_1} \cdots \underline{w_{l-1}}^{d_{l-1}}.
\end{eqnarray*}
Here all powers are positive integers determined by $\tilde{b}$ and $w_i$ is the longest element in a nontrivial standard parabolic subgroup $W_i<W$ for all $i$. (Here we set $W_0=W$). Moreover, we have a filtration $W_0>W_1>\cdots>W_{l-1}\neq \{\id\}$ and this leads to another filtration $R^+=\Inv(w_0^{-1})\supset \Inv(w_1^{-1})\supset\cdots\supset \Inv(w_{l-1}^{-1})\supset (R^{\tilde{w}})^+$. Moreover, we have $\Inv(w_{l-1}^{-1})=(R^{\tilde{w}})^+$ if and only if $w'\neq \id$ if and only if $R^{\tilde{w}}\neq \emptyset$.

Denote the product $S(w')\times \cdots \times S(w')\times S(w'w)$ with $(2n-1)$-copies of $S(w')$ by $S(\text{DG}(\tilde{b}))$ and the product of $d_i$ copies of $S(w_i)$ by $S(w_i,d_i)$. Define
\begin{equation*}
    S(\text{DG}(\tilde{b}),d)=S(\text{DG}(\tilde{b}))\times \delta_G(S(\text{DG}(\tilde{b})))\times \cdots \times \delta_G^{d-1}(S(\text{DG}(\tilde{b}))).
\end{equation*} 
Notice that by our choice of representatives, we have
\begin{eqnarray*}
    \delta_G^i(S(\text{DG}(\tilde{b})))&=&\delta_G^i(S(w'))\times \cdots \times \delta_G^i(S(w'))\times  \delta_G^i(S(w'w))\\
    &=&S(\delta^i(w'))\times \cdots \times S(\delta^i(w')) \times S(\delta^i(w'w)).
\end{eqnarray*}
By \cite[\S 2.9]{HL} and the good form of $\tilde{b}^d$, there is a bijection 
\begin{equation*}
    \mathscr{G}_{\tilde{b}} :S(\text{DG}(\tilde{b}),d)\xlongrightarrow{\sim} S(w_0,d_0)\times \cdots \times S(w_{l-1},d_{l-1}).
\end{equation*}
For any $S(w_i)$, since $W_i$ is standard parabolic and $\Inv(w_i^{-1})\supset (R^{\tilde{w}})^+$, then we have $U_{R^+\setminus R^{\tilde{w}}}S(w_i)U_{R^+\setminus R^{\tilde{w}}}=U_{R^+\setminus R^{\tilde{w}}}U^{w_i}\dot{w}_iU_{R^+\setminus R^{\tilde{w}}}=U^{w_i}U_{R^+\setminus \Inv(w_i^{-1}) }\dot{w}_iU_{R^+\setminus R^{\tilde{w}}}=S(w_i)U_{R^+\setminus R^{\tilde{w}}}$. The second equation is due to \cite[\S 2.7]{HL}. Denote the product of $d_i$ copies of $S(w_i)U_{R^+\setminus R^{\tilde{w}}}$ as $\dot{U}(w_i,d_i)$. 

Therefore, one can define a natural $(U_{R^+\setminus R^{\tilde{w}}})^{d_0+\cdots+d_{l-1}}$-action on $  \dot{U}(w_0,d_0)\times \cdots \times \dot{U}(w_{l-1},d_{l-1})\times T^{\tilde{w}}U_{R^+\setminus R^{\tilde{w}}}$ as follows. Let $(u_1,u_2,\ldots,u_{d_0+\cdots+d_{l-1}})\in (U_{R^+\setminus R^{\tilde{w}}})^{d_0+\cdots+d_{l-1}}$. It takes $(g_1,g_2,\ldots,g_{d_0+\cdots+d_{l-1}},h)$ to $(g_1u_1^{-1},u_1g_2u_2^{-1},\ldots,u_{d_0+\cdots+d_{l-1}}h)$. Denote the orbit space of this action as $\tilde{U}(\tilde{w}_*,d_*)$. We define an $U_{R^+\setminus R^{\tilde{w}}}\times U_{R^+\setminus R^{\tilde{w}}}$-action on $\tilde{U}(\tilde{w}_*,d_*)$ as
\begin{equation*}
    (u,u')[g_1,g_2,\ldots,g_{d_0+\cdots+d_{l-1}},h]=[ug_1,g_2,\ldots,g_{d_0+\cdots+d_{l-1}},h\delta_G^d(u')^{-1}]
\end{equation*}

Since $U_{R^+\setminus R^{\tilde{w}}}S(w_i)U_{R^+\setminus R^{\tilde{w}}}=S(w_i)U_{R^+\setminus R^{\tilde{w}}}$, we have any element in $\tilde{U}(\tilde{w}_*,d_*)$ admits a unique expression as $[x_{0,1},\ldots,x_{0,d_0},\ldots,x_{l-1,d_{l-1}},tu]$ where $u\in U_{R^+\setminus R^w}$, $t\in T^{\tilde{w}}$ and $x_{i,j}\in S(w_i)$ for $0\leq i\leq l-1$ and $1\leq j\leq d_i$.

Let $[g_1^{(1)},\ldots,g_1^{(2n)},\ldots,\delta_G^{d-1}(g_d^{(1)}),\ldots,\delta_G^{d-1}(g_d^{(2n)}),\delta_G^d(t)u]$ be an arbitrary element in $\tilde{U}(\text{DG}(\tilde{b}),d)$ where $\delta_G^d(t)\in \delta_G^d(T^{\tilde{w}})= T^{\tilde{w}}$, $u\in U_{\delta^d(R^+\setminus R^{\tilde{w}})}=U_{R^+\setminus R^{\tilde{w}}}$, $g_i^{(2n)}\in S(w'w)$ and $g_i^{(j)}\in S(w')$ for $1\leq j\leq 2n-1$ by our choice of $d$. Let $(x_{0,1},\ldots,x_{0,d_0},\ldots,x_{l-1,d_{l-1}})$ be the image of $(g_1^{(1)},\ldots,g_1^{(2n)},\ldots,\delta_G^{d-1}(g_d^{(1)}),\ldots,\delta_G^{d-1}(g_d^{(2n)}))$ under $\mathscr{G}_b$ where $x_{i,j}\in S(w_i)$ for $0\leq i\leq l-1$ and $1\leq j\leq d_i$. Then we can define a map
\begin{equation*}
    \tilde{\mathscr{G}}_{\tilde{b}}:\tilde{U}(\text{DG}(\tilde{b}),d)\to \tilde{U}(w_*,d_*),
\end{equation*}
which sends $[g_1^{(1)},\ldots,g_1^{(2n)},\ldots,\delta_G^{d-1}(g_d^{(1)}),\ldots,\delta_G^{d-1}(g_d^{(2n)}),\delta_G^d(t)u]\in \tilde{U}(\text{DG}(\tilde{b}),d)$ to $[x_{0,1},\ldots,x_{0,d_0},\ldots,x_{l-1,d_{l-1}},\delta_G^d(t)u]$ as above.
\begin{lemma}
 The map $\tilde{\mathscr{G}}_{\tilde{b}}$ is bijective and is compatible with $U_{R^+\setminus R^{\tilde{w}}}\times U_{R^+\setminus R^{\tilde{w}}}$-actions on both sides.
\end{lemma}
\begin{proof}
    This map is well-defined and bijective due to the uniqueness of the above expressions and the bijectivity of $\mathscr{G}_{\tilde{b}}$.

    It remains to show $\tilde{\mathscr{G}}_{\tilde{b}}$ is compatible with $U_{R^+\setminus R^{\tilde{w}}}\times U_{R^+\setminus R^{\tilde{w}}}$-actions. It acts on the left hand side since our choice of $d$ ensures $U_{\delta^d(R^+\setminus R^{\tilde{w}})}=U_{R^+\setminus R^{\tilde{w}}}$. This follows from \cite[\S2.9]{HL} and the equations $U_{R^+\setminus R^{\tilde{w}}}S(w_i)U_{R^+\setminus R^{\tilde{w}}}=S(w_i)U_{R^+\setminus R^{\tilde{w}}}$ for all $i$.
\end{proof}

Based on above, we get a surjective map
\begin{equation*}
    \tilde{\mathscr{G}}_{\tilde{b}} \circ \tilde{\tau}_{\tilde{w}}^d :\tilde{U}(\text{DG}(\tilde{b}),T^{\tilde{w}},d)\to \tilde{U}(\tilde{w}_*,d_*),
\end{equation*}
which is compatible with $U_{R^+\setminus R^{\tilde{w}}}\times U_{R^+\setminus R^{\tilde{w}}}$-actions on both sides.

Finally, we will generalize the cross-sections.

\begin{proposition}\label{4.12A}
Let $\CC\in [\tilde{W}]$ and $\tilde{b}=b\delta$ be a good position braid representative of $(\tilde{w},\underline{\Theta})$ where $\tilde{w}=w\delta \in \CC$ and $\underline{\Theta}\subset \Gamma_{\tilde{w}}$ is the increasing complete sequence in $(0,\pi]\cup \{2n\pi\}$. The twisted conjugation map 
    \begin{equation*}
        \Xi_{\tilde{b}}^{\delta}:U_{R^+\setminus R^{\tilde{w}}}\times S_{Br}(\tilde{b}) \to S_{Br}(\tilde{b})U_{\delta(R^+\setminus R^{\tilde{w}})}
    \end{equation*}
    given by $(u,g)\mapsto ug\delta_G(u)^{-1}$ is a bijection.
\end{proposition}
\begin{proof}
    Let $h$ be arbitrary in $S_{Br}(\tilde{b})U_{\delta(R^+\setminus R^{\tilde{w}})}$. Consider the element $[h,\cdots,\delta_G^{d-1}(h)]$ in $\tilde{U}(\tilde{b}^d)$. By lemma \ref{4.7}, we can find a preimage $\eta_h$ of it in $\tilde{U}(\text{DG}(\tilde{b}),T^{\tilde{w}},r)$. By applying $\tilde{\mathscr{G}}_{\tilde{b}} \circ \tilde{\tau}_{\tilde{w}}^d$ to $\eta_h$, we get an element in $\tilde{U}(w_*,d_*)$, denoted as $\xi_h$. Recall our discussion above, one can uniquely write $\xi_h$ as $[x_{0,1},\ldots,x_{l-1,d_{l-1}},tu_\xi]$ with $u_\xi \in U_{R^+\setminus R^{\tilde{w}}}$, $t\in T^{\tilde{w}}$ and $(x_{0,1},\ldots,x_{l-1,d_{l-1}})\in S(w_0,d_0)\times \cdots \times S(w_{l-1},d_{l-1})$.

     Consider $\delta_G^{-d}(u_{\xi})h \delta_G^{-d+1}(u_\xi)^{-1}$ in $U_{R^+\setminus R^{\tilde{w}}}S_{Br}(\tilde{b})U_{\delta(R^+\setminus R^{\tilde{w}})}=S_{Br}(\tilde{b})U_{\delta(R^+\setminus R^{\tilde{w}})}$. By lemma \ref{4.6}, there is a unique decomposition $\delta_G^{-d}(u_{\xi})h \delta_G^{-d+1}(u_\xi)^{-1}=g' \delta_G(u')$ with $u'\in U_{R^+\setminus R^{\tilde{w}}}$ and $g'\in S_{Br}(\tilde{b})$. We define a map 
\begin{equation*}
    (\Xi_{\tilde{b}}^\delta)': S_{Br}(\tilde{b})U_{\delta(R^+\setminus R^{\tilde{w}})} \to U_{R^+\setminus R^{\tilde{w}}}\times S_{Br}(\tilde{b})
\end{equation*}
 as $(\Xi_{\tilde{b}}^\delta)'(h)=(\delta_G^{-d}(u_\xi) ^{-1},g')$. We claim that $\Xi_{\tilde{b}}^\delta$ and $(\Xi_{\tilde{b}}^\delta)'$ are inverse to each other. 

We first need to check that $(\Xi_{\tilde{b}}^\delta)'$ is well-defined. In other words, we need to show that $u_{\xi}$ is independent of the choice of $\eta_h$. Let $\eta_h'\neq \eta_h$ be another preimage of $[h,\ldots,\delta_G^{d-1}(h)]$ in $\tilde{U}(\text{DG}(\tilde{b}),T^{\tilde{w}},r)$. Write $\eta_h$ and $\eta_h'$ uniquely as 
\begin{eqnarray*}
    \eta_h&=&[t_1,g_1^{(1)},\ldots,g_1^{(2n)},\ldots,\delta_G^{d-1}(t_d),\delta_G^{d-1}(g_d^{(1)}),\ldots,\delta_G^{d-1}(g_d^{(2n)})u_{\eta}]\\
    \eta_h'&=&[t_1',g_1'^{(1)},\ldots,g_1'^{(2n)},\ldots,\delta_G^{d-1}(t_d'),\delta_G^{d-1}(g_d'^{(1)}),\ldots,\delta_G^{d-1}(g_d'^{(2n)})u_{\eta}'].
\end{eqnarray*}
Here $u_{\eta}\in U_{\delta^d(R^+\setminus R^{\tilde{w}})}=U_{R^+\setminus R^{\tilde{w}}}$, $t_i\in T^{\tilde{w}}$, $g_i^{(2n)}\in S(w'w)$ and $g_i^{(j)}\in S(w')$ for $1\leq j\leq 2n-1$. Similar facts hold for the decomposition of $\eta_h'$. By definition of $\tilde{\Upsilon}_{\tilde{b}}^d$, we have 
\begin{eqnarray*}
    [h,\delta_G(h),\ldots,\delta_G^{d-1}(h)]
    &=&[t_1\prod\limits_{j=1}^{2n}g_1^{(j)},\ldots,\delta_G^{d-1}(t_d\prod\limits_{j=1}^{2n}g_d^{(j)})u_{\eta}]\\
    &=&[t_1'\prod\limits_{j=1}^{2n}g_1'^{(j)},\ldots,\delta_G^{d-1}(t_d'\prod\limits_{j=1}^{2n}g_d'^{(j)})u_{\eta}'].
\end{eqnarray*}
 Now notice that $t_i\prod\limits_{j=1}^{2n}g_i^{(j)}$ and $t_i'\prod\limits_{j=1}^{2n}g_i'^{(j)}$ both lie in $S_{Br}(\tilde{b})$, we have $u_{\eta}=u_{\eta}'$ by the uniqueness of above form. Then by our definition of $\tilde{\mathscr{G}}_{\tilde{b}} \circ \tilde{\tau}_{\tilde{w}}^d$, we have $u_{\xi}=u_\eta$ is independent of the choice of $\eta_h$.

We first verify $(\Xi_{\tilde{b}}^\delta)'\circ \Xi_{\tilde{b}}^\delta=\id$. Let $h=ug\delta_G(u)^{-1}=\Xi_{\tilde{b}}^\delta(u,g)$ be an element in $S_{Br}(\tilde{b})U_{\delta(R^+\setminus R^{\tilde{w}})}$. Then we have 
\begin{eqnarray*}
    [h,\ldots,\delta_G^{d-1}(h)]&=&[ug\delta_G(u)^{-1},\delta_G(u)\delta_G(g)\delta_G^2(u)^{-1},\ldots,\delta_G^{d-1}(u)\delta_G^{d-1}(g)\delta_G^d(u)^{-1}]\\
    &=&[ug,\delta_G(g),\ldots,\delta_G^{d-1}(g)\delta_G^d(u)^{-1}]=(u,\delta_G^d(u))[g,\ldots,\delta_G^{d-1}(g)].
\end{eqnarray*}
Let $\eta_g$ be a preimage of $[g,\ldots,\delta_G^{d-1}(g)]$. By our choice of $d$, the map $\tilde{\Upsilon}_{\tilde{b}}^d$ is compatible with the action of $U_{R^+\setminus R^{\tilde{w}}}\times U_{R^+\setminus R^{\tilde{w}}}$. Then we have $\eta_h:=(u,\delta_G^d(u))\eta_g$ is a preimage of $[h,\ldots,\delta_G^{d-1}(h)]$. By applying $\tilde{\mathscr{G}}_{\tilde{b}} \circ \tilde{\tau}_{\tilde{w}}^d$ to $\eta_g$ and $\eta_h$, we have $\xi_h=(u,\delta_G^d(u))\cdot \xi_g$ since $\tilde{\mathscr{G}}_{\tilde{b}} \circ \tilde{\tau}_{\tilde{w}}^d$ is also compatible with the action of $U_{R^+\setminus R^{\tilde{w}}}\times U_{R^+\setminus R^{\tilde{w}}}$. 

Since $g\in S_{Br}(\tilde{b})$, we have $\xi_g$ can be written as $[x_{0,1}',\ldots,x_{l-1,d_{l-1}}',t']$ where $t'\in T^{\tilde{w}}$ and $(x_{0,1}',\ldots,x_{l-1,d_{l-1}}')\in S(w_0,d_0)\times \cdots \times S(w_{l-1},d_{l-1})$. Now we write $\xi_h=[x_{0,1},\ldots,x_{l-1,d_{l-1}},tu_{\xi}]$, then
\begin{eqnarray*}
    [x_{0,1},\ldots,x_{l-1,d_{l-1}},t]&=&(u,u_{\xi}\delta_G^d(u))[x_{0,1}',\ldots,x_{l-1,d_{l-1}}',t']\\
    &=&[ux_{0,1}',\ldots,x_{l-1,d_{l-1}}',t'\delta_G^d(u)^{-1}u_\xi^{-1}].
\end{eqnarray*}
 Now notice that $ux_{0,1}'\in U_{R^+\setminus R^{\tilde{w}}}S(w_0)=S(w_0)$ since $U^{w_0}=U$. Then $u_{\xi}=\delta_G^d(u)^{-1}$ due to the uniqueness of above expression, which implies that $\delta_G^{-d}(u_{\xi})=u^{-1}$. Therefore, we have $u_\xi h\delta_G(u_\xi)^{-1}=g$ which means $g'=g$ and thus $(\Xi_{\tilde{b}}^\delta)'(h)=(u,g)$. 

Conversely, let $h\in S_{Br}(\tilde{b})U_{\delta(R^+\setminus R^{\tilde{w}})}$ be any element. By definition of $(\Xi_{\tilde{b}}^\delta)'$, it suffices to show that $\delta_G^{-d}(u_{\xi})h \delta_G^{-d+1}(u_\xi)^{-1}$ is in $S_{Br}(\tilde{b})$. If this is true, then $g'=\delta_G^{-d}(u_{\xi})h \delta_G^{-d+1}(u_\xi)^{-1}$ and $\Xi_{\tilde{b}}^\delta (\delta_G^d(u_{\xi})^{-1},g')=h$. By lemma \ref{4.6}, we assume $h=g_0\delta_G(u_0)$ for some $u_0\in U_{R^+\setminus R^{\tilde{w}}}$ and $g_0\in S_{Br}(\tilde{b})$. We have 
\begin{eqnarray*}
    [h,\ldots,\delta_G^{d-1}(h)]&=&[g_0\delta_G(u_0),\ldots,\delta_G^{d-1}(g_0)\delta_G^d(u_0)]\\
    &=&[g_0,\delta_G(u_0g_0),\ldots,\delta_G^{d-2}(u_0g_0),\delta_G^{d-1}(u_0g_0)\delta_G^d(u_0)]\\
    &=&[g_0,\delta_G(g_1)\delta_G^2(u_1),\ldots,\delta_G^{d-2}(g_1)\delta_G^{d-1}(u_1),\delta_G^{d-1}(g_1)\delta_G^d(u_1u_0)]\\
    &=& [g_0,\delta_G(g_1),\delta_G^2(u_1g_1),\ldots,\delta_G^{d-2}(u_1g_1),\delta_G^{d-1}(u_1g_1)\delta_G^d(u_1u_0)]\\
    &=& \cdots \\
    &=&[g_0,\delta_G(g_1),\ldots,\delta_G^{d-2}(g_{d-2}),\delta_G^{d-1}(g_{d-1})\delta_G^d(u_{d-1}\cdots u_1u_0)].
\end{eqnarray*}
Here we have $u_i\in U_{R^+\setminus R^{\tilde{w}}}$ and $g_i\in S_{Br}(\tilde{b})$ satisfying that $u_ig_i=g_{i+1}\delta_G(u_{i+1})$ for $0\leq i \leq d-2$. By our discussion above we have $u_\xi=\delta_G^d(u_{d-1}\cdots u_1u_0)$. Then 
\begin{eqnarray*}
    \delta_G^{-d}(u_{\xi})h \delta_G^{-d+1}(u_\xi)^{-1}&=&u_{d-1}\cdots u_1u_0g_0\delta_G(u_0)\delta_G(u_0)^{-1}\cdots \delta_G(u_{d-1})^{-1}\\
    &=& u_{d-1}\cdots u_1g_1\delta_G(u_1)\delta_G(u_1)^{-1}\cdots \delta_G(u_{d-1})^{-1}\\
    &=& \cdots \\
    &=& u_{d-1}g_{d-1}.
\end{eqnarray*}
 Meanwhile, one notice that 
 \begin{equation*}
   (u_0,1)[g_0,\delta_G(g_1)\ldots,\delta_G^{d-1}(g_{d-1})]=[g_1,\delta_G(g_2),\ldots,\delta_G^{d-2}(g_{d-1}),\delta_G^{d-1}(u_{d-1}g_{d-1})].   
 \end{equation*}
Let $\eta$ be a preimage of $[g_0,\delta_G(g_1)\ldots,\delta_G^{d-1}(g_{d-1})]$ under $\tilde{\Upsilon}_{\tilde{b}}^d$. Then the image of $\eta$ under $\tilde{\mathscr{G}}_{\tilde{b}} \circ \tilde{\tau}_{\tilde{w}}^d$ can be written as $[x_{0,1}'',\ldots,x_{l-1,d_{l-1}}'',t'']$ where $t''\in T^{\tilde{w}}$ and $(x_{0,1}'',\ldots,x_{l-1,d_{l-1}}'')\in S(w_0,d_0)\times \cdots \times S(w_{l-1},d_{l-1})$. 

Let $\eta'$ be a preimage of $[g_1,\delta_G(g_2),\ldots,\delta_G^{d-2}(g_{d-1}),\delta_G^{d-1}(u_{d-1}g_{d-1})]$ under $\tilde{\Upsilon}_{\tilde{b}}^d$. Then the image of $\eta'$ under $\tilde{\mathscr{G}}_{\tilde{b}} \circ \tilde{\tau}_{\tilde{w}}^d$ is $(u_0,1)[x_{0,1}'',\ldots,x_{l-1,d_{l-1}}'',t'']=[u_0x_{0,1}'',\ldots,x_{l-1,d_{l-1}}'',t'']$ and $u_0x_{0,d_0}''\in S(w_0)$ since $U^{w_0}=U$. Therefore, we have $\delta_G^{d-1}(u_{d-1}g_{d-1})\in \delta_G^{d-1}(S_{Br}(\tilde{b}))$. Thus $\delta_G^{-d}(u_{\xi})h \delta_G^{-d+1}(u_\xi)^{-1}=u_{d-1}g_{d-1}\in S_{Br}(\tilde{b})$ and our claim is proved. 

Therefore, we have $\Xi_{\tilde{b}}^\delta$ is a bijection with inverse $(\Xi_{\tilde{b}}^\delta)'$.
\end{proof}

One can then prove the similar result for enlarged $G$-slices.

\begin{corollary}\label{4.14a}
    Keep the notations above and further require that $\tilde{b}=b(\tilde{w})$ is actually a good position braid representative of $\CC$. Then the twisted conjugation map
    \begin{equation*}
        \widehat{\Xi}_{\tilde{b}}^\delta: U_{R^+\setminus R^{\tilde{w}}}\times \widehat{S_{Br}}(\tilde{b})\to  \widehat{S_{Br}}(\tilde{b})U_{\delta(R^+\setminus R^{\tilde{w}})}
    \end{equation*}
    given by $(u,g)\mapsto ug\delta_G(u)^{-1}$ is a bijection.
\end{corollary}
\begin{proof}
    First we recall that $\widehat{S_{Br}}(\tilde{b})=S_{Br}(\tilde{b})$ if $R^{\tilde{w}}$ is trivial. We only need to prove for the case when $R^{\tilde{w}}$ is nontrivial. In this case, we have $\tilde{b}=\underline{w'}\cdot \underline{w'w}\delta$ where $w'$ is the longest element in $W'<W$ corresponding to $R^{\tilde{w}}$.
    
    We first need to show that this map is well-defined. In other word, one need to show that $ug\delta_G(u)^{-1}$ always lie in $\widehat{S_{Br}}(\tilde{b})U_{\delta(R^+\setminus R^{\tilde{w}})}$. Recall that $\widehat{S_{Br}}(\tilde{b})=L_{\tilde{w}}U^w\dot{w}$ where $L_{\tilde{w}}$ is generated by $T^{\tilde{w}}$ and the root subgroups corresponding to $R^{\tilde{w}}$, then we have
    \begin{equation*}
        U_{R^+\setminus R^{\tilde{w}}}L_{\tilde{w}}U^w\dot{w}U_{\delta(R^+\setminus R^{\tilde{w}})}=L_{\tilde{w}}U_{R^+\setminus R^{\tilde{w}}}U^w\dot{w}U_{\delta(R^+\setminus R^{\tilde{w}})}=L_{\tilde{w}}U^w\dot{w}U_{\delta(R^+\setminus R^{\tilde{w}})}
    \end{equation*}
    and thus $\widehat{\Xi}_{\tilde{b}}^\delta$ is well-defined. Moreover, one shall notice that the proof of lemma \ref{4.6} generates to $\widehat{S_{Br}}(\tilde{b})$ as our proof there actually investigated root subgroups in $L_{\tilde{w}}U^w$. Thus any element in $\widehat{S_{Br}}(\tilde{b})U_{\delta(R^+\setminus R^{\tilde{w}})}$ can be uniquely written as a product of elements in $\widehat{S_{Br}}(\tilde{b})$ and $U_{\delta(R^+\setminus R^{\tilde{w}})}$.

    Let $h$ be arbitrary in $\widehat{S_{Br}}(\tilde{b})U_{\delta(R^+\setminus R^{\tilde{w}})}$. We write $h=lu_1\dot{w}u_2$ with $l\in L_{\tilde{w}}$, $u_1\in U^w$ and $u_2\in U_{\delta(R^+\setminus R^{\tilde{w}})}$. Now for $l$, by definition of $L_{\tilde{w}}$, we have $l$ must lie in $T^{\tilde{w}}U_{(R^{\tilde{w}})^-}U_{(R^{\tilde{w}})^+}\cdots U_{(R^{\tilde{w}})^-}U_{(R^{\tilde{w}})^+} $ with a finite number of copies of $U_{(R^{\tilde{w}})^-}U_{(R^{\tilde{w}})^+}$. Recall our definition of $G$-braid slices, then by above there exists a good position braid representative $\tilde{b}'$ of some $(\tilde{w},\underline{\Theta'})$ such that $lu_1\dot{w}\in S_{Br}(\tilde{b}')$. Moreover, we assume $\tilde{b}'$ is of minimal length in $B^+(\tilde{W})$ with above property (i.e. the braid with smallest number of parts in its Deligne-Garside form).  
    Then we have $h\in S_{Br}(\tilde{b}')U_{\delta(R^+\setminus R^{\tilde{w}})}$. By our proof of proposition \ref{4.12A}, we then have $(\Xi_{\tilde{b}'}^\delta)'(h)=(u,g)$ for some $u\in U_{R^+\setminus R^{\tilde{w}}}$ and $g\in S_{Br}(\tilde{b}')\subset \widehat{S_{Br}}(\tilde{b})$ satisfying $ug\delta_G(u)^{-1}=h$. 

    Now we define a map 
    \begin{equation*}
        (\widehat{\Xi}_{\tilde{b}}^\delta)': \widehat{S_{Br}}(\tilde{b})U_{\delta(R^+\setminus R^{\tilde{w}})} \to  U_{R^+\setminus R^{\tilde{w}}}\times \widehat{S_{Br}}(\tilde{b}),
    \end{equation*}
where $(\widehat{\Xi}_{\tilde{b}}^\delta)'(h)$ is exactly the pair $(u,g)$ above. We first show this map is well-defined. Let $h$ and $\tilde{b}'$ be as above. Let $\tilde{b}''$ be a good position braid representative of $(\tilde{w},\underline{\Theta''})$ for another sequence $\underline{\Theta''}$ such that $lu_1\dot{w}\in S_{Br}(\tilde{b}'')$. Let $r(\underline{\Theta'})$ and $r(\underline{\Theta''})$ be the irredundant subsequence respectively. By our assumption that $R^{\tilde{w}}$ is nontrivial and the completeness of $\underline{{\Theta'}}$ and $\underline{\Theta''}$, we must have the largest angles in  $r(\underline{\Theta'})$ and $r(\underline{\Theta''})$ are both multiples of $2\pi$. We may then assume that $\tilde{b}'=(\underline{w'})^{2n_1-1}\cdot \underline{w'w}\delta$ and $\tilde{b}''=(\underline{w'})^{2n_2-1}\cdot \underline{w'w}\delta$.
By our assumption on $\tilde{b}'$ we know $n_1<n_2$ and thus $S_{Br}(\tilde{b}')\subset S_{Br}(\tilde{b}'')$. Then by applying proposition \ref{4.12A} to $(\Xi_{\tilde{b}''}^\delta)'$ we must have $(\Xi_{\tilde{b}'}^\delta)'(h)=(\Xi_{\tilde{b}''}^\delta)'(h)$. Therefore, the map $(\widehat{\Xi}_{\tilde{b}}^\delta)'$ is well-defined.

It remains to check that $\widehat{\Xi}_{\tilde{b}}^\delta$ and  
$(\widehat{\Xi}_{\tilde{b}}^\delta)'$ are inverse to each other. By definition it is clear that $\widehat{\Xi}_{\tilde{b}}^\delta\circ (\widehat{\Xi}_{\tilde{b}}^\delta)'$ is identity. As for $(\widehat{\Xi}_{\tilde{b}}^\delta)'\circ \widehat{\Xi}_{\tilde{b}}^\delta(u,g)=(u,g)$, this reduces to proposition \ref{4.12A} by replacing $(\widehat{\Xi}_{\tilde{b}}^\delta)'$ with $(\Xi_{\tilde{b}'}^\delta)'$ where $\tilde{b}'$ is a good position braid representative of some $(\tilde{w},\underline{\Theta'})$ such that $g\in S_{Br}(\tilde{b}')$.
\end{proof}

Now we focus on case \RNum{1}. One can prove the first part of Theorem \ref{B}.

\textbf{Proof of Theorem \ref{B}.(1)} 
It is clear that $\Xi_{\tilde{b}}^\delta$ is a morphism of varieties by definition. As for $(\Xi_{\tilde{b}}^\delta)'$, we first notice that the map from $U_{R^+\setminus R^{\tilde{w}}}S_{Br}(\tilde{b})$ to $U_{R^+\setminus R^{\tilde{w}}}$ which maps $h$ to $u_{\xi}$ is actually a morphism of varieties since it is determined by a series of conjugations by our computations above. Now in case \RNum{1}, we have $u_\xi\mapsto \delta_G^{-d}(u_\xi) ^{-1}$ is also a morphism of varieties since $\delta_G^{-d}$ is given by the conjugation of some element in $D$. Then $(\Xi_{\tilde{b}}^\delta)'$ is a morphism of varieties. Therefore, we have $\Xi_{\tilde{b}}^\delta$ is an isomorphism of varieties.

Now going back to the slice $S_{Br}^D(\tilde{b})=S_{Br}(\tilde{b})\tilde{g}_D$. Elements of $S_{Br}^D(\tilde{b})$ are of the form $g\tilde{g}_D$ for some $g\in S_{Br}(\tilde{b})$. By definition of $\tilde{g}_D$, we have $\Xi_{\tilde{b}}(u,g\tilde{g}_D)=\Xi_{\tilde{b}}^\delta(u,g)\tilde{g}_D$. Similarly, one can construct its inverse $(\Xi_{\tilde{b}})'$ by the composition of multiplication of $\tilde{g}_D^{-1}$ and $(\Xi_{\tilde{b}}^\delta)'$. Therefore, we have $\Xi_{\tilde{b}}$ is also an isomorphism of varieties.
\qed

\begin{remark}\label{4.15a}
   (1) One shall notice that in this section we do not require $\CC\in [\tilde{W}]$ to be the ``most elliptic'' conjugacy class in $\Phi^{-1}(\CO)$ for some unipotent $G$-orbit $\CO$. This implies that the first part of Theorem \ref{B} actually works for any $\CC$.
 
   (2) In case \RNum{2}, the bijective map $\Xi_{\tilde{b}}^\delta$ is not an isomorphism of varieties anymore. This is because the definition of $(\Xi_{\tilde{b}}^\delta)'$ involves the inverse of $\delta_G=F$, which is not an isomorphism of varieties.

   (3) It is direct from above proof that the enlarged slice $\widehat{S_{Br}^D}(\tilde{b})$ also have the above cross section property.

   (4) At the beginning of this subsection we fixed a special form of the representatives $\dot{w}$ for all $w\in W$. However, the above results actually hold for any representative $\dot{w}$ by applying multiplications in $G$.
\end{remark}

\subsection{Transversality} In this section, we always assume that $\CC$ is the ``most elliptic'' $W$-conjugacy class in $\Phi^{-1}(\CO)$ for some $\CO\in [\tilde{G}_u]$ if there is no specific clarification. We prove the second part of Theorem \ref{B}. However, we start by investigating the enlarged $G$-slice $\widehat{S_{Br}}(\tilde{b})$. We start with a similar result as first part of Theorem \ref{B} for $\widehat{S_{Br}}(\tilde{b})$. Note this result does not need the ``most elliptic'' assumption.

\begin{corollary}\label{4.10}
Keep the notations in corollary \ref{4.14a}. The twisted conjugation map 
    \begin{equation*}
         (\widehat{\Xi}_{\tilde{b}}^\delta)^-:U_{-(R^+\setminus R^{\tilde{w}})}\times \widehat{S_{Br}}(\tilde{b})\to U_{-(R^+\setminus R^{\tilde{w}})}\widehat{S_{Br}}(\tilde{b})U_{\delta(-(R^+\setminus R^{\tilde{w}}))}
    \end{equation*}
given by $(u,g)\mapsto ug\delta_G(u)^{-1}$ is a bijection. Moreover, it is an isomorphism of varieties in case \RNum{1}.
\end{corollary}
\begin{proof}
    Recall that there is an anti-automorphism $\iota$ of $G$ which interchanges root subgroups $U_{\alpha}$ and $U_{-\alpha}$. By our discussion in \S4.1, we have
    \begin{eqnarray*}
        \widehat{S_{Br}}(\tilde{b})&=&L_{\tilde{w}}U^w\dot{w}=L_{\tilde{w}}U_{\Inv(w^{-1})}\dot{w}\\
        &=& L_{\tilde{w}}\dot{w}U_{-\Inv(w)}\\
        &=& \dot{w}\cdot \delta_G(L_{\tilde{w}})U_{-\Inv(w)} .
    \end{eqnarray*}
    By applying $\delta_G^{-1}$ to it, we have
    \begin{equation*}
        \delta_G^{-1}(\widehat{S_{Br}}(\tilde{b}))=\delta_G^{-1}(\dot{w})\cdot L_{\tilde{w}}U_{\delta^{-1}(-\Inv(w))}.
    \end{equation*}
    Now consider the image $\iota(\delta_G^{-1}(\widehat{S_{Br}}(\tilde{b})))$, we then have
    \begin{equation*}
        \iota(\delta_G^{-1}(\widehat{S_{Br}}(\tilde{b})))=U_{\Inv((\delta^{-1}(w^{-1}))^{-1})}L_{\tilde{w}}\cdot \iota(\delta_G^{-1}(\dot{w})).
    \end{equation*}
   Notice that $\iota(\delta_G^{-1}(\dot{w}))$ is a representative of $\delta^{-1}(w^{-1})$. Meanwhile, we have $(w\delta)^{-1}=\delta^{-1}(w^{-1})\cdot \delta^{-1}$ and $\tilde{b}':=\underline{w'}\cdot \underline{w'\delta^{-1}(w^{-1})}\delta^{-1}$ is the Deligne-Garside normal form of a good position braid representative of $\delta^{-1}(w^{-1})\cdot \delta^{-1}$. Therefore, by above equation one can check that $\iota(\delta_G^{-1}(\widehat{S_{Br}}(\tilde{b})))$ is the enlarged $G$-slice associated to $\tilde{b}'$. Since the twist part of $\tilde{b}'$ is $\delta^{-1}$, by applying remark \ref{4.15a} we have the twisted conjugation map
    \begin{equation*}
         \Xi_{\tilde{b}'}^{\delta^{-1}}:U_{R^+\setminus R^{\tilde{w}}}\times \iota(\delta_G^{-1}(\widehat{S_{Br}}(\tilde{b})))\to \iota(\delta_G^{-1}(\widehat{S_{Br}}(\tilde{b})))U_{\delta^{-1}(R^+\setminus R^{\tilde{w}})}
    \end{equation*}
    given by $(u,g)\mapsto ug\delta_G^{-1}(u)^{-1}$ is an isomorphism of varieties.   

    Notice that $\iota(\delta_G^{-1}(\widehat{S_{Br}}(\tilde{b})))U_{\delta^{-1}(R^+\setminus R^{\tilde{w}})}=U_{R^+\setminus R^{\tilde{w}}}\iota(\delta_G^{-1}(\widehat{S_{Br}}(\tilde{b})))U_{\delta^{-1}(R^+\setminus R^{\tilde{w}})}$. Thus we have $\delta_G\circ \iota(\iota(\delta_G^{-1}(\widehat{S_{Br}}(\tilde{b})))U_{\delta^{-1}(R^+\setminus R^{\tilde{w}})})=U_{-(R^+\setminus R^{\tilde{w}})}\widehat{S_{Br}}(\tilde{b})U_{\delta(-(R^+\setminus R^{\tilde{w}}))}$. Therefore, by applying $\delta_G\circ \iota$ to $\Xi_{\tilde{b}'}^{\delta^{-1}}$, the corollary is proved.
\end{proof}

From now on we only consider case \RNum{1}. We prove that the enlarged slice $\widehat{S_{Br}^D}(\tilde{b})$ does intersect the unipotent $G$-orbit $\CO$. Let $\CB$ be the variety of Borel subgroups of $G\subset \tilde{G}$. There is a natural diagonal action of $G$ on $\CB\times \CB$ with $W$ as the index set of the orbits. Let $O_w$ be the index set for $w\in W$. The twist $\delta$ defined by $D$ satisfies
\begin{equation*}
    (B',B'')\in O_w \quad \Longrightarrow \quad (gB'g^{-1},gB''g^{-1})\in O_{\delta(w)},
\end{equation*}
for any $g\in D$. Now define
\begin{equation*}
    \fkB_w^D=\{(g,B')\in D\times \CB \mid (B',gB'g^{-1})\in O_w\}.
\end{equation*}
For our purpose, instead of showing that every slice intersects $\CO$, we prove the following lemma. This lemma does not require the ``most elliptic'' assumption.

\begin{lemma}\label{4.15}
    Keep the notations in corollary \ref{4.14a}. There exists a representative $\dot{w}\in G$ such that the corresponding slice $\widehat{S_{Br}^D}(\tilde{b})=L_{\tilde{w}}U^w\dot{w}\tilde{g}_D$ intersects $\CO$.
\end{lemma}
\begin{proof}  
We first prove that the intersection of $\CO$ and the Bruhat cell $BwB$ is nonempty. Recall the definition of Lusztig's map $\Phi$ in \cite{Lu1} and \cite{Lu21}, we have $\CO\cap BuB\tilde{g}_D$ is nonempty for any minimal length element $\tilde{u}=u\delta\in \CC$.

Consider any two element $w_1,w_2\in W$ with $w_2=u_0w_1\delta(u_0)^{-1}$ for some $u_0\in W$ such that $l(w_2)\geq l(w_1)$. Suppose that $\CO\cap Bw_1B\tilde{g}_D$ is nonempty, then there exists $(g,B')\in \fkB_{w_1}^D$ with $g\in \CO$. We know $(B',gB'g^{-1})\in O_{w_1}$. There exists $B''\in \CB$ such that $(B'',B')\in O_u$. By property of $\delta$ we have $(gB'g^{-1},gB''g^{-1})\in O_{\delta(u)^{-1}}$ and thus $(B'',gB''g^{-1})\in O_{w_2}$. Then we have $\CO\cap Bw_2B\tilde{g}_D\neq \emptyset$. Now by combining this with \cite[Theorem 3.1]{HN}, we have $\CO\cap BwB\tilde{g}_D$ is nonempty.

Let $h\in \CO\cap BwB\tilde{g}_D$. We know $BwB\tilde{g}_D=U^w w\tilde{g}_DB$. Then we have $h\in U^w\dot{w}\tilde{g}_DU$ for some representative $\dot{w}\in G$. By our definition of the slice and the equation $U^w\dot{w}\tilde{g}_DU=U^w\dot{w}\tilde{g}_DU_{(R^{\tilde{w}})^+}U_{R^+\setminus R^{\tilde{w}}}=U^wU_{(R^{\tilde{w}})^+}\dot{w}\tilde{g}_DU_{R^+\setminus R^{\tilde{w}}}\subset L_{\tilde{w}}U^w\dot{w}\tilde{g}_DU_{R^+\setminus R^{\tilde{w}}}$. Then we have $\CO\cap \widehat{S_{Br}^D}(\tilde{b})U_{R^+\setminus R^{\tilde{w}}}$ is nonempty for some representative $\dot{w}$. Then $\CO\cap \widehat{S_{Br}^D}(\tilde{b})$ is also nonempty by Theorem \ref{B}.(1) and remark \ref{4.15a}.(3).
\end{proof}

Finally, before we prove the second part of Theorem \ref{B}, we first prove the same result for the enlarged slice $\widehat{S_{Br}^D}(\tilde{b})$.

\begin{proposition}\label{4.18a}
    Keep the notations in corollary \ref{4.14a}. The enlarged slice $\widehat{S_{Br}^D}(\tilde{b})$ intersects any $G$-conjugacy class it meets transversally. In particular, it is a transversal slice of $\CO=\Phi(\CC)$.
\end{proposition}
\begin{proof}
     Let $\fkn,\fkn^-$ be the Lie algebras of $U,U^-$ respectively and $\fkt$ be the corresponding Cartan subalgebra. Recall \cite[Lemma 1.1]{Lu0}, it suffices to show that 
    \begin{equation*}
        (\text{Id}-\Ad(g))(\fkg)+(\fkn_{(R^{\tilde{w}})^+}+\fkt^{\tilde{w}} + \fkn_{(R^{\tilde{w}})^-}+\fkn_{\Inv(w^{-1})})=\fkg,
    \end{equation*}
    for any $g\in  \widehat{S_{Br}^D}(\tilde{b})$.
    Here $\fkn_{R'}$ corresponds to $U_{R'}$ for any $R'<R$ and $\fkt^{\tilde{w}}$ is the fixed point subset of $\tilde{w}$ in $\fkt$. 

    We know $U_{-(R^+\setminus R^{\tilde{w}})}\widehat{S_{Br}}(\tilde{b})U_{\delta(-(R^+\setminus R^{\tilde{w}}))}=\widehat{S_{Br}}(\tilde{b})U_{\delta(-(R^+\setminus R^{\tilde{w}}))}$. Similar as the proof of Theorem \ref{B}.(1), corollary \ref{4.10} also implies that the conjugation map
    \begin{equation*}
        (\Xi_{\tilde{b}})^-:U_{-(R^+\setminus R^{\tilde{w}})}\times \widehat{S_{Br}^D}(\tilde{b})\to \widehat{S_{Br}^D}(\tilde{b})U_{-(R^+\setminus R^{\tilde{w}})}
    \end{equation*}
    is an isomorphism of varieties. Therefore, by above isomorphism and remark \ref{4.15a}.(3), we have $\fkn+\fkn^{-}+\fkt^{\tilde{w}}$ is contained in the left hand side of the above equation we need. Then it remains to show that $\fkt_{\tilde{w}}$ is also contained in it, where $\fkt_{\tilde{w}}$ is the orthogonal complement to $\fkt^{\tilde{w}}$. Indeed, for this one just have to notice that 
    \begin{equation*}
            ((\text{Id})-\Ad(g)(\fkt_{\tilde{w}}))\vert_{\fkt_{\tilde{w}}}=(\text{Id}-\Ad(\dot{w}\tilde{g}_{D}))(\fkt_{\tilde{w}})
    \end{equation*}
    since $g\in \widehat{S_{Br}^D}(\tilde{b})$.
    Furthermore, we have $(\text{Id}-\Ad(\dot{w}\tilde{g}_{D}))(\fkt_{\tilde{w}})=\fkt_{\tilde{w}}$ by definition of $\fkt_{\tilde{w}}$. Therefore, by above discussion the enlarged slice $\widehat{S_{Br}^D}(\tilde{b})$ intersects any $G$-conjugacy class it meets transversally.

    Finally, by lemma \ref{4.15} we know there is a representative $\dot{w}$ whose corresponding enlarged slice $\widehat{S_{Br}^D}(\tilde{b})$ intersects $\CO$. Meanwhile, notice that
    \begin{equation*}
        \dim \widehat{S_{Br}^D}(\tilde{b})=l(\tilde{b})+\dim T^{\tilde{w}}=l_{\text{good}}(\CC)+\dim T^{\tilde{w}}
    \end{equation*}
    by our construction. Since $\CC$ is the most elliptic $W$-conjugacy class in $[\tilde{W}]$, by proposition \ref{3.5} we have $\dim \widehat{S_{Br}^D}(\tilde{b})=\codim_{\tilde{G}}\CO$. This proves the slice condition.
\end{proof}

\textbf{Proof of Theorem \ref{B}.(2)}
    Let $\tilde{b}'$ be the good position braid representative of $\tilde{w}$ that gives the enlarged slice $\widehat{S_{Br}^D}(\tilde{b}')$. We then have $S_{Br}^D(\tilde{b})$ is an open dense subvariety of $\widehat{S_{Br}^D}(\tilde{b}')$ given the same choice of representative $\dot{w}$. Similar as lemma \ref{4.15}, one can prove there is a representative $\dot{w}$ whose corresponding slice $S_{Br}^D(\tilde{b})$ intersects $\CO$. The remainings are just direct from proposition \ref{4.18a}.
\qed

\begin{remark}\label{4.11}
   Recall that Lusztig's map $\Phi$ is surjective. The above theorem tells that we can construct transversal slices for any unipotent orbit in $\tilde{G}$.
\end{remark}

\begin{example}
    Recall our second example in \ref{4.3}. In this case, the corresponding slice is a transversal slice of $\Phi^{(2)}(\CC)$ in characteristic 2. However, it does not satisfy the slice condition in any other characteristics due to our dimension check.
\end{example}

\section{Applications to affine Springer fibers}
In this section, we focus on $G$ instead of $\tilde{G}$, which is the connected case. One may assume that the base field $\Bk$ is $\BC$ for convenience. (We only need the characteristic of $\Bk$ to be $0$).

\subsection{Affine Springer fibers and the function $\delta$} We first recall some facts and the paper \cite{Y}. Keep the settings in \S 3.3 for $G,B,T$ and let $\fkg,\fkt$ be the Lie algebras of $G,T$. Let $F=\BC ((\varpi))$ be the field of formal Laurent series and $\BC[[\varpi]] $ be its ring of integers. For any scheme $X$ over $\BC$, let $LX$ be the formal loop space of $X$ and $L^+ X$ be the formal arc space of $X$, i.e., $LX(R):=X(R((\varpi)))$ and $L^+X(R):=X(R[[\varpi]])$ for any $\BC$-algebra $R$. Consider the loop group $LG$ and the positive loop group $L^+G\subset LG$. The affine Grassmannian of $G$ is $\Gr=LG/L^+G$. Let $\mathbf{I}\subset L^+G$ be the Iwahori subgroup corresponding to $B$. The affine flag variety of $G$ is $\Fl=LG/ \mathbf{I}$.

Consider $L\fkg$ and $L^+\fkg$. Let $\gamma\in L\fkg$ be regular semisimple. Kazhdan and Lusztig defined the affine Springer fibers of $\gamma$ in \cite[\S 2]{KL} as follows.
\begin{eqnarray*}
   \Gr_{\gamma} &=& \{g L^+ G\in \Gr \mid \Ad(g^{-1})\gamma \in L^+g\}, \\
   \Fl_{\gamma} &=& \{g I \in \Fl \mid \Ad(g^{-1})\gamma \in \Lie(\mathbf{I})\}.
\end{eqnarray*}   
Here $\Ad(g^{-1})$ means the adjoint action of $g^{-1}$ on $L\fkg$. We know these are closed sub-ind-schemes of $\Gr$ and $\Fl$ respectively. Throughout this paper, we identify them with their underlying reduced ind-scheme.

Following the notion in \cite[\S1.3]{Y}, let $L^{\heart}\fkg$ be the subset of $L\fkg$ consisting of topologically nilpotent regular semisimple elements. In this paper we assume $\gamma \in L^{\heart}\fkg$. The centralizer $G_{\gamma}=\{g\in LG\mid Ad(g^{-1})\gamma=\gamma\}$ is a loop group $LT_{\gamma}$ of some maximal torus $T_{\gamma}$ of $G(F)$. Recall by \cite[\S 1, Lemma 2]{KL}, there is a bijection between the $LG$-conjugacy classes of maximal torus of $G(F)$ and the conjugacy classes of $W$. If $G_{\gamma}$ corresponds to $\CC$ under the bijection, we say $T_{\gamma}$ is of type $\CC$ and $\gamma$ is of type $\CC$. For any $\gamma$ of type $\CC$, there is a formula for the dimension of the affine Springer fibers of $\gamma$ as
\begin{equation*}
    \dim \Fl_{\gamma}=\dim  \Gr_{\gamma} =\frac{1}{2}(\val(\Delta(\gamma))-(\dim \fkt -\dim \fkt^{\CC})),
\end{equation*}
where $\Delta$ is the discriminant function and $\dim \fkt^{\CC}:=\dim \fkt^w$ for any $w\in \CC$. It is conjectured by Kazhdan and Lusztig in \cite[\S 0]{KL} and proved by Bezrukavnikov in \cite[Proposition]{B}. 

Denote the subset of $L^{\heart}\fkg$ consisting of elements of type $\CC$ as $(L^{\heart}\fkg)_{\CC}$. It is natural to consider the dimensions of affine Springer fibers of all $\gamma\in (L^{\heart}\fkg)_{\CC}$. In particular, define the function $\delta$ as $\delta_{\CC}=\min \{ \dim \Gr_{\gamma} \mid \gamma\in (L^{\heart}\fkg)_{\CC} \}$. Following \cite{Y}, an element $\gamma$ of type $\CC$ is called \textit{shallow} of type $\CC$ if $\dim \Gr_{\gamma}=\delta_{\CC}$. The set of shallow elements of type $\CC$ is denoted as $(L^{\heart}\fkg)_{\CC}^{sh}$. Let $LT_{\gamma}$ be any loop torus of type $\CC$ with loop Lie algebra $L\fkt_{\gamma}$, we denote its subset of shallow elements as $(L\fkt_{\gamma})^{sh}$.

\subsection{Root valuation and good position braid representatives}
Recall Yun deduced the following formula for $\delta_{\CC}$ when $\CC$ is elliptic by investigating the root valuation of any shallow element.
\begin{theorem}\cite[Theorem 2.3]{Y}
Let $\CC\in [W]$ be elliptic. We have
\begin{equation*}
   \delta_{\CC}=\frac{l_{\min}(\CC)-r}{2},
\end{equation*}   
where $r$ is the semisimple rank of $G$.
\end{theorem}   

This equation no longer holds for non-elliptic conjugacy classes $\CC$. However, we may replace minimal length elements with good position braid representatives . Before we generalize this result, we need several technical lemmas.

Let $\CC$ be any conjugacy class in $[W]$ and $b=b(w)\in B^+(W)$ be a good position braid representative of $\CC$. Let $\underline{\Theta}=(\theta_1,\ldots,\theta_m)$ be the increasing complete sequence in $\Gamma_w\cap ((0,\pi] \cup \{2\pi\})$ corresponding to $w$ and $r(\underline{\Theta})=(\theta_{i_1},\ldots,\theta_{i_l})$ be its irredundant subsequence.

\begin{lemma}\label{5.2}
    Let $\alpha$ be a positive root and $\frac{2a\pi}{n}\in \Gamma_w $ with $a,n$ coprime. Suppose that $(\alpha, V_w^{\frac{2a\pi}{n}})=0$, then for any $b\in \BN$ such that $b,n$ coprime, we have $(\alpha, V_w^{\frac{2b\pi}{n}})=0$.
\end{lemma}

\begin{proof}
     Consider the complex eigenspaces $(V_{\BC})_{w}^{\frac{2a\pi}{n}}$ and $(V_{\BC})_{w}^{\frac{2b\pi}{n}}$ of $w$. It suffices to show $(\alpha, (V_{\BC})_{w}^{\frac{2b\pi}{n}})=0$ while we know $(\alpha, (V_{\BC})_{w}^{\frac{2a\pi}{n}})=0$.
    
    Let $M_w$ be the matrix of the action of $w$ on $V_{\BC}$. It is clear that $M_w$ is a rational matrix with respect to the basis $\{\alpha_i\}_{i\in I}$. Now consider $M_w-e^{i\frac{2a\pi}{n}} \cdot \text{I}$, the entries of this matrix all lie in $\BQ(e^{i\frac{2a\pi}{n}})$. Now for any $v\in (V_{\BC})_{w}^{\frac{2a\pi}{n}}$, we have $v$ lies in the kernel of $M_w-e^{i\frac{2a\pi}{n}} \cdot \text{I}$. Therefore, there exists a collection of rational polynomials $\{f_{v,j} \in \BQ[x]\}_{j\in I}$ such that $v=\sum\limits_{j\in I} f_{v,j}(e^{i\cdot \frac{2a\pi}{n}}) \alpha_j$. For any $j$, write $w(\alpha_j)=\sum\limits_{k\in I}\lambda_{j,k}\alpha_k$ with $\lambda_{j,k}\in \BQ$. By definition of $v$, we have $\sum\limits_{j\in I} \lambda_{j,k} f_{v,j}(e^{i\cdot \frac{2a\pi}{n}})=e^{i\cdot \frac{2a\pi}{n}} f_{v,k}(e^{i\cdot \frac{2a\pi}{n}})$ for any $k\in I$. These equations are also true if we replace $e^{i\cdot \frac{2a\pi}{n}}$ by $e^{i\cdot \frac{2b\pi}{n}}$ since they are equations of $\BQ$-polynomials. Then $v'=\sum\limits_{j\in I} f_{v,j}(e^{i\cdot \frac{2b\pi}{n}}) \alpha_j$ is in $(V_{\BC})_{w}^{\frac{2b\pi}{n}}$.
    
    Meanwhile, We have $(\alpha,v')=\sum\limits_{j\in I} f_{v,j}(e^{i\cdot \frac{2b\pi}{n}}) (\alpha, \alpha_j)$ and $g_v:=\sum\limits_{i\in I} (\alpha, \alpha_j)f_{v,j}$ is again a rational polynomial in $\BQ[x]$. Since $g_v(e^{i\cdot \frac{2a\pi}{n}})=(\alpha,v)=0$, then $(\alpha,v')=g_v(e^{i\cdot \frac{2b\pi}{n}})=0$. Therefore, we have $(\alpha, V_w^{\frac{2b\pi}{n}})=0$.    
\end{proof}

Based on above, we may determine the irredundant subsequence of a given $\underline{\Theta}$.

\begin{lemma}\label{5.3}
    Suppose that $b=b(w)\in B^+(W)$ is a good position braid representative of some $\CC\in [W]$. Then any $\theta_{i_j}$ in the irredundant subsequence $r(\underline{\Theta})$ satisfies $\frac{2\pi}{\theta_{i_j}}\in \BN$.
\end{lemma}

\begin{proof}
    Suppose that $\frac{2m\pi}{n}$ is in the sequence $\underline{\Theta}$ with $m,n$ coprime and $1<m<n$. Then $w$ has an eigenvalue $e^{i\cdot \frac{2m\pi}{n}}$, which is a primitive $n$-th root of unit. From the proof lemma \ref{5.2}, We see every primitive $n$-th root of unit is also an eigenvalue of $w$. In particular, we have $\frac{2\pi}{n}$ is in $\underline{\Theta}$ and is ahead of $\frac{2m\pi}{n}$ since $\underline{\Theta}$ is increasing. By lemma \ref{5.2}, we have $\frac{2m\pi}{n}$ cannot lie in the irredundant sequence since $W_{V_w^{\frac{2m\pi}{n}}}=W_{V_w^{\frac{2\pi}{n}}}$. Therefore, every $\theta$ in $r(\underline{\Theta})$ is of the form $\frac{2\pi}{n}$ for some $n\in \BN$.
\end{proof}

Now we generalize the result of Yun. Let $d$ be the order of $w\in\CC$ and $\omega$ be a primitive $d$-th root of unity. Recall by \cite[\S 4.3]{GKM}, there is a loop torus of type $\CC$ whose corresponding loop Lie algebra $L\fkt_{w}$ can be written as follows.
    \begin{equation*}
        L\fkt_{w}=\sum\limits_{k\in \BZ_{\geq 0}} \varpi^{\frac{k}{d}} \fkt_{\underline{k}}.
    \end{equation*}
    Here $\fkt_{\underline{k}}$ is the $\omega^{\underline{k}} $-eigenspace of $w$ in $\fkt$ where ${\underline{k}}:=k \text{ mod } d$. Recall $\fkt$ is identified with the $\BC$-span of coroots. Given any $n \vert d$, we have $\alpha(\fkt_{\frac{d}{n}})=0 $ is equivalent to $(\alpha,\sum\limits_{\substack{1\leq m\leq n \\(m,n)=1}} V_w^{\frac{2m\pi}{n}})=0$. By above lemmas, this is further equivalent to $(\alpha, V_w^{\frac{2\pi}{n}})=0$.

\begin{proposition}\label{5.4}
    Let $\CC$ be any conjugacy class in $W$. We have
    \begin{equation*}
        \delta_{\CC}=\frac{l_{\good}(\CC)-(r -r_{\CC})}{2},
\end{equation*}   
    where $r$ (resp. $r_{\CC}$) is the semisimple rank of $G$ (resp. maximal split subtorus of $T_\gamma$ for any $\gamma$ of type $\CC$).
\end{proposition}
\begin{proof}
    It suffices to show $\val (\Delta(\gamma))=l_{\good}(\CC)$ for any shallow element $\gamma\in (L^{\heart}\fkg)_{\CC}$. Let $b=b(w)$ be a good position braid representative of $\CC$. By symmetry of affine Springer fibers, one may assume $\gamma\in (L^{\heart} \fkt_w)^{sh}$. 
    
    By lemma \ref{5.2} and \ref{5.3}, we have $r(\underline{\Theta})=(\frac{2\pi}{n_1} ,\frac{2\pi}{n_2} ,\ldots,\frac{2\pi}{n_l} )$ where $n_1>\cdots >n_l$ are all integers. Recall by lemma \ref{3.3} we have
    \begin{equation*}
        l_{\good}(\CC)=2\sum\limits_{j=1}^l \frac{1}{n_j} \lvert \fkH_{F_{i_{j-1}}}-\fkH_{F_{i_j}} \rvert .
    \end{equation*}
    Therefore, it suffices to show that $\val (\alpha(\gamma))= \frac{1}{n_j}$, for any positive root $\alpha$ with $H_{\alpha}\in  \fkH_{F_{i_{j-1}}}-\fkH_{F_{i_j}}$. 

    Let $\alpha$ be a positive root with $H_{\alpha}\in  \fkH_{F_{i_{j-1}}}-\fkH_{F_{i_{j}}}$ for some $1\leq j\leq l$. Then we have $(\alpha, V_w^{\frac{2\pi}{n_j}})\neq 0$. By \cite[Lemma 2.2]{Y} and our discussion after lemma \ref{5.3}, we have $\val(\alpha(\gamma))\leq \frac{1}{n_j}$ since $\gamma$ is a shallow element. Moreover, since $(\alpha,V_w^{\theta})=0$ for any $\theta\in \Gamma_w$ smaller than $\frac{2\pi}{n_j}$, we have $\val(\alpha(\gamma))= \frac{1}{n_j}$ and our statement is proved.
\end{proof}

\subsection{Dimension of affine Springer fibers}
In this section, we concern about the dimensions of affine Springer fibers of any given type $\CC\in [W]$ and prove our last main result.

\begin{proposition}\label{c}
        Let $\CC$ be any conjugacy class in $W$ and $\gamma\in (L^{\heart}\fkg)_{\CC}$. There exists an indecomposably-good position braid representative $b$ such that $b$ projects to some element in $\CC$ under the natural projection and
        \begin{equation*}
        \dim \Gr_{\gamma}=\dim \Fl_{\gamma}=\frac{l(b)-(r -r_{\CC})}{2},
\end{equation*}   
    where $r$ (resp. $r_{\CC}$) is the semisimple rank of $G$ (resp. maximal split subtorus of $T_\gamma$).
\end{proposition}
\begin{proof}
    Similar to the proof of proposition \ref{5.4}, it suffices to find an indecomposably-good position braid representative $b=b(w,\CV,\underline{\Theta}_{\CV})$ such that $l(b)=\val(\Delta(\gamma))$. We determine the sequence $\CV$ and $\underline{\Theta}_{\CV}$ by investigating all $\alpha(\gamma)$. Then the existence of $w$ and $b$ follows directly from our discussion in \S 3.5.

    Let $d$ be the order of any element in $\CC$. Without loss of generality, fix $\omega=e^{i\frac{2\pi}{d}}$ as a primitive $d$-th root of unity. Recall that one can write $\gamma=\sum\limits_{n\in \BZ_{\geq 0}} \varpi^{\frac{n}{d}}\gamma_n$, where $\gamma_n\in \fkt_{\underline{n}}$. Since $\gamma$ is topologically nilpotent, we have $\gamma_0=0$. For any $n\geq 0$, define 
    \begin{equation*}
        R_n=\{\alpha\in R\mid \alpha(\gamma_j)\neq 0 \text{ for some } j\leq n\}.
    \end{equation*}
    For any root $\alpha$, we then have $\val (\alpha(\gamma))=\frac{n_{\alpha}}{d}$ where $n_{\alpha}=\min \{n\mid \alpha\in R_n\}$. 
    
    Let $(n_1,n_2,\ldots,n_m)$ be an increasing subsequence consisting of all $n$ such that $R_n\setminus R_{n-1}$ is nonempty. This sequence is finite since there are only finitely many roots. By definition, we have $\gamma_{n_1}\neq 0$ and $w(\gamma_{n_1})=\omega^{n_1}\gamma_{n_1}$. By similar argument as in lemma \ref{5.2}, we have
    \begin{center}
        $\gamma_{n_1}=\sum\limits_{j\in I}f_{n_1,j}(e^{i\frac{2n_1\pi}{d}})\alpha_j^\vee$,
    \end{center}
    where $f_{n_1,j}\in \BQ[x]$. Define $\gamma_{n_1}'=\sum\limits_{j\in I}f_{n_1,j}(e^{i\frac{-2n_1\pi}{d}})\alpha_j^\vee$. Again by the proof of lemma \ref{5.2}, we have $w(\gamma_{n_1}')=\omega^{-n_1}\gamma_{n_1}'$. Therefore, we have
    \begin{equation*}
        w(\gamma_{n_1}+\gamma_{n_1}')+w^{-1}(\gamma_{n_1}+\gamma_{n_1}')=2\cos \frac{2n_1\pi}{d}(\gamma_{n_1}+\gamma_{n_1}').
    \end{equation*}
    Meanwhile, since $f_{n_1,j}$ are all rational polynomials, we have $f_{n_1,j}(e^{i\frac{2n_1\pi}{d}})+f_{n_1,j}(e^{i\frac{-2n_1\pi}{d}})$ is a real number for all $j$. Thus $\gamma_{n_1}+\gamma_{n_1}'\in \fkt_{\BR}$.
    
    Now we consider $v_1:=\sum\limits_{j\in I}(f_{n_1,j}(e^{i\frac{2n_1\pi}{d}})+f_{n_1,j}(e^{i\frac{-2n_1\pi}{d}}))\alpha_j\in V\otimes \BC$. By above we have $v_1\in V_w^{\frac{2n_1\pi}{d}}$. Let $V_1=\BR v_1+\BR w(v_1)\subset V$. This gives us an indecomposable real eigenspace of $w$ in $V_w^{\frac{2n_1\pi}{d}}\subset V$ as $w^2(v_1)+v_1\in \BR w(v_1)$. Set $\theta_1=\frac{2n_1\pi}{d}$. Similarly, for each $\gamma_{n_k}$, we can construct an indecomposable real eigenspace $V_k$ of $w$ in $V_w^{\frac{2n_k\pi}{d}}\subset V$ and set $\theta_k=\frac{2n_k\pi}{d}$. Therefore, we have construct the sequences $\CV'=(V_1,\ldots,V_m)$ and $\underline{\Theta}_{\CV'}=(\theta_1,\ldots,\theta_m)$. By definition we know $\underline{\Theta}_{\CV'}$ is increasing. As for admissibility, by lemma \ref{3.11}, one can always extend $(\CV',\underline{\Theta}_{\CV'})$ into $(\CV,\underline{\Theta}_{\CV})$ where $\CV=(V_1,\ldots, V_m,V_{m+1},\ldots), \underline{\Theta}_{\CV}=(\theta_1,\ldots,\theta_m,\theta_{m+1},\ldots)$ such that $\CV$ is admissible (actually even complete) and $\underline{\Theta}_{\CV}$ is weakly-increasing. 

    Next we determine the irredundant subsequence $r(\CV)$ of $\CV$. As in \S2 and \S3, we set $F_j=\sum\limits_{k=1}^j V_k$ and $R_{F_j}=\{\alpha\in R \mid H_{\alpha}\supset F_j\}$ for $1\leq j\leq m$, where $H_{\alpha}$ is the root hyperplane corresponding to $\alpha$. Set $R_{F_0}=R$. Let $\alpha$ be any root, by definition we have $\alpha\in R_{n_k}\setminus R_{n_k-1}=R_{n_k}\setminus R_{n_{k-1}}$ for some $k$. By our definition of $\gamma_{n_k}$ and $\gamma_{n_k}'$, we know $V_k\otimes \BC =\BC v_k+\BC w(v_k)=\BC (\sum\limits_{j\in I}f_{n_k,j}(e^{i\frac{2n_k\pi}{d}})\alpha_j) +\BC (\sum\limits_{j\in I}f_{n_k,j}(e^{-i\frac{2n_k\pi}{d}})\alpha_j)$. Therefore, we have $(\alpha,V_k)=0$ is equivalent to $\alpha(\gamma_{n_k})=\alpha(\gamma_{n_k}')=0$. This tells that $R_{F_{k-1}}\setminus R_{F_k}=R_{n_k}\setminus R_{n_{k-1}}$. Moreover, we have the irredundant subsequence $r(\CV)$ is exactly $\CV'$.
    
    Finally, by corollary \ref{3.9}, we have
    \begin{eqnarray*}
        l(b)&=&\sum\limits_{k=1}^m \frac{\theta_k}{\pi}\lvert \fkH_{F_{k-1}}-\fkH_{F_k}\rvert 
        = \sum\limits_{k=1}^m \frac{n_1}{d}\lvert R_{F_{k-1}}-R_{F_k}\rvert \\
        &=& \sum\limits_{k=1}^m \frac{n_1}{d}\lvert R_{n_{k}}-R_{n_{k-1}}\rvert=\sum\limits_{\alpha\in R} \val (\alpha(\gamma))=\val(
        \Delta(\gamma)),
    \end{eqnarray*}
    and thus the theorem is proved. 
    \end{proof}

In a recent work of Trinh \cite{Tr}, he related conjugacy classes in the braid groups to affine Springer fibers in a different approach. It is interesting to see if there are any hidden math of this connection.

\end{document}